\documentclass[a4paper,10pt]{article}
\usepackage[english]{babel}
\usepackage{amsthm,amsmath,amssymb,bbm}
\usepackage{stackengine}
\usepackage[round]{natbib}

\usepackage{tikz}
\usetikzlibrary{positioning, calc}

\newtheorem{thm}{Theorem}[section]

\newtheorem{claim}[thm]{Claim}
\newtheorem{lemma}[thm]{Lemma}
\newtheorem{ex}[thm]{Example}
\newtheorem{prop}[thm]{Proposition}
\newtheorem{defin}[thm]{Definition}
\newtheorem{rema}[thm]{Remark}
\newtheorem{corollary}[thm]{Corollary}
\usepackage{xcolor}

\def \n{\Vert}

\def\E{{\mathbb{E}}}

\def\P{{\mathbb{P}}}
\def\R{{\mathbb{R}}}
\def\N{{\mathbb{N}}}

\def\F{{\cal{F}}}

\def\limt{\lim_{t\to\infty}}
\def\limt0{\lim_{t\to 0}}

\def\|{\,|\,}

\def\bn#1\en{\begin{align*}#1\end{align*}}
\def\bnn#1\enn{\begin{align}#1\end{align}}

\allowdisplaybreaks

\title{Functional worst risk minimization}

\author{Philip Kennerberg${}^*$ and Ernst C. Wit\footnote{philip.kennerberg@usi.ch, wite@usi.ch}} 
\date{\today}

\begin{document}
	\maketitle

    \begin{abstract}
        Statistical learning methods typically assume that the training and test data originate from the same distribution, enabling effective risk minimization. However, real-world applications frequently involve distributional shifts, leading to poor model generalization. To address this, recent advances in causal inference and robust learning have introduced strategies such as invariant causal prediction and anchor regression. While these approaches have been explored for traditional structural equation models (SEMs), their extension to functional systems remains limited. This paper develops a risk minimization framework for functional SEMs using linear, potentially unbounded operators. We introduce a functional worst-risk minimization approach, ensuring robust predictive performance across shifted environments. Our key contribution is a novel worst-risk decomposition theorem, which expresses the maximum out-of-sample risk in terms of observed environments. We establish conditions for the existence and uniqueness of the worst-risk minimizer and provide consistent estimation procedures. Empirical results on functional systems illustrate the advantages of our method in mitigating distributional shifts. These findings contribute to the growing literature on robust functional regression and causal learning, offering practical guarantees for out-of-sample generalization in dynamic environments.
    \end{abstract}
	\section{Introduction}
	
	Traditionally statistical learning methods operate under the key assumption that the distribution of data used during model estimation will match the distribution encountered at test time. This assumption enables effective minimization of future expected risk, allowing models to generalize well within a fixed distribution. Techniques such as cross-validation \citep{laan2006cross, van2007super}, and information criteria like AIC \citep{akaike1974new} or Mallow's Cp \citep{gilmour1996interpretation}, rely on this principle. However, in real-world applications, data often experience shifts due to evolving conditions, unseen scenarios, or changing environments. Such shifts may arise from sampling biases, where the training data reflects only a subpopulation of the target population, or temporal separations punctuated by external shocks, a scenario common in fields like economics and finance \citep{kremer2018risk}. When distribution shifts occur, traditional prediction methods can fail, resulting in degraded performance. Addressing and mitigating the impact of these shifts is essential for developing robust models that maintain their reliability and predictive power even in dynamic contexts.
	
	In recent years, the field of statistics has made significant strides in tackling distributional shifts through innovative approaches to risk minimization. \cite{icp} introduced a method for identifying causal relationships that remain valid across diverse environments, called \emph{invariant causal prediction}. By focusing on robust and invariant causal links, this approach ensures the generalizability of predictions under changing conditions. Similarly, \cite{causaldantzig} proposed the \emph{Causal Dantzig}, which leverages inner product invariance under additive interventions to estimate causal effects while accounting for confounders. Expanding on these ideas, \cite{arjovsky2019invariant} developed \emph{invariant risk minimization}, which seeks to identify data representations yielding consistent classifiers across multiple environments. However, as shown by \cite{rosenfeld2020risks} and \cite{kamath2021does}, these strict causal methods often struggle when the training and test data differ significantly, undermining their effectiveness in addressing the very problem they aim to solve.
	
	To overcome these challenges, alternative strategies have been proposed. \cite{Rot} introduced \emph{anchor regression}, a technique that mitigates confounding bias by incorporating anchor variables—covariates strongly linked to both treatment and outcome—in the regression model. This method improves the accuracy of causal effect estimation. Building on this concept, \cite{kania2022causal} proposed \emph{causal regularization}, which does not require explicit auxiliary variable information and offers strong out-of-sample risk guarantees. 
	Together, these methods represent a growing arsenal of tools for addressing distributional shifts and achieving robust statistical learning.
	
	Thus far, much of the research addressing distributional shifts and causal inference has centered around traditional structural equation models (SEMs). These models emphasize explicit relationships between variables through linear or nonlinear equations, offering a framework to study causal pathways and account for confounders under various assumptions. While SEMs have been instrumental in advancing the understanding of causality and robustness, they are often limited in their ability to capture complex, high-dimensional functional dependencies that arise in real-world data. Some research has been done for functional SEMs, but so far these SEMs have been defined in ``score space", i.e., the structural relationship is given in terms of how the scores of the target and covariates relate to each other rather than having a direct relationship between the target and covariate processes \citep{muller2008functional}. This strategy makes proving results easier but it comes at the cost of being more cumbersome and it is also much harder to interpret such models. Another strategy that has been employed is to work in the reproducing Hilbert space (RKHS) framework \citep{lee2022functional}. This again gives many tools for proving results, but it also excludes the canonical $L^2$ setting for functional-regression and the SEMs are based on  Hilbert-Schmidt, i.e., compact operators. In this manuscript, we overcome these restrictions and instead assume SEMs that are expressed directly in terms of the target and covariate processes, based on linear (not necessarily bounded) operators. We provide out-of-sample guarantees, which provide for robustness in this setting. Importantly, the out-of-sample shift space is defined as the most natural analogue to the non-functional setting, again avoiding a definition in ``score space". 
    
	
The aim of this paper is to extend worst risk minimization, also called worst average loss minimization, to the functional realm. This means finding a functional regression representation that will be robust to future distribution shifts on the basis of data from two environments. In the classical non-functional realm, linear structural equations are based on a transfer matrix $B$. In section~\ref{sec:sfr}, we generalize this to consider a linear operator $\mathcal{T}$ on square integrable processes that plays the the part of $B$. By requiring that $(I-\mathcal{T})^{-1}$ is bounded --- as opposed to $\mathcal{T}$ --- this will allow for a large class of unbounded operators to be considered. In Section~\ref{sec:worstrisk} we provide the central result of this paper, the functional worst-risk decomposition. This result considers the worst risk among all the shifted environments corresponding to shifts in the out-of-sample shift set. This shift set is defined, given $\gamma>0$, as all shifts in some pre-defined subset $\mathcal{A}$ of square-integrable processes such that $\sqrt{\gamma}A\in\overline{\mathcal{A}}$ (the closure of $\mathcal{A}$), where $A$ is the shift corresponding to the observed in-sample shifted environment. The worst risk is then expressed as a linear combination of the risk for the observational (or reference) environment and the risk of the observed shifted environment. Interestingly, the actual decomposition of the worst risk has the same structure as in the non-functional case. In order to provide such a short formulation of the model and this decomposition, the proof of this result is quite massive. 
In section~\ref{sec:minimizer}.1, we prove a necessary and sufficient condition for existence of a unique minimizer to this worst risk in the space of square integrable kernels. As a special case of this we get a multivariate and robust generalization of the ``basic theorem for functional linear models" \cite[Theorem 2.3]{He2010}. In section~\ref{sec:minimizer}.2 we provide sufficient conditions for finding a, not necessarily unique, minimizer in any arbitrary ON-basis. This removes the necessity of estimating eigenfunctions, when these conditions are fulfilled. In section~\ref{sec:estimation}.1 we provide a family of estimators, corresponding to the minimizer of section \ref{sec:minimizer}.1, that are consistent, while in section \ref{sec:estimation}.2 we provide consistent estimators corresponding to section \ref{sec:minimizer}.2.

\section{Functional structural equation models}
	\label{sec:sfr}
	In this section we will define a functional structural system that may be subject to distributional shifts, which are referred to as \emph{environments}. Both the target function $Y$ and the other functions $X$ are defined on some compact interval $[T_1,T_2]$, with left and right endpoints $T_1$ and $T_2$ respectively. We will denote $T=T_2-T_1$.

    \subsection{General definition of functional structural systems}
    Given some probability space $(\Omega,\F,\P)$, the random sources that form the building blocks of the functional structural model are the noise, i.e., the underlying stochasticity of the system, 
    and the shifts. These are random elements in the separable Hilbert space $L^2([T_1,T_2])^{p+1}$ that are $\F-\mathbb{B}\left(L^2([T_1,T_2])^{p+1}\right)$-measurable, where $\mathbb{B}\left(L^2([T_1,T_2])^{p+1}\right)$ denotes the Borel sigma algebra on $L^2([T_1,T_2])^{p+1}$. Here we use the notation $L^2([T_1,T_2])$ for the set of square integrable functions on $[T_1,T_2]$ with respect to Lebesgue measure. In fact, all of our main results in Sections \ref{sec:worstrisk} and \ref{sec:minimizer} hold with respect to any other space $L^2(\mathbb{T},\sigma,\mu)$ where $\mathbb{T}$ is some arbitrary index set, replacing $[T_1,T_2]$, and $\mu$ is a finite measure on the sigma-algebra $\sigma$ such that $L^2(\mathbb{T},\sigma,\mu)$ is separable. In section \ref{sec:estimation} however we make the assumption that our processes are cadlag, so this implies that the index set must be some subset of the real line, but we may however consider other measures besides the Lebesgue measure in this section as well. Let 
	$$\mathcal{V}=\left\{U \textit{ is } \F-\mathbb{B}\left(L^2([T_1,T_2])^{p+1}\right) \mbox{ measurable: }\sum_{i=1}^{p+1}\int_{[T_1,T_2]}\E\left[U_t(i)^2\right]dt<\infty\right\}.$$ 
The expected value of an element in $\mathcal{V}$ exists in the sense of a Bochner integral. 

Let  $\mathcal{T}:D(\mathcal{T})\to L^2([T_1,T_2])^{p+1}$, where $D(\mathcal{T})\subseteq L^2([T_1,T_2])^{p+1}$, be any operator such that
\begin{itemize}
    \item $ \overline{\mathsf{Range}(I-\mathcal{T})}=L^2([T_1,T_2])^{p+1}$; note that this is trivially true if $\mathsf{Range}(I-\mathcal{T})$ is onto $L^2([T_1,T_2])^{p+1}$,
    \item $\mathcal{R}:=\mathsf{Range}(I-\mathcal{T})$ 
is a Lusin space; this is again true if $\mathsf{Range}(I-\mathcal{T})$ is onto $L^2([T_1,T_2])^{p+1}$, but also if $\mathcal{R}$ is a polish space. 
\item $\mathcal{S}:=(I-\mathcal{T})^{-1},$ is bounded and linear --- as well as injective by definition --- on $\mathsf{Range}(I-\mathcal{T})$.
\end{itemize}
As $\mathcal{S}$ is continuous and linear, this implies that for any $L^2([T_1,T_2])^{p+1}$-measurable random element $X$, $\mathcal{S}X$ is also a $L^2([T_1,T_2])^{p+1}$-measurable random element and $\mathcal{S}$ maps elements in $\mathcal{V}$ to elements in $\mathcal{V}$. 
Recall that for a Hilbert space $H$, if we endow the (Cartesian) product space $H^k$ with the inner product 
	$$\langle (x_1,\ldots,x_k),(y_1,\ldots,y_k)  \rangle_{H^k} =  \sum_{i=1}^k \langle x_i,y_i  \rangle_H,$$
	then this also becomes a Hilbert space. This is the inner product we will use on all product Hilbert spaces going forward. In particular we have that if $A,B\in \mathcal{V}$ then the inner product $\langle A,B  \rangle_{\mathcal{V}} =\sum_{i=1}^{p+1}\int_{[T_1,T_2]}\E\left[A_t(i)B_t(i)\right]dt$ makes $\mathcal{V}$ into a Hilbert space, with norm $\n A \n_{\mathcal{V}}=\sqrt{\sum_{i=1}^{p+1}\int_{[T_1,T_2]}\E\left[A_t(i)^2\right]dt}$. 
 Denote for $x=\left(x_1,\ldots,x_k \right)\in H^k$, by $\pi_i$, projection on the $i$:th coordinate, i.e. $\pi_i x=x_i$.
We shall also define the space
$$\tilde{\mathcal{V}}=\left\{U \in \mathcal{V}: \P\left(U\in \mathsf{Range}(I-\mathcal{T}) \right)=1\right\}, $$
note that since $\mathsf{Range}(I-\mathcal{T})$ is Lusin this implies that it is also $\mathbb{B}\left(L^2([T_1,T_2])^{p+1}\right)$-measurable (see for instance Theorem 5, Chapter 2 part 2 of \cite{schwartz1973radon}) and therefore the event $\{U\in \mathsf{Range}(I-\mathcal{T})\}\in \mathcal{F}$. We also have that $\overline{\tilde{\mathcal{V}}}=\mathcal{V}$, indeed this follows from the fact that the $\mathcal{R}$-simple elements in $\mathcal{V}$ are dense, which in turn follows from the fact that $\overline{\mathcal{R}}=L^2([T_1,T_2])^{p+1}$ and the $L^2([T_1,T_2])^{p+1}$-simple elements in $\mathcal{V}$ are dense. 

	\begin{ex}
		We will now give an example of an unbounded operator $\mathcal{T}$ such that $(I-\mathcal{T})^{-1}$ is injective, bounded, linear and defined on all of $L^2([T_1,T_2])^{p+1}$. Let $[T_1,T_2]=[0,1]$, $Q:D\to L^2([T_1,T_2])$ be the first derivative operator $Qf=f'$, where $D=\left\{f\in H^1(0,1): f(0)=0\right\}$. We then have that $(Q^{-1}f)(x)=\int_0^x f(y)dy$ for any $f\in L^2([T_1,T_2])$ and $\lVert Q^{-1}\rVert_{L^2([T_1,T_2])} \le 1$. Let $B$ be a full rank $(p+1)\times(p+1)$ matrix and set for $\textbf{f}=\left(f_1,\ldots,f_{p+1}\right)\in D^{p+1}$,
		\begin{equation}\label{T}
			\mathcal{T}\textbf{f}
			=
			B\begin{bmatrix}
				f_1'\\
				\vdots \\
				f_{p+1}'
			\end{bmatrix} 
			+
			\begin{bmatrix}
				f_1\\
				\vdots \\
				f_{p+1}
			\end{bmatrix} 	.
		\end{equation}
		Then for any $\textbf{f}\in L^2([T_1,T_2])^{p+1}$
		\begin{equation}\label{ResT}
			\mathcal{S}\textbf{f}
			=
			-B^{-1}\begin{bmatrix}
				Q^{-1}f_1\\
				\vdots \\
				Q^{-1}f_{p+1}
			\end{bmatrix} 
		\end{equation}
		and $\lVert \mathcal{S}\rVert_{L^2([T_1,T_2])^{p+1}} \le \lVert B \rVert_\infty (p+1)$, where $ \lVert B \rVert_\infty$ is the maximum absolute row sum of $B$.
	\end{ex}
	\begin{ex}
		An elementary example in the bounded case is if we take any  $(p+1)\times(p+1)$ matrix $B$ such that $I-B$ is full rank. Then if we let
		$\mathcal{T}\textbf{f}=B\textbf{f}$ for $\textbf{f}\in L^2([T_1,T_2])^{p+1}$ then $\left(I-\mathcal{T}\right)^{-1}\textbf{f}=\left(I-B\right)^{-1}\textbf{f}$ and we have the same bound as in the previous example, $\lVert \mathcal{S}\rVert_{L^2([T_1,T_2])^{p+1}} \le \lVert B \rVert_\infty (p+1)$.
	\end{ex}
	\begin{rema}
The above example is really the functional analogue of the classical linear SEM as it is path-wise linear. Our general model does not assume path-wise linearity, it is only based on linear operators. We illustrate this in the next example.
\end{rema}
\begin{ex}
Let $g_1,\ldots g_{p+1}$ be measurable functions on $[T_1,T_2]$ such that $\left|\frac{1}{1-g_i(t)}\right|$ are all bounded, for $1\le i\le p+1$ and let $\mathcal{T}\left(f_1,\ldots,f_{p+1}\right)=\left(g_1f_1,\ldots,g_{p+1}f_{p+1}\right)$. Then
$$\mathcal{S}f= \left(\frac{1}{1-g_1}f_1,\ldots,\frac{1}{1-g_{p+1}}f_{p+1}\right)$$
and $\lVert \mathcal{S} \rVert\le \max_{1\le i\le p+1} \sup_{t\in [T_1,T_2]}\left|\frac{1}{1-g_i(t)}\right|$.
\end{ex}
Sometimes we may wish to start in the other end, with a bounded, linear map $\mathcal{S}$ as illustrated in the next example, which is related to functional regression. 
\begin{ex}\label{regrEX}
Let $\beta_1,\ldots,\beta_p\in L^2([T_1,T_2]^2)$ and for $f=\left(f_1,\ldots,f_{p+1}\right)\in L^2([T_1,T_2])^{p+1}$ define, for any bounded and linear $S:L^2([T_1,T_2])^p\to L^2([T_1,T_2])$,
$$\mathcal{S}f= \left(S\left(f_2,\ldots,f_{p+1}\right) +f_1,f_2,\ldots,f_{p+1}\right).$$ 
Then $\mathcal{S}$ is clearly bounded and solving $\mathcal{S}f=0$ yields the only solution $f=0$, i.e. $Ker(\mathcal{S}f)=\{0\}$, implying that $\mathcal{S}$ is injective. Of particular interest is when $S\left(f_2,\ldots,f_{p+1}\right)=\int_{[T_1,T_2]}\sum_{i=2}^{p+1}(\beta(.,\tau))(i-1)f_i(\tau) d\tau$, i.e. multivariate classical functional regression with a functional response
\end{ex}

	\begin{figure}[tb]
	\begin{center}
		\begin{tikzpicture}[node distance=3cm, auto]
			\node[draw, circle, minimum size=1cm, align=center] (X1) {$X_t(1)$};
			\node[draw, circle, minimum size=1cm, align=center, right of=X1] (Y) {$Y_t$};
			\node[draw, circle, minimum size=1cm, align=center, right of=Y] (X2) {$X_t(2)$};
			
			\draw[->, thick] (X1) -- (Y) node[midway, above] {$\beta_{x_1y}$};
			\draw[->, thick] (Y) -- (X2) node[midway, above] {$\beta_{yx_2}$};
			
			\node[below=0.5cm of X1] (note1) {};
			\node[below=0.5cm of X2] (note2) {};
		\end{tikzpicture}
	\end{center}
	\caption{Observational environment: a functional system that serves as an illustration of a structural system throughout the manuscript, in which $X(1)$ is the cause of $Y$ and $Y$ is the cause of $X(2)$. Our aim is to minimize the out-of-distribution prediction error of $Y$ using both $X(1)$ and $X(2)$. \label{fig:structural}}
	\end{figure}

\subsection{Functional system observed in two environments}
We now fix some noise $\epsilon\in\tilde{\mathcal{V}}$ such that  $\E\left[\epsilon\right]=0$, which is to be interpreted as the zero function in $L^2([T_1,T_2])^{p+1}$, or more formally the corresponding equivalence class. 
For any $A\in \mathcal{V}$ we may extend $(\Omega,\F,\P)$ to $(\Omega_A,\F_A,\P_A)$, to contain a copy of $\epsilon$, which we denote $\epsilon_A$, such that $\epsilon_A$ and $A$ are independent as $\F-\mathbb{B}\left(L^2([T_1,T_2])^{p+1}\right)$-measurable random elements.
If $A\in\tilde{\mathcal{V}}$ then on $(\Omega_A,\F_A,\P_A)$ we may consider equation systems, which we refer to as environments (just as in the non-functional setting), of the following form
	\begin{align}\label{SEMA1}
		(Y^A,X^A)=\mathcal{T}(Y^A,X^A) +A+\epsilon_A
	\end{align}
	and denote the special case of the observational environment
	$(Y^O,X^O)=\mathcal{T}(Y^O,X^O) +\epsilon_O.$
    	\begin{rema}
	Note that this implies that we assume that the target and the covariates are centralized in the observational environment. In the empirical setting, if the observed samples are not centralized we must first centralize them (more on this in Section \ref{sec:estimation}). 
	\end{rema}
	This implies that we have the unique solutions
	\begin{align}\label{SEM}
	(Y^A,X^A)=\mathcal{S}\left(A+\epsilon_A\right),
	\end{align}
implying that $Y^A,X^A$ are both $L^2([T_1,T_2])^{p+1}$-measurable. Since $\mathcal{S}$ is linear and bounded on $\mathsf{Range}(I-\mathcal{T})$ we may extend it to a linear and bounded operator on all of $L^2([T_1,T_2])^{p+1}$ and therefore we may then define $(Y^A,X^A)$ for any $\epsilon, A\in\mathcal{V}\setminus \tilde{\mathcal{V}}$ directly through \eqref{SEM}.  Throughout the rest of this paper we will assume that we work with the extended version of $\mathcal{S}$. In principle one may also start with \eqref{SEM} in order to define $(Y^A,X^A)$, for any shift $A$, without requiring that $\mathcal{S}$ is injective, although in that case, \eqref{SEMA1} will not be fulfilled. 

	\subsection{Example: a functional structural system}
	\label{sec:illus1}

        	\begin{figure}[tb]
		\begin{center}
			\begin{tikzpicture}[node distance=3cm, auto]
			\node[draw, circle, minimum size=1cm, align=center] (X1) {$X_t(1)$};
			\node[draw, circle, minimum size=1cm, align=center, right of=X1] (Y) {$Y_t$};
			\node[draw, circle, minimum size=1cm, align=center, right of=Y] (X2) {$X_t(2)$};
			
			\node[draw, circle, minimum size=1cm, align=center, above of=X1] (A1) {$A_t(1)$};
			\node[draw, circle, minimum size=1cm, align=center, above of=X2] (A2) {$A_t(2)$};
			
			\draw[->, thick] (X1) -- (Y) node[midway, above] {$\beta_{x_1y}$};
			\draw[->, thick] (Y) -- (X2) node[midway, above] {$\beta_{yx_2}$};
			\draw[->, thick] (A1) -- (X1) node[midway, right] {};
			\draw[->, thick] (A2) -- (X2) node[midway, right] {};			
			\end{tikzpicture}
		\end{center}
		\caption{Interventional environment: the structural functional system is also observed under a slightly intervened conditions. In particular, the scores $\xi_1$ of $X_t(1)$ and $\xi_2$ of $X_t(2)$ are affected by shifts $A_1$ and $A_2$, respectively. \label{fig:shift}}
	\end{figure}

	We consider a structural functional system with three functional variables,  $X_t(1),X_t(2)$ and $Y_t$, on the interval $[0,1]$. The variables are structurally related as shown in Figure~\ref{fig:structural}. In particular, we have that 
	\[ 			
    \left\{    \begin{array}{rcl}
    Y_t &=&  \int_{[0,1]} \beta_{x_1y}(t,\tau) X_\tau(1)~d\tau + \epsilon_{t}(1) \\
		X_t(1) &=& \epsilon_{t}(2) \\
		X_t(2) &=&  \int_{[0,1]} \beta_{yx_2}(t,\tau) Y(\tau)~d\tau + \epsilon_{t}(3)		 
	\end{array}
	\right.\]
	In particular, we have that the structural effect $\beta_{x_2y}$ of $X(2)$ on $Y$ is zero. The effect $\beta_{x_1y}$ of $X(1)$ on $Y$ can be defined relative to an ON-basis. We consider the ON-basis given by
	\[  \mathcal{B} = \{ \phi_k(t)=\sqrt{2} \sin(2k\pi t)~|~k=1,\ldots, n=10; t\in[0,1] \}.   \] 
	We assume that $(Y,X(1),X(2))$ is a random function system that can be expressed in terms of $\mathcal{B}$, such that for $\phi=(\phi_1,\ldots,\phi_{10})$
	\[ Y = \zeta^t\phi,~~ X(1) =  \xi_1^t \phi, ~~ X(2)=\xi_2^t\phi, \] 
	whereby the random scores $(\zeta,\xi_1,\xi_2)\in \mathbb{R}^{30}$ in the observational environment $O$ are related according to
	\begin{equation} \begin{bmatrix}
		\zeta^O \\ \xi_1^O\\ \xi_2^O
	\end{bmatrix}  =
	B \begin{bmatrix}
		\zeta^O \\ \xi_1^O\\ \xi_2^O
	\end{bmatrix} + \epsilon^O, \label{eq:obsexample} \end{equation}
	where $\epsilon^O\sim N(0,\Sigma)$ with $\Sigma=I_{30\times 30}$. In our case, we assume homogeneous effects across all the basis functions and choose,
	\[    B = \begin{bmatrix}
		0 & b_{x_1y} & 0 \\
		0 & 0 & 0 \\
		b_{yx_2} & 0 & 0
	\end{bmatrix} \bigotimes I_{10 \times 10}, \]
	where the $3\times 3$ matrix describes the structural relatedness of $X(1)$ to $Y$ and of $Y$ to $X(2)$. We choose $b_{x_1y}=b_{yx_2}=1$. As \cite{He2010} showed, the functional coefficients $\beta$ are intrinsically related to structural matrix $B$ for the scores, besides the first and second moments of the scores, which in our case are all zeroes and ones. In fact, the causal functional coefficient $\beta_{x_1y}$ can be shown to be given as  
	\[ \beta_{x_1y}(t,\tau) =  b_{x_1y} \sum_{k=1}^{10}\phi_k(t)\phi_k(\tau),  \]
	whereas the causal functional coefficient for $X(2)$ is the zero functional $\beta_{x_2y}=0$. 
	Furthermore, we consider the structural functional system under interventions $A_1=\alpha_1^t\phi$ and $A_2=\alpha_2^t\phi$, that affect $X(1)$ and $X(2)$ through their scores $\xi_1$ and $\xi_2$. In particular, given a minor score shift $\alpha_{jk}\sim N(\mu^A_j,\Sigma^A_{jj})$ with $\mu^A_j=\sqrt{\Sigma^A_{jj}} =0.1$ for each covariate function $j=1,2$ and each basis dimension $k=1,\ldots,10$, the scores in the intervened environment are given by
		\begin{equation} \begin{bmatrix}
		\zeta^A \\ \xi_1^A\\ \xi_2^A
	\end{bmatrix}  =
	B \begin{bmatrix}
		\zeta^A \\ \xi_1^A\\ \xi_2^A
	\end{bmatrix} 
	+ \begin{bmatrix}
		0 \\ \alpha_1\\ \alpha_2
	\end{bmatrix} + \epsilon^A, \label{eq:shiftexample} \end{equation}
	where $\epsilon^A \stackrel{D}{=} \epsilon^O$. Figure~\ref{fig:yxasample} shows a sample of the system $(Y^A, X^A(1),X^A(2))$ from the shifted environment.

	\begin{figure}[tb!]
		\centering
		\includegraphics[width=0.8\textwidth]{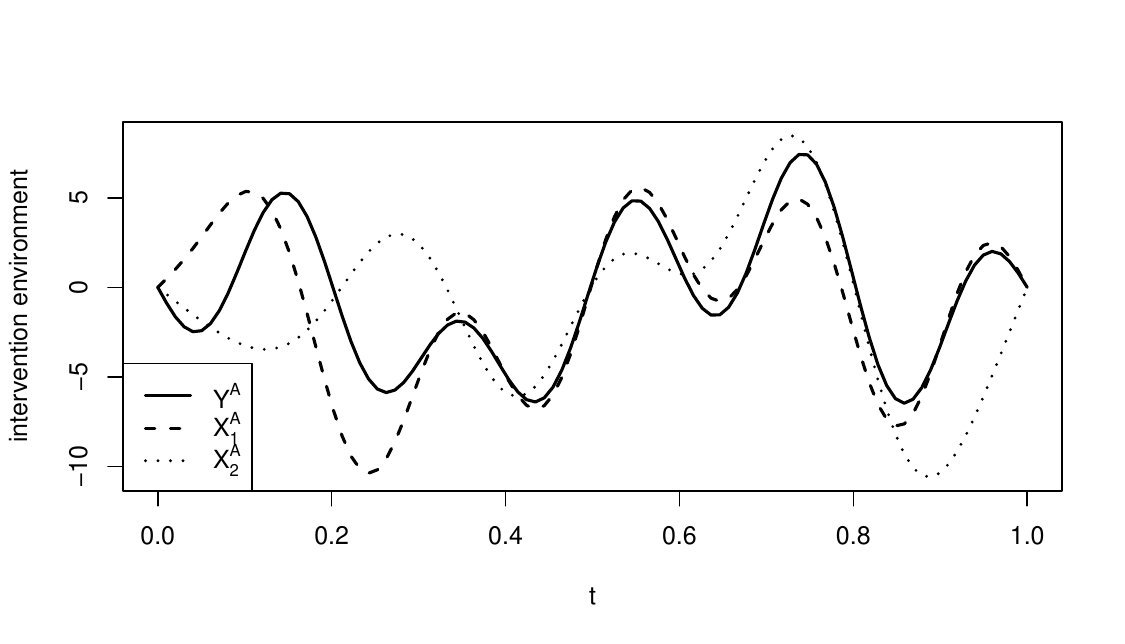} 
		\caption{Sample $(Y^A,X^A(1),X^A(2))$ from the shifted environment. Note that both $X(1)$ and $X(2)$ seem quite predictive for $Y$, but only $X(1)$ is causal --- and therefore $X(1)$ has the most robust out-of-sample risk behaviour, if $Y$ is not intervened, as in this example.}
		\label{fig:yxasample}
	\end{figure}

	\section{Functional worst risk decomposition}
	\label{sec:worstrisk}
For $A\in\mathcal{V}$, define the joint shift covariance function matrix,
	\begin{equation}\label{Kernel}
		K_A(s,t)=\begin{bmatrix}
			K_{A(1),A(1)}(s,t) & K_{A(1),A(2)}(s,t) & \hdots & K_{A(1),A(p+1)}(s,t)\\
			\vdots & \ddots\\
			K_{A(1),A(p+1)}(s,t) & K_{A(2),A(p+1)}(s,t) & \hdots & K_{A(p+1),A(p+1)}(s,t)
		\end{bmatrix} 	,
	\end{equation}
where $K_{A(i),A(j)}(s,t)=\E\left[A_s(i)A_t(j)\right]$, which is a well-defined element in $L^2([T_1,T_2]^2)$ (since $A\in\mathcal{V}$ and the Cauchy-Schwarz inequality).
	For $A\in \mathcal{V}$ and $\beta\in(L^2([T_1,T_2]^2))^p$ we define the risk-function associated with $A$,
	\begin{align}
		R_A(\beta)=\E_A\left[\int_{[T_1,T_2]}\left(Y^{A}_t-\int_{[T_1,T_2]}\sum_{i=1}^p(\beta(i))(t,\tau)X^A_\tau(i)d\tau\right)^2dt\right]
	\end{align}
	In order for the risk concept to be meaningful in the functional context we want some level of regularity in terms of the shifts. To be more precise, we want small perturbations in the shifts to cause small perturbations in the risk for the corresponding environments. This is exactly the content of the following Lemma (which is used to prove Theorem \ref{WR1}).
	\begin{lemma}\label{PropEnvCont}
		If $A'_n\xrightarrow{\mathcal{V}}A'$ then $R_{A'_n}(\beta)\to R_{A'}(\beta)$ for any $\beta\in (L^2([T_1,T_2]^2))^p$.
	\end{lemma}
We now define our shift-space as an analogue to the non-functional case. Indeed, in the non-functional setting the shift set of level $\gamma>0$ is defined as
$ \left\{A'\in \mathcal{A}: \E\left[ A'^TA'\right] \preceq \gamma \E\left[ A^TA\right]\right\}, $
	for some $\mathcal{A}$ is some set of random vectors that contain $\sqrt{\gamma}A$ and whose components are square integrable. $\E\left[ A'^TA'\right] \preceq \gamma \E\left[ A^TA\right]$ denotes $\gamma \E\left[ A^TA\right]- \E\left[ A'^TA'\right]$ being positive semi-definite. A concept that is often touted as the continuous analogue of this is the so-called Mercer condition, i.e., if $A'$ and $A$ are square integrable (univariate) processes then the corresponding criteria would be 
	$$\int_{[T_1,T_2]^2} g(s) (\gamma K_A(s,t)-K_{A'}(s,t)) g(t)dsdt\le 0 $$
	for every $g\in L^2([T_1,T_2])$ . Since we have a multivariate functional model, the natural generalization is therefore the following definition. 
\begin{defin}\label{ShiftDef} {\bf Collection of out-of-sample environments.}
For $A\in\mathcal{V}$, $\mathcal{A}\subseteq\mathcal{V}$ and $\gamma\in\R^+$ let the collection of future out-of-sample environments be defined as, 
		\begin{align}\label{ShiftSet}
			C^\gamma_{\mathcal{A}}(A)=&\left\{A'\in \mathcal{A}: \int_{[T_1,T_2]^2} g(s)K_{A'}(s,t)g(t)^Tdsdt
\le \gamma\int_{[T_1,T_2]^2} g(s)K_{A}(s,t)g(t)^Tdsdt,\forall g\in L^2([T_1,T_2])^{p+1}
			\right\}.
		\end{align}
	\end{defin}
	The above definition is in fact equivalent to the following one, which seemingly imposes a weaker condition on the shifts. Let $\mathcal{G}\subseteq L^2([T_1,T_2])$ be such that $\overline{\mathcal{G}}=L^2([T_1,T_2])$, where the closure is with respect to the subspace topology. Then
	\begin{prop}\label{Gprop}
		\begin{align*}
			C^\gamma_{\mathcal{A}}(A)=&\left\{A'\in \mathcal{A}: \int_{[T_1,T_2]^2} \left(g_1(s),....,g_{p+1}(s)\right)K_{A'}(s,t)\left(g_1(t),....,g_{p+1}(t)\right)^Tdsdt
			\right.\nonumber
			\\
			&\left.\le \gamma\int_{[T_1,T_2]^2} \left(g_1(s),....,g_{p+1}(s)\right)K_{A}(s,t)\left(g_1(t),....,g_{p+1}(t)\right)^Tdsdt,\forall g_1,....,g_{p+1}\in \mathcal{G}
			\right\}.
		\end{align*}
	\end{prop}
The above proposition implies that we can restrict our attention to say, for instance, polynomial functions or step functions on $[T_1,T_2]$ when verifying if a shift belongs to $C^\gamma_{\mathcal{A}}(A)$. 
	\\
	A consequence of the following proposition is that for $A'\in \mathcal{V}$, a sufficient condition for $A'\in C^\gamma_{\mathcal{A}}(A)$ is that for any ON-basis of $L^2([T_1,T_2])$, $\{\phi_n\}_{n\in\N}$ ,the partial sums of the form 
	$$S_n(A')=\left(\sum_{k=1}^n \alpha_k(1)\phi_k,\ldots,\sum_{k=1}^n \alpha_k(p+1)\phi_k \right),$$
	where $\alpha_k(i)=\int_{[T_1,T_2]}A'_t(i)\phi_k(t)dt$, all lie in $C^\gamma(A)$ for sufficiently large $n$. Or, we may instead take an approximating sequence $A^n$ of $A'$ (so that $A^n\xrightarrow{\mathcal{V}}A'$), consisting of simple step function processes described as follows.
	Consider the set of $\left(L^2([T_1,T_2])\right)^{p+1}$-simple processes, i.e. processes of the form $\sum_{i=1}^n b_i1_{B_i}$, where $B_i\in\mathcal{F}$ and $b_i\in\left(L^2([T_1,T_2])\right)^{p+1}$, which are dense in $\mathcal{V}$ (see for instance Lemma 1.4 in \cite{bosq}). In fact, since the step functions are dense in $\left(L^2([T_1,T_2])\right)^{p+1}$ it means that the set of processes of the form $\sum_{i=1}^n b_i1_{B_i}$, where now instead the $b_i$'s a are step functions, are also dense in $\mathcal{V}$.
	\begin{prop}\label{closed}
		$C^\gamma(A)$ is closed in $\mathcal{V}$ if $\mathcal{A}$ is closed in $\mathcal{V}$.
	\end{prop}
We now provide some examples where we give an explicit characterization of the shift set.
	\begin{ex}
		If $\{\phi_1,\ldots,\phi_n\}$ are orthonormal and $\mathcal{A}=\mathsf{span}\{\phi_1,\ldots,\phi_n\}$, $A(i)=\sum_{k=1}^na_{i,k}\phi_k$, where each $a_{i,k}\in L^2(\P)$, then $C^\gamma_{\mathcal{A}}(A)$ will consist of the set of shifts of the form $A'(i)=\sum_{k=1}^na'_{i,k}\phi_k$, $1\le i\le p+1$, with $a'_{i,k}\in L^2(\P)$ when $\E\left[\textbf{a'}^T\textbf{a'}\right]\preceq \gamma\E\left[\textbf{a}^T\textbf{a}\right]$, where $\textbf{a}=\left(a_{1,1},\ldots,a_{1,n},a_{2,1},\ldots,a_{p+1,n} \right) $ and $\textbf{a'}=\left(a'_{1,1},\ldots,a'_{1,n},a'_{2,1},\ldots,a'_{p+1,n} \right) $. 
	\end{ex}
	
	\begin{prop}\label{WSS}
		Suppose $\mathcal{A}$ is some set of wide-sense stationary processes in $\mathcal{V}$ and let $A\in \mathcal{V}$ also be wide-sense stationary. Writing, as customary $K_{A'}(s,t)=K_{A'}(s-t,0)=K_{A'}(s-t)$ for any process in $\mathcal{A}$. Then
		$$C^\gamma_{\mathcal{A}}(A)= \left\{A'\in \mathcal{A}: \gamma \hat{K}_{A}(\omega) -\hat{K}_{A'}(\omega) \textit{ is positive semidefinite a.e.}(\omega) \right\}, $$
		where we use the hat symbol to denote the Fourier transform.
	\end{prop}
We now state the main result of the paper, the worst-risk decomposition, which shows that the worst future out-of-sample risk can be written in terms of risks related to the two actually observed environments. Let $R_\Delta(\beta)=R_A(\beta)-R_O(\beta)$ be the risk difference and $R_+(\beta)=R_A(\beta)+R_O(\beta)$ the pooled risk, the following results holds.
	\begin{thm}\label{WR1} {\bf Worst Risk Decomposition.}
		Suppose $A \in \mathcal{V}$, $\sqrt{\gamma}A\in\bar{\mathcal{A}}$ and $\gamma>0$ then
		\begin{align*}
			\sup_{A'\in C^\gamma_{\mathcal{A}}(A)} R_{A'}(\beta)=\frac12 R_+(\beta)+\left(\gamma -\frac12\right)R_\Delta(\beta),
		\end{align*}
		for every $\beta\in (L^2([T_1,T_2]))^p$.
\end{thm}
The theorem states that if the out-of-sample looks like what has been observed previously, then its risk is described by $\gamma \in (0,1]$. The more extreme the out-of-sample environment is, the risk scaled proportionally by a term given by the risk difference. This risk difference also appears in the non-functional Causal Dantzig approach \citep{causaldantzig}. Also, the above decomposition is an exact analogue of the non-functional case \citep{kania2022causal}. We wish to stress the case $\mathcal{A}=\tilde{\mathcal{V}}$ which restricts us to shifts for which the SEMs \eqref{SEMA1}, in case the range of $I-\mathcal{T}$ is only dense (in this case $\mathcal{T}$ cannot be closed) and \eqref{SEMA1} is fulfilled, with the following result.
	\begin{corollary}
Suppose $A \in \mathcal{V}$ and $\gamma>0$ then
		\begin{align*}
			\sup_{A'\in C^\gamma_{\tilde{\mathcal{V}}}(A)} R_{A'}(\beta)=\frac12 R_+(\beta)+\left(\gamma -\frac12\right) R_\Delta(\beta),
		\end{align*}
		for every $\beta\in (L^2([T_1,T_2]))^p$.
	\end{corollary}
An immediate consequence of Theorem \ref{WR1} is also that the worst risk is determined by dense subsets.
	\begin{corollary}
		If $A \in \mathcal{V}$, $\gamma>0$ and $\sqrt{\gamma}A\in \bar{\mathcal{A}}$ then
		$$ \sup_{A'\in C^\gamma_{\mathcal{A}}(A)} R_{A'}(\beta)=\sup_{A'\in C^\gamma_{\bar{\mathcal{A}}}(A)} R_{A'}(\beta).$$
	\end{corollary}
This result is useful when the shift set $C^\gamma_{\mathcal{A}}(A)$ is more easily characterized than the set $C^\gamma_{\bar{\mathcal{A}}}(A)$, as discussed in the paragraph after Proposition \ref{closed}. 

	\section{Functional worst-risk minimization}
	\label{sec:minimizer}

In this section, we derive results in the population setting for worst-risk minimization in functional structural equation models (SEMs). We formulate the worst-risk minimization problem in terms of functional regression and provide conditions for the existence of minimizers. Section~\ref{sec:uniqueness} focuses on establishing necessary and sufficient conditions for the uniqueness of the worst-risk minimizer within the space of square-integrable kernels. This result generalizes classical worst-risk minimization principles to the functional domain, ensuring robustness to distributional shifts.
In contrast, Section~\ref{sec:ON} considers minimizers in an arbitrary orthonormal (ON) basis. This approach removes the need for estimating eigenfunctions explicitly, allowing for a more flexible implementation in practical scenarios.

	\subsection{The uniqueness condition}
    \label{sec:uniqueness}
In this section we provide necessary and sufficient conditions for the existence of a (unique) minimizer $\beta$ in $L^2([T_1,T_2]^2)^p$. The solution is provided in terms of a certain eigenbasis for the covariate process and an arbitrary ON-basis (chosen by the user) corresponding to the target. Fix $\gamma\ge 0$. Consider $\mathcal{K} :L^2([T_1,T_2])^p\to L^2([T_1,T_2])^p$, defined by 
$$(\mathcal{K} f)(t)=\gamma\int_{[T_1,T_2]}K_{X^A}(s,t)f(s)ds+(1-\gamma)\int_{[T_1,T_2]}K_{X^O}(s,t)f(s)ds,$$
for $f\in L^2([T_1,T_2])^p$. This is a compact, self-adjoint operator and we will denote its eigenfunctions by $\{\psi_n\}_{n\in\N}$, which are orthonormal in $L^2([T_1,T_2])^p$. If $\mathsf{Ker}\left(\mathcal{K}\right)\not=\{0\}$ then there exists an ON-basis, $\{\eta_l\}_{l\in\N}$ for $\mathsf{Ker}\left(\mathcal{K}\right)$. Let $\{\phi_n\}_{n\in\N}$ be an arbitrary ON-basis for $L^2([T_1,T_2])$ and
$$W=\left\{\sum_{k=1}^\infty\sum_{l=1}^\infty  \alpha_{k,l} \phi_k\otimes \eta_l: \sum_{k=1}^\infty\sum_{l=1}^\infty  \alpha_{k,l}^2<\infty\right\}. $$
Define
	\begin{itemize}
		\item[] $\chi_k^A=\langle X^{A},\psi_k \rangle_{L^2([T_1,T_2])^p}$
		\item[] $\chi^O_k=\langle X^O,\psi_k \rangle_{L^2([T_1,T_2])^p}$
		\item[] $Z^A_k=\langle Y^A,\phi_k \rangle_{L^2([T_1,T_2])}$
		\item[] $Z^O_k=\langle Y^O,\phi_k \rangle_{L^2([T_1,T_2])}$.
	\end{itemize}
	\begin{thm}\label{MinimizerThm}
Under the conditions of Theorem \ref{WR1}, we have that there is a unique solution 
		\begin{align*}
			&\arg\min_{\beta\in L^2([T_1,T_2]^2)^p}\sup_{A'\in C^\gamma_{\mathcal{A}}(A)} R_{A'}(\beta)
			= \sum_{k=1}^\infty\sum_{l=1}^\infty  \frac{\gamma\E\left[ Z_k^A\chi_l^A\right]+(1-\gamma)\E\left[ Z_k^O\chi_l^O\right] }{\gamma\E\left[ (\chi_l^A)^2\right] +(1-\gamma)\E\left[ (\chi_l^O)^2\right]}\phi_k\otimes \psi_l
		\end{align*}
if and only if
\begin{align}\label{summationkrit}
\sum_{k=1}^\infty\sum_{l=1}^\infty  \frac{\left(\gamma\E\left[ Z_k^A\chi_l^A\right]+(1-\gamma)\E\left[ Z_k^O\chi_l^O\right] \right)^2}{\left(\gamma\E\left[ (\chi_l^A)^2\right] +(1-\gamma)\E\left[ (\chi_l^O)^2\right]\right)^2}<\infty
\end{align}
and the operator $\mathcal{K} $ is injective. If on the other hand, we have \eqref{summationkrit} but $\mathcal{K}_X$ is not injective then 
\begin{align}\label{argminformula}
			&\arg\min_{\beta\in L^2([T_1,T_2]^2)^p}\sup_{A'\in C^\gamma_{\mathcal{A}}(A)} R_{A'}(\beta)
			= \left\{\sum_{k=1}^\infty\sum_{l=1}^\infty  \frac{\gamma\E\left[ Z_k^A\chi_l^A\right]+(1-\gamma)\E\left[ Z_k^O\chi_l^O\right] }{\gamma\E\left[ (\chi_l^A)^2\right] +(1-\gamma)\E\left[ (\chi_l^O)^2\right]}\phi_k\otimes \psi_l\right\}+W.
		\end{align}
	\end{thm}
\begin{ex}
If we apply the above theorem to Example \ref{regrEX} we get a sort of multivariate \textit{and} robust generalization of Theorem 2.3 in \cite{He2010} (the ``Basic theorem for functional linear models"). To be precise we have the SEM,
$$Y^{A'}_t=S\left( X^{A'}(1),\ldots,X^{A'}(p)\right)_t+A'_t(1)+\epsilon_t^{A'}(1),$$
with $X^{A'}_t(i)=A'_t(i+1)+\epsilon_t^{A'}(i+1)$, for $1\le i\le p$ and $A'\in\mathcal{V}$, so that $X^{A'}-A'\sim X^{O}$. In particular with
$$Y^{A'}_t=\int_{[T_1,T_2]}\sum_{i=1}^{p}(\beta(t,\tau))(i)X^{A'}_\tau(i) d\tau+A'_t(1)+\epsilon_t^{A'}(1),$$
we have classical functional regression with a functional response considered over our shift-space. We may then apply Theorem \ref{MinimizerThm} to find our robust minimizer. 
\end{ex}
We now move on to the situation where we consider a minimizer in an arbitrary ON-basis, where the minimizer might not be unique.

	\subsection{Minimizer(s) in an arbitrary ON-basis}
    \label{sec:ON}
Let $V$ either be $L^2([T_1,T_2])$ or a finite dimensional subspace of $L^2([T_1,T_2])$. Keeping in mind that $L^2([T_1,T_2])\otimes L^2([T_1,T_2])=L^2([T_1,T_2]^2)$ we let
$$S=\arg\min_{\beta\in \left(V\otimes V\right)^p}\sup_{A'\in C^\gamma_{\mathcal{A}}(A)} R_{A'}(\beta),$$
i.e., the set of $\arg\min$-solutions of the worst-risk minimization problem for either $L^2([T_1,T_2]^2)$ or a finite dimensional subspace. A-priori we could have that this set is empty, contains a unique element or contains several solutions. Let $N=dim(V)$ and $\{\phi_n\}_{n=1}^N$ be any ON-system that spans $V$. If $N=\infty$ this means that this is a basis for $L^2([T_1,T_2])$, if $N\in\N$ then this system spans a finite dimensional subspace thereof. Denote for $W\in L^2([T_1,T_2])^p$ and $n\in \N$,
\begin{align*}
F_{1:n}(W)=&\left(\int_{[T_1,T_2]}W_t(1)\phi_{1}(t)dt,\ldots,\int_{[T_1,T_2]}W_t(1)\phi_{n}(t)dt,\ldots,\int_{[T_1,T_2]}W_t(p)\phi_{1}(t)dt,\ldots,\int_{[T_1,T_2]}W_t(p)\phi_{n}(t)dt\right).
\end{align*}
Define for $n< N+1$ (rather than $\le N$, to handle the case $N=\infty$),
$$G_n=\gamma \E_{A}\left[F_{1:n}\left(X^{A}\right)^TF_{1:n}\left(X^{A}\right)\right]+(1-\gamma)\E_{O}\left[F_{1:n}\left(X^{O}\right)^TF_{1:n}\left(X^{O}\right)\right], $$
\begin{align*}
\left(\lambda_{1,k,1}(n),\ldots,\lambda_{1,k,n}(n),\ldots,\lambda_{p,k,n}(n) \right)
&=G_n^{-1}\left(\gamma\E_{A}\left[ Z_k^{A}F_{1:n}\left(X^{A}\right)\right]
+
(1-\gamma)\E_{O}\left[ Z_k^{O}F_{1:n}\left(X^{O}\right)\right]
\right)1_{\det\left(G_n\right)\not=0}
\\&+
\left(n,\ldots,n\right) 1_{\det\left(G_n\right)=0},
\end{align*}
for $1\le k\le n$ and
\begin{align*}
\beta_n=\left( \sum_{k=1}^n\sum_{l=1}^n\lambda_{1,k,l}(n)\phi_k\otimes \phi_l,\ldots, \sum_{k=1}^n\sum_{l=1}^n\lambda_{p,k,l}(n)\phi_k\otimes \phi_l\right).
\end{align*}
For notational convenience we also let
$$\lambda(n):= \left(\lambda_{1,1,1}(n),\ldots,\lambda_{p,n,n}(n) \right).$$
To allow for a concise statement of the next result, let us use the convention that if $N<\infty$ then we set $G_m=G_N$, $\lambda(m)=\lambda(N)$ and $\beta_m=\beta_N$, for $m>N$. 
	\begin{thm}\label{Minimizerpg1}
Assume the conditions of Theorem \ref{WR1}. 
\begin{itemize}
\item[1)] If $\{ \lambda(n) \}_{n\in\N}$ contains a subsequence, $\{ \lambda(n_k) \}_{k\in\N}$ that converges in $\mathit{l}^2$ then $S\not=\emptyset$.
\item[2)] If every subsequence $\{ \lambda(n_k) \}_{k\in\N}$ of $\{ \lambda(n) \}_{n\in\N}$ contains a further subsequence $\{ \lambda\left( n_{k_l}\right)  \}_{l\in\N}$ that converges in $\mathit{l}^2$ then
$$\mathsf{dist}\left(\beta_n,S\right)=\inf_{s\in S}\lVert \beta_n -s\rVert_{L^2([T_1,T_2]^2)^p}\to 0.$$
\end{itemize}
\begin{rema}
If $N<\infty$ both the conditions in 1) and 2) above are fulfilled if and only if $G_N$ is full rank (since we are now considering bounded sequences in $\R^N$).
\end{rema}	\end{thm}

\subsection{Illustration of population risk minimization}
\label{sec:illus3}	
Here we illustrate how the procedure works in the population setting. We consider the system described in section~\ref{sec:illus1} in two settings. The first consists of an observational environment and the second of an interventional environment. We consider a range of values for the tuning parameter $\gamma$. The value $\gamma=1/2$ corresponds to the pooled risk setting, where we combine data from both settings to predict $Y$, whereas larger values of $\gamma$ correspond to increased amount of regularization.

With respect to the orthonormal basis $\mathcal{B}$, we can calculate the scores associated with the observational environment $({X}^O,Y^O)$,
\[   \left\{  \begin{array}{rcl}
	\zeta^{O}_k &=& \int_{[0,1]} Y^O(t)\phi_k(t) ~dt \\
	\xi^{O}_{jk} &=& \int_{[0,1]} {X}^O_j(t)\phi_k(t) ~dt, ~~~~j=1,2,
\end{array}   \right.\]
and similarly for the interventional system $({X}^A,Y^A)$. As the scores are observable from the process, so are their moments. 
\begin{eqnarray*} 
	M^O & =&  E ({\xi}^{O},\zeta^{O})({\xi}^{O},\zeta^{O})^t \\
	&=& V(({\xi}^{O},\zeta^{O})^t) + E ({\xi}^{O},\zeta^{O})E({\xi}^{O},\zeta^{O})^t\\
	&=&B \Sigma B^t \\
	M^A &=& E ({\xi}^{A},\zeta^{A})({\xi}^{A},\zeta^{A})^t \\
	&=& V(({\xi}^{A},\zeta^{A})^t) + E ({\xi}^{A},\zeta^{A})E({\xi}^{A},\zeta^{A})^t\\
	&=&B(\Sigma+\Sigma^A)B^t + \mu^A\mu^{At}
\end{eqnarray*}
With these second moments, we can then define the score Grammians and rotated responses for each of the two environments $e\in \{O,A\}$,
\begin{eqnarray*}
	G^e &=& M^e_{\xi\xi} \\
	Z^e &=& M^e_{\xi\zeta},
\end{eqnarray*}
whereby the subscripts indicate the submatrices of the second moment matrix. For each $\gamma\in [1/2,\infty)$ we can now define	the regularized covariance operator for each of the covariates $X(1)$ and $X(2)$, through the two $10\times 10$ submatrices $C_1^\gamma$ and $C_2^\gamma$,
\[ \begin{bmatrix}
	C_1^\gamma \\ C_2^\gamma
\end{bmatrix}
  = \left[\gamma G^A + (1-\gamma)G^O\right]^{-1} \left[\gamma Z^A + (1-\gamma)Z^O\right]. \] 
With these matrices, we can define the solution the worst risk minimizer,
\[   \beta^\gamma_{x_j y}(t,\tau) = \phi^t(t)C^\gamma_j\phi(\tau).\]
In Figures~\ref{fig:beta}b and \ref{fig:beta}c, we present the solutions for two values of the regularization parameter, namely $\gamma=1/2$ and $\gamma=500$. The former corresponds to the solution that minimizes the pooled risk. It is clear that the solution both identifies $X(1)$ and $X(2)$ as predictors of $Y$. This is clearly not a bad assumption, if future data will come environments similar to the ones that we have already seen. However, if we want to be robust to heavily perturbed out-of-sample data scenarios, then Figure~\ref{fig:beta}c show that the near-causal solutions $\beta^{500}_{x_1y}$ and $\beta^{500}_{x_2y}$ offer robust alternatives.   

\begin{figure}[tb]
	\begin{center}
		\begin{tabular}{cccc}
			\includegraphics[width=0.23\textwidth]{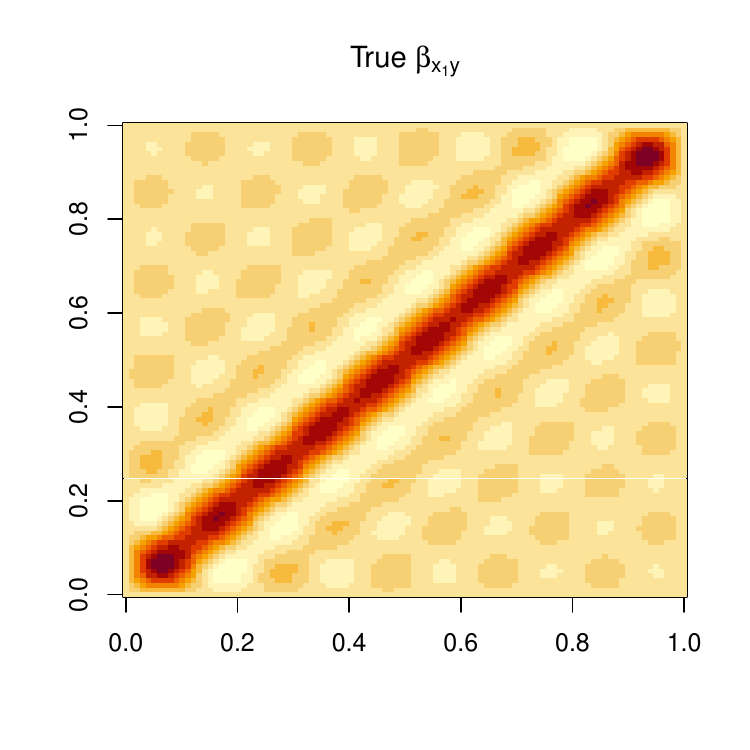}
			& 
			\includegraphics[width=0.23\textwidth]{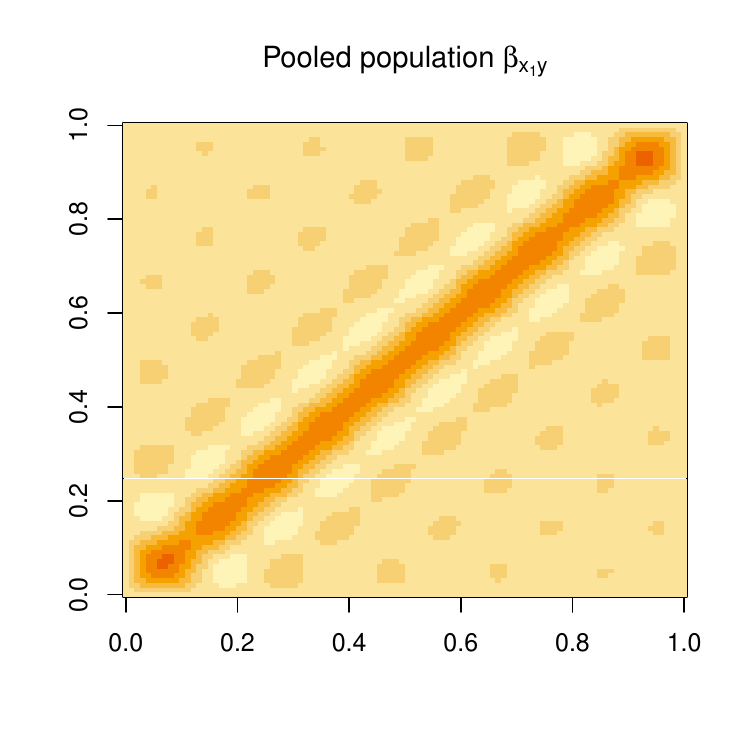}
			& 
			\includegraphics[width=0.23\textwidth]{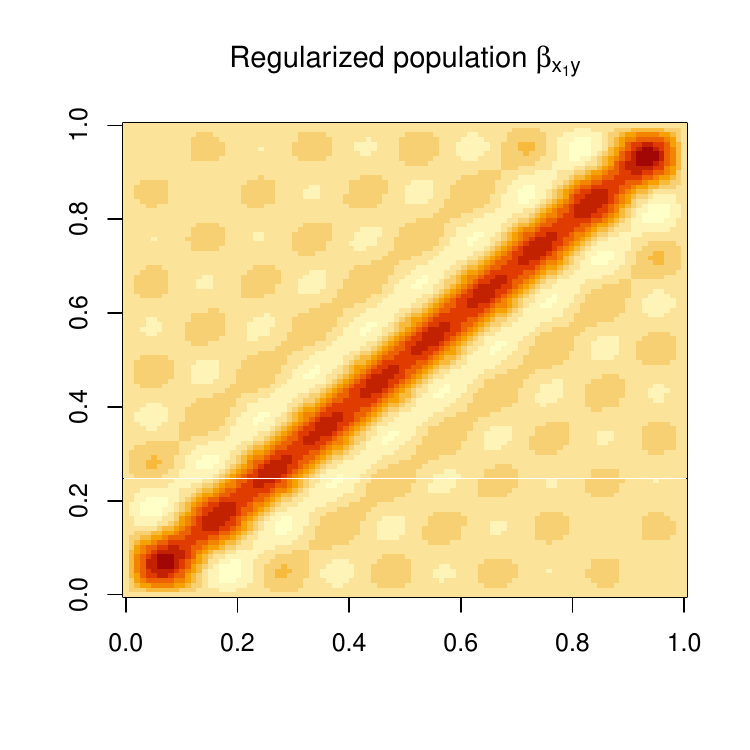}
			& 
			\includegraphics[width=0.23\textwidth]{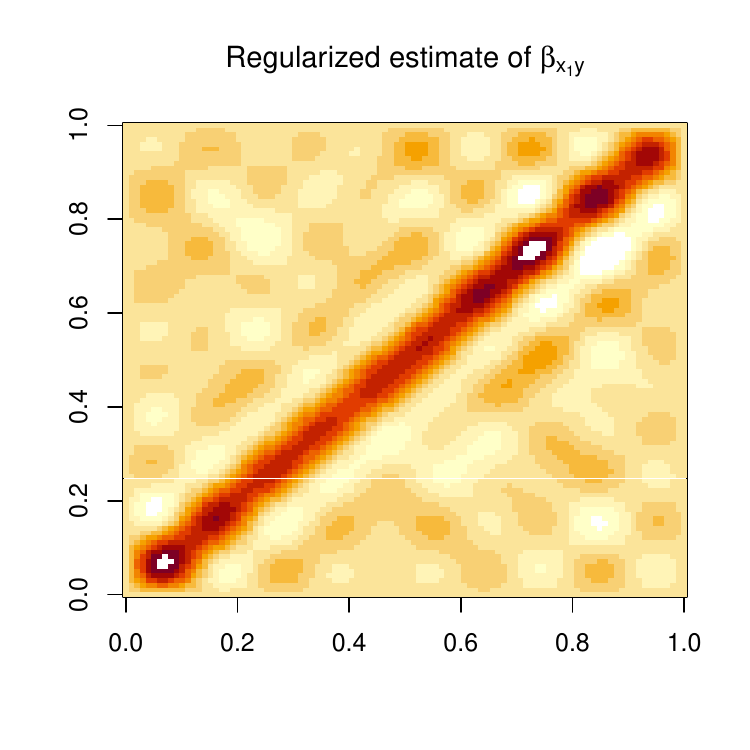}
			\\ 
			\includegraphics[width=0.23\textwidth]{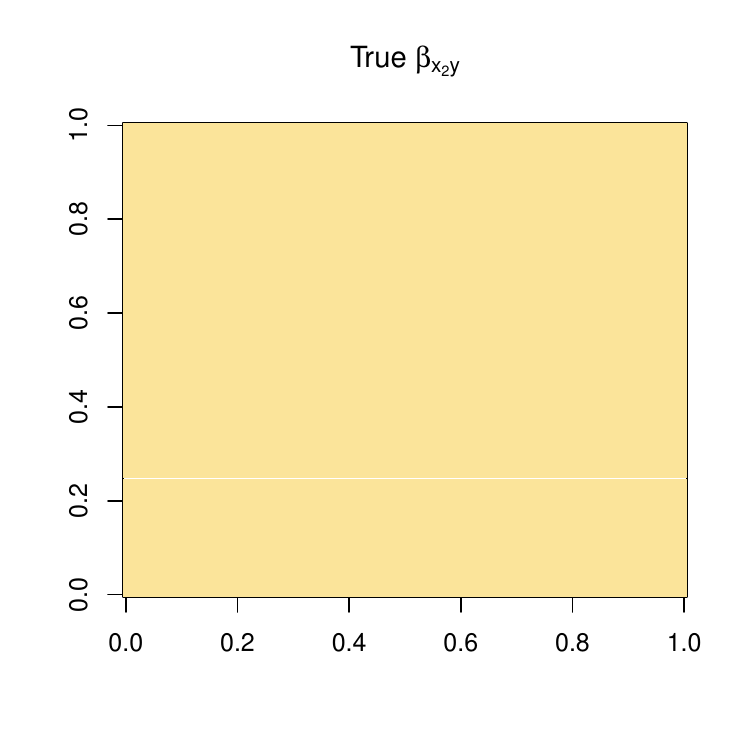}
			& 
			\includegraphics[width=0.23\textwidth]{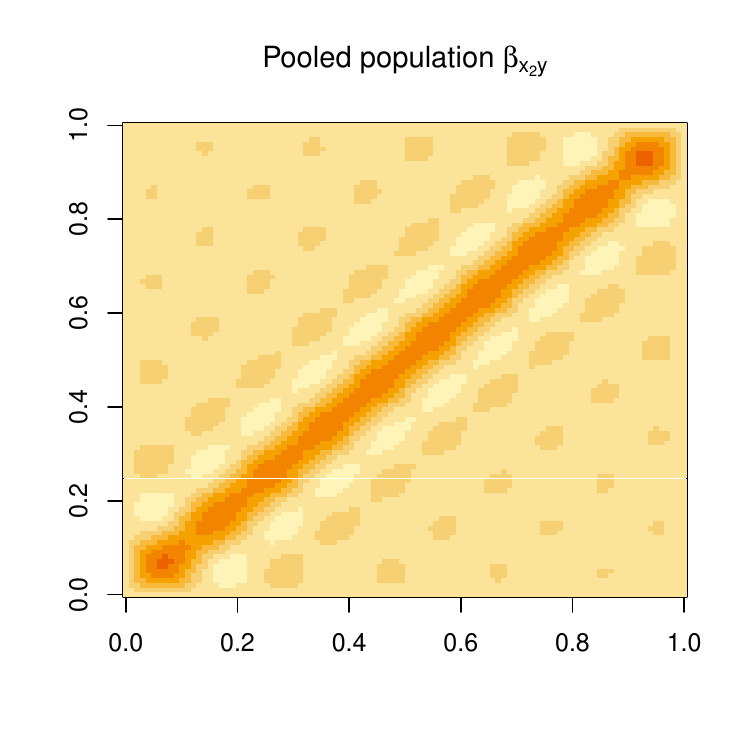}
			& 
			\includegraphics[width=0.23\textwidth]{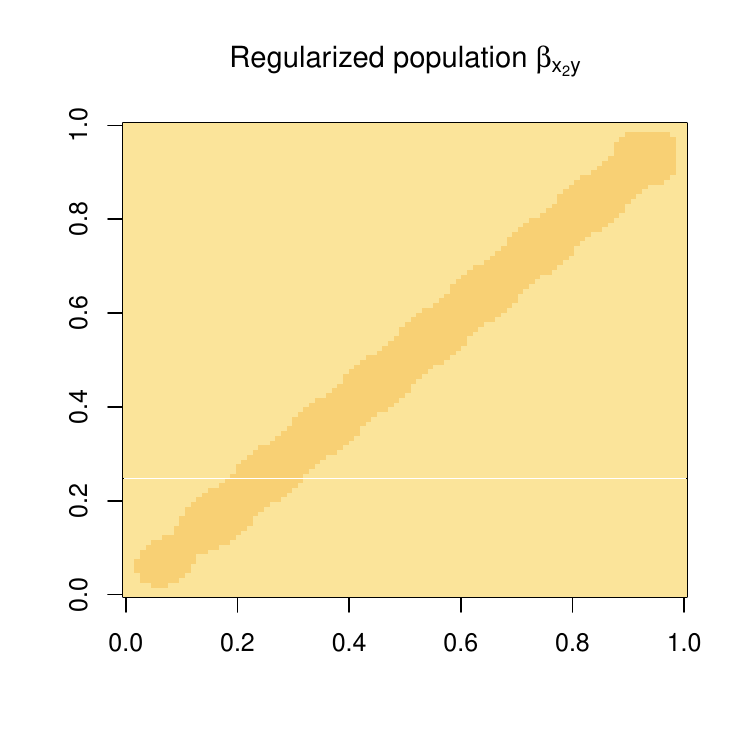}
			& 
			\includegraphics[width=0.23\textwidth]{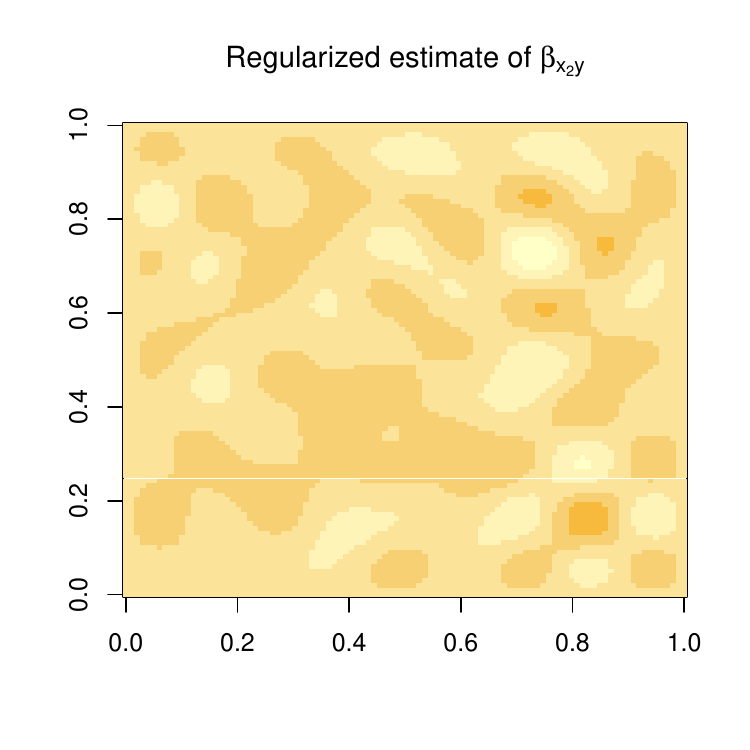}
			\\
			(a) & (b) & (c) & (d)
		\end{tabular}
	\end{center}
	\caption{Functional regression coefficients $\beta_{x_1y}$ and $\beta_{x_2y}$ shown on the same x-y-z scale: (a) True causal parameters; (b) population values, pooling the two data-environments, ``mistakenly'' finds that $X(2)$ affects $Y$; (c) population values minimizing the out-of-sample risk in $\mathcal{C}_{\gamma=500}$ recovering largely the causal parameters; (d) empirical estimates in a data-setting with $n=1000$ samples in an observational and a slightly perturbed environment, using regularization parameter $\gamma=10$. 
	\label{fig:beta}}
\end{figure}

	\section{Estimation of the optimal future worst risk minimizer}
	\label{sec:estimation}

In this section, we derive empirical estimation results for worst-risk minimization in functional structural equation models. We present estimation procedures that allow for robust out-of-sample predictions using observed data from multiple environments.
Sections~\ref{sec:estimation1} and \ref{sec:consistency} focuses on the estimation of worst-risk minimizers under the conditions established in section~\ref{sec:uniqueness}, ensuring consistency in the space of square-integrable kernels. This involves deriving consistent estimators for functional regression models subject to distributional shifts. 
Section~\ref{sec:estimation2} extends these results by providing sufficient conditions for the consistency of estimators when the minimizer is sought in an arbitrary orthonormal basis, as discussed in Section~\ref{sec:ON}. Section~\ref{sec:illus4} illustrates the practical implementation of these estimators with finite data, demonstrating their empirical performance in different environments.

	\subsection{Estimation in the setting of Theorem \ref{MinimizerThm}}
    \label{sec:estimation1}
In this setting we now consider a probability space $\left(\Omega,\F,\P\right)$, which might not coincide with the one from the previous sections. There will be no need to extend this space since we will only observe samples from two fixed environments. On this space we consider i.i.d. sequences $\left\{\left(X^{A,m},X^{O,m},Y^{A,m},Y^{O,m}\right)\right\}_{m\in\N}$, where $\left(X^{A,m},X^{O,m},Y^{A,m},Y^{O,m}\right)$ are distributed according to $\left(X^{A},X^{O},Y^{A},Y^{O}\right)$ (as defined in Section \ref{sec:sfr}) for all $m\in\N$. Whether or not $\left(X^{A,m},X^{O,m},Y^{A,m},Y^{O,m}\right)$ fulfils the same type of SEM as in the population case will not affect our estimator. In order to avoid imposing any extra moment conditions for our estimator we will do separate (independent) estimation for the denominator coefficients. To that end, let 
	$$\mathcal{F}_k=\sigma\left(X^{A,1},X^{O,1},Y^{A,1},Y^{O,1},\ldots,X^{A,k},X^{O,k},Y^{A,k},Y^{O,k}\right)$$ 
and assume $\left\{\left( X'^{A,m}, X'^{O,m}\right)\right\}_{m\in\N}$ is such that $\left( X'^{A,m}, X'^{O,m}\right)$ is independent of $\mathcal{F}_m$ and distributed according to $\left( X^{A}, X^{O}\right)$ (in practice this can always be achieved by splitting the samples from the covariates). We also assume that for every $l\in\N$ we have estimators $\{\hat{\psi}_{l,n}\}_{n\in\N}$ that are independent (again, this can be achieved by splitting) of both $\mathcal{F}_n$ and $\sigma\left(X'^{A,1},X'^{O,1},\ldots,X'^{A,k},X'^{O,k}\right)$, and also fulfils 
$$\lim_{n\to\infty}\P\left(\lVert \hat{\psi}_{l,n}-\psi_l\rVert_{L^2([T_1,T_2])}\ge \epsilon\right)=0,$$ 
for all $\epsilon>0$. For a partition $\Pi=\{t_1,...,t_N\}$ of $[T_1,T_2]$ and a continuous time process $X$ we let $P_\Pi(X,t)=\sum_{i=1}^{N-1}X_{t_i}1_{[t_i,t_{i+1})}(t)$. We also let $\left|\Pi\right|=\max_{1\le i\le N-1}(t_{i+1}-t_i)$. Let $\tau^{\delta}_{m,0}=0$ and for $k\in\N$
\begin{align*}
\tau^{\delta}_{m,k}&=\inf\left\{t>\tau^{\delta}_{m,k-1}:
\right.
\\
&\left. \max\left(\max_{1\le i\le p}\left|X^{A,m}_t(i)-X^{A,m}_{\tau^{\delta}_{m,k-1}}(i)\right|,\max_{1\le i\le p}\left|X^{O,m}_t(i)-X^{O,m}_{\tau^{\delta}_{m,k-1}}(i)\right|, \left|Y^{A,m}_t-Y^{A,m}_{\tau^{\delta}_{m,k-1}}\right|, \left|Y^{O,m}_t-Y^{O,m}_{\tau^{\delta}_{m,k-1}}\right|\right)
	\ge \delta\right\}\wedge T.
\end{align*} 
	We also let $\tau'^{\delta}_{m,0}=0$ and for $k\in\N$
	$$\tau'^{\delta}_{m,k}=\inf\left\{t>\tau^{\delta}_{m,k-1}: \max\left(\max_{1\le i\le p}\left|X'^{A,m}_t(i)-X'^{A,m}_{\tau'^{\delta}_{m,k-1}}(i)\right|,\max_{1\le i\le p} \left|X'^{O,m}_t(i)-X'^{O,m}_{\tau'^{\delta}_{m,k-1}}(i)\right|\right)\ge \delta\right\}\wedge T.$$
	Let $\{d_n\}_{n\in\N}$ be such that $d_n\to 0^+$. Let $\Pi_n^m=\{\tau^{d_n}_{m,k}\}_{k=1}^{J_m} $ and $\Pi_n'^m=\{\tau'^{d_n}_{m,k}\}_{k=1}^{J'_m} $, where $J_m=\inf\{k\in\N:\tau^{d_n}_{m,k}=T\}$ and $J'_m=\inf\{k\in\N:\tau'^{d_n}_{m,k}=T\}$. We finally define 
	\begin{itemize}
		\item[] $\tilde{C}_l^{A,m,n}=\langle P_{\Pi_n}(X^{A,m},.),\hat{\psi}_{l,n} \rangle_{L^2([T_1,T_2])^p},$
		\item[] $\tilde{C}_l^{O,m,n}=\langle P_{\Pi_n}(X^{O,m},.),\hat{\psi}_{l,n} \rangle_{L^2([T_1,T_2])^p},$
		\item[] $\tilde{C}_l'^{A,m,n}=\langle P_{\Pi_n'}(X'^{A,m},.)\hat{\psi}_{l,n} \rangle_{L^2([T_1,T_2])^p},$
		\item[] $\tilde{C}_l'^{O,m,n}=\langle P_{\Pi_n'}(X'^{O,m},.)\hat{\psi}_{l,n} \rangle_{L^2([T_1,T_2])^p},$
		\item[] $D_k^{A,m,n}=\langle P_{\Pi_n}(Y^{A,m},.)\phi_l \rangle_{L^2([T_1,T_2])}$ and
		\item[] $D_k^{O,m,n}=\langle P_{\Pi_n}(Y^{O,m},.)\phi_l \rangle_{L^2([T_1,T_2])}$.
	\end{itemize}
    
	\subsection{Consistency}
    \label{sec:consistency}
Define the truncation operator $T_M:L^2([T_1,T_2])^p\to L^2([T_1,T_2])^p$, 
$$T_M(\psi)=\left( \psi_1 1_{ \lVert \psi_1 \rVert_{L^2([T_1,T_2])}\le M}+ M1_{ \lVert \psi_1 \rVert_{L^2([T_1,T_2])}> M},\ldots,\psi_p 1_{ \lVert \psi_p1 \rVert_{L^2([T_1,T_2])}\le M}+ M1_{ \lVert \psi_p \rVert_{L^2([T_1,T_2])}> M} \right),$$
for $M\in\R^+$.
	\begin{thm}\label{ConsisThm}
		Assume $X^A,X^O$ and $Y^A,Y^O$ are processes in $\mathcal{V}$ that have paths that are a.s. cadlag on $[T_1,T_2]$. Then there exists $E(n)$ such that if $\{e(n)\}_n$ are such that $e(n)\le E(n)$ then all the estimators $\hat{\beta}_n$, of the form
		$$\hat{\beta}_n= \frac{1}{n} \sum_{m=1}^n \sum_{k=1}^{e(n)} \sum_{l=1}^{e(n)} \frac{\gamma\tilde{C}_l^{A,m,n}D_ k^{A,m,n}+(1-\gamma)\tilde{C}_l^{O,m,n}D_ k^{O,m,n}}{\frac{1}{n}\sum_{m=1}^n \gamma (\tilde{C}_l'^{A,m,n})^2+ (1-\gamma)(\tilde{C}_l'^{O,m,n})^2} \phi_k\otimes T_M\left(\hat{\psi}_{l,n}\right)  , $$
for $M>1$ are consistent estimators of the solution in Theorem \ref{MinimizerThm}.
	\end{thm}
	\begin{rema}
		Since $X$ and $Y$ are cadlag it follows from the modulus of continuity property (see for instance Lemma 1 of Chapter 3 in Billingsley) that $\Pi_n$ and $\Pi'_n$ only contains a finite (albeit random) number of points for each path.
	\end{rema}
	\begin{rema}
		If $\E\left[\int_{[T_1,T_2]}(Y^{A}_t)^4dt\right]<\infty$,$\E\left[\int_{[T_1,T_2]}(Y^O_t)^4dt\right]<\infty$ and similarly for the covariates then we do not need to do separate sampling for the denominators.
	\end{rema}
    Recall that we assumed that the target and the covariates are centralized. If the samples are not centralized, we must first centralize them. One can then prove that the following estimator is also consistent --- making a cumbersome proof even more cumbersome,
$$\hat{\beta}_n= \frac{1}{n} \sum_{m=1}^n \sum_{k=1}^{e(n)} \sum_{l=1}^{e(n)} \frac{\gamma\left(\tilde{C}_l^{A,m,n}-\mu_{l,n}^O\right)\left(D_ k^{A,m,n}-\nu_{k,n}^O\right)+(1-\gamma)\left(\tilde{C}_l^{O,m,n}-\mu_{l,n}^O\right)\left(D_ k^{O,m,n}-\nu_{k,n}^O\right)}{\frac{1}{n}\sum_{m=1}^n\left( \gamma (\tilde{C}_l'^{A,m,n}-\mu'^O_{l,n})^2+ (1-\gamma)(\tilde{C}_l'^{O,m,n}-\mu'^O_{l,n})^2\right)} \phi_k\otimes T_M\left(\hat{\psi}_{l,n}\right)  , $$
where $\mu_{l,n}^O=\frac{1}{n} \sum_{m=1}^n \tilde{C}_l^{O,m,n}$, $\nu_{k,n}^O=\frac{1}{n} \sum_{m=1}^n \tilde{D}_k^{O,m,n}$ and $\mu_{l,n}'^O=\frac{1}{n} \sum_{m=1}^n \tilde{C}_l'^{O,m,n}$. One may choose not to centralise the samples from the shifted environment but this changes the assumption on the underlying shift in this case.
	
		\subsection{Estimation in the setting of Theorem \ref{Minimizerpg1}}	
        \label{sec:estimation2}
As in the previous section we again consider a probability space $\left(\Omega,\F,\P\right)$. On this space we consider i.i.d. sequences $\left\{\left(X^{A,m}(1),\ldots,X^{A,m}(p),X^{O,m}(1),\ldots,X^{O,m}(p),Y^{A,m},Y^{O,m}\right)\right\}_{m\in\N}$, where \\$\left(X^{A,m}(1),\ldots,X^{A,m}(p),X^{O,m}(1),\ldots,X^{O,m}(p),Y^{A,m},Y^{O,m}\right)$ are distributed according to \\$\left(X^{A}(1),\ldots,X^{A}(p),X^{O}(1),\ldots,X^{O}(p),Y^{A},Y^{O}\right)$ for all $m\in\N$. In order to avoid imposing any extra moment conditions for our estimator we will do separate (independent) estimation for the denominator coefficients. To that end, let 
	$$\mathcal{F}_k=\sigma\left(X^{A,m}(1),\ldots,X^{A,m}(p),X^{O,m}(1),\ldots,X^{O,m}(p),Y^{A,m},Y^{O,m}:m\le k\right)$$ 
	and assume $\left\{\left( X'^{A,m}, X'^{O,m}\right)\right\}_{m\in\N}$ is such that $\left( X'^{A,m}, X'^{O,m}\right)$ is independent of $\mathcal{F}_m$ and distributed according to $\left( X^{A}, X^{O}\right)$ (in practice this can always be achieved by splitting the samples from the covariates). Let $\tau^{\delta}_{m,0}=0$ and for $k\in\N$
\begin{align*}
\tau^{\delta}_{m,k}&=\inf\left\{t>\tau^{\delta}_{m,k-1}: \left|X^{A,m}_t(i)-X^{A,m}_{\tau^{\delta}_{m,k-1}}(i)\right|\vee\left|X^{O,m}_t(i)-X^{O,m}_{\tau^{\delta}_{m,k-1}}(i)\right|\vee \left|Y^{A,m}_t-Y^{A,m}_{\tau^{\delta}_{m,k-1}}\right|\vee \left|Y^{O,m}_t-Y^{O,m}_{\tau^{\delta}_{m,k-1}}\right|\ge \delta 
\right.
\\
&\left.,1\le i\le p\right\}\wedge T.
\end{align*}
	We also let $\tau'^{\delta}_{m,0}=0$ and for $k\in\N$
	$$\tau'^{\delta}_{m,k}=\inf\left\{t>\tau^{\delta}_{m,k-1}: \left|X'^{A,m}_t-X'^{A,m}_{\tau'^{\delta}_{m,k-1}}\right|\vee \left|X'^{O,m}_t-X'^{O,m}_{\tau'^{\delta}_{m,k-1}}\right|\ge \delta\right\}\wedge T.$$
	Let $\{d_n\}_{n\in\N}$ be such that $d_n\to 0^+$, as in the previous section. Define
	\begin{itemize}
				\item[] $C_l^{A,m,n}(i)=\langle P_{\Pi_n}(X^{A,m}(i),.),\phi_l\rangle_{L^2([T_1,T_2])^p},$
		\item[] $C_l^{O,m,n}(i)=\langle P_{\Pi_n}(X^{O,m}(i),.),\phi_l\rangle_{L^2([T_1,T_2])^p},$
		\item[] $C_l'^{A,m,n}(i)=\langle P_{\Pi_n}(X'^{A,m}(i),.),\phi_l\rangle_{L^2([T_1,T_2])^p},$
		\item[] $C_l'^{O,m,n}(i)=\langle P_{\Pi_n}(X'^{O,m}(i),.),\phi_l\rangle_{L^2([T_1,T_2])^p},$
		\item[] $D_k^{A,m,n}=\langle P_{\Pi_n}(Y^{A,m},.),\phi_k\rangle_{L^2([T_1,T_2])}$,
		\item[] $D_k^{O,m,n}=\langle P_{\Pi_n}(Y^{O,m},.),\phi_k\rangle_{L^2([T_1,T_2])}$,
	\end{itemize}
\begin{align*}
\hat{G}_{n}(M)
&=\gamma\frac{1}{M}\sum_{m=1}^M\left(C_1^{A,m,M}(1),\ldots,C_n^{A,m,M}(p)\right)^T\left(C_1^{A,m,M}(1),\ldots,C_n^{A,m,M}(p)\right)
\\
&+(1-\gamma)\frac{1}{M}\sum_{m=1}^M\left(C_1^{O,m,M}(1),\ldots,C_n^{O,m,M}(p)\right)^T\left(C_1^{O,m,M}(1),\ldots,C_n^{O,m,M}(p)\right), 
\end{align*}
\begin{align*}
\left(\hat{\lambda}_{1,k,1}(n,M),\ldots,\hat{\lambda}_{1,k,n}(n,M),\ldots,\hat{\lambda}_{p,k,n}(n,M) \right)
&=
\hat{G}_{n,2}(M)^{-1}\left( \gamma \frac{1}{M}\sum_{m=1}^M D_k^{A,m,M}\left(C_1^{A,m,M}(1),\ldots,C_n^{A,m,M}(p)\right)
+
\right.
\\
&\left.
\frac{1-\gamma}{M}\sum_{m=1}^MD_k^{O,m,M}\left(C_1^{O,m,M}(1),\ldots,C_n^{O,m,M}(p)\right)
\right)1_{\det\left(\hat{G}_{n,2}(M) \right)\not=0}
,
\end{align*}
for $1\le k\le n$ and
\begin{align*}
\hat{\beta}_{n,M}=\left( \sum_{k=1}^n\sum_{l=1}^n\hat{\lambda}_{1,k,1}(n,M)\phi_k\otimes \phi_l,\ldots, \sum_{k=1}^n\sum_{l=1}^n\hat{\lambda}_{p,k,n}(n,M)\phi_k\otimes \phi_l\right).
\end{align*}
	\begin{thm}\label{ConsisThmpg1}
		Assume $X^A(i),X^O(i)$ and $Y^A,Y^O$, $1\le i\le p+1$ are processes in $\mathcal{V}$ that have paths that are a.s. cadlag on $[T_1,T_2]$. Then there exists $E(n)$ such that if $\{e(n)\}_n$ are such that $e(n)\le E(n)$ then all the estimators of the form $\hat{\beta}_{n,e(n)}$, described above fulfil 
$$\mathsf{dist}\left(\hat{\beta}_{n,e(n)},S\right)\xrightarrow{a.s.}0. $$		
	\end{thm}

\subsection{Illustration of risk minimization with finite data}
\label{sec:illus4}	
	
We observe data from two almost identical environments of the system described in sections~\ref{sec:illus1} and \ref{sec:illus3}. Besides $n=1000$ observations from the observational enviroment, we also observe $n=1000$ observations from a mildly shifted environment. Each observation $i$ in environment $e\in\{O,A\}$ consists of 3 discretely sampled curves $x^e_{1i}(t_l), x^e_{2i}(t_l)$ and $y_i^e(t_l)$ at 100 points $t_l\in [0,1]$ ($l=1,\ldots,100$). From the sampled curves, we estimate the scores relative to the $k$th basis vector $\phi_k$ in $\mathcal{B}$, 
\[   \left\{  \begin{array}{rcl}
	\hat\zeta^{e}_{ki} &=& \sum_{l=1}^{99} (t_{l+1}-t_l) y_i^e(t_l)\phi_k(t_l)  \\
	\hat\xi^{e}_{jki} &=& \sum_{l=1}^{99} (t_{l+1}-t_l) x_{ji}^e(t_l)\phi_k(t_l), ~~~~j=1,2,
\end{array}   \right.\]
Even if the curves are measured with noise, consistent estimates for the scores are readily available \cite{wood2017generalized}. From the estimated scores, we obtain the second moment, plug-in, $30\times 30$ matrices  under the observational and interventional environments,
\[ \widehat{M}^e = \frac{1}{n} \sum_{i=1}^n (\hat{\zeta}^e_{\cdot i},\hat{\xi}^e_{\cdot\cdot i}) (\hat{\zeta}^e_{\cdot i},\hat{\xi}^e_{\cdot\cdot i})^t. \] 
Just as in the population setting, we can then define the empirical score Grammians $\widehat{G}^e$ and the empirical rotated score responses $\widehat{Z}^e$ for each of the two observed environments $e\in\{O,A\}$ as submatrices from the matrix of second moments, i.e., $\widehat{G}^e = \widehat{M}^e_{\xi\xi}$ and $\widehat{Z}^e = \widehat{M}^e_{\xi\zeta}$. For each level of regularization $\gamma\in [1/2,\infty)$ we can now define	the empirical regularized covariance operator for each of the covariates $X(1)$ and $X(2)$, through the two $10\times 10$ submatrices $\widehat{C}_1^\gamma$ and $\widehat{C}_2^\gamma$,
\[ \begin{bmatrix}
	\widehat{C}_1^\gamma \\ \widehat{C}_2^\gamma
\end{bmatrix}
= \left[\gamma \widehat{G}^A + (1-\gamma)\widehat{G}^O\right]^{-1} \left[\gamma \widehat{Z}^A + (1-\gamma)\widehat{Z}^O\right]. \] 
With these matrices, we can define the plug-in estimators of the worst risk minimizer,
\[   \hat\beta^\gamma_{x_j y}(t,\tau) = \phi^t(t)\widehat{C}^\gamma_j\phi(\tau).\]	
Figure~\ref{fig:beta}d shows the estimates for regularization parameter $\gamma=10$. It shows that even in the empirical setting with a minor distribution shift, the method is able to recover, at least approximately, the worst risk minimizer, which in this case corresponds to the causal solution.

\section{Conclusion}

In this paper, we introduced a novel framework for functional structural equation models (SEMs) that extends worst-risk minimization to the functional domain. By leveraging linear, potentially unbounded operators, we provided a formulation that circumvents limitations associated with score-space representations and reproducing kernel Hilbert space (RKHS) methods. Our key theoretical contribution is the functional worst-risk decomposition theorem, which characterizes the maximum out-of-sample risk in terms of observed environments. This decomposition enables robust predictive modeling in the presence of distributional shifts.

We established sufficient conditions for the existence and uniqueness of worst-risk minimizers and proposed consistent estimators for practical implementation. Through empirical illustrations, we demonstrated that our approach effectively mitigates the impact of distributional shifts and improves out-of-sample generalization in functional regression settings. Future work may explore extensions of this framework to nonlinear functional SEMs \citep{fan2015functional}. Additionally, applying our methodology to real-world applications in finance, healthcare, and climate modeling could further validate its practical utility.

	\bibliographystyle{plainnat}
	\bibliography{bibliopaper}  

	\appendix
	
	\section{Proofs}
    Before all of the proofs we remind the reader of the following fact. If $\{\phi_{i,n}\}_{n\in\N}$ are an ON-bases for a Hilbert space $H$, for $1\le i\le k$ then 
	$$ \mathcal{H}^k\left(\{\phi_{1,n}\}_{n\in\N},\ldots,\{\phi_{k,n}\}_{n\in\N}\right)=\left\{(\phi_{1,j_1},0,\ldots,0),(0,\phi_{2,j_2},0,\ldots,0),\ldots,(0,\ldots,0,\phi_{k,j_k})\right\}_{j_1,...,j_k\in\N}, $$
	always forms an ON-basis for $H^k$.
	\subsection{Proof of Proposition \ref{closed}}
	\begin{proof}
		Suppose $\{A_n\}_{n\in\N}\subset \mathcal{A}$ is such that $A_n\xrightarrow{\mathcal{V}}A'$. If $\mathcal{A}$ is closed then $A'\in\mathcal{A}$. By definition of $C^\gamma_{\mathcal{A}}(A)$ we have for any $g_1,\ldots, g_{p+1}\in L^2([T_1,T_2])$.
		\begin{align*}
			&\int_{[T_1,T_2]^2} \left(g_1(s),....,g_{p+1}(s)\right)K_{A_n}(s,t)\left(g_1(t),....,g_{p+1}(t)\right)^Tdsdt
			\\
			&\le \gamma\int_{[T_1,T_2]^2} \left(g_1(s),....,g_{p+1}(s)\right)K_{A}(s,t)\left(g_1(t),....,g_{p+1}(t)\right)^Tdsdt
		\end{align*}
		By estimates analogous to those in \eqref{ApproxKg},
		\begin{align*}
			&\lim_{n\to\infty}\left|\int_{[T_1,T_2]^2} \left(g_1(s),....,g_{p+1}(s)\right)K_{A_n}(s,t)\left(g_1(t),....,g_{p+1}(t)\right)^Tdsdt
			-
			\right.
			\\
			&\left. \int_{[T_1,T_2]^2} \left(g_1(s),....,g_{p+1}(s)\right)K_{A'}(s,t)\left(g_1(t),....,g_{p+1}(t)\right)^Tdsdt\right|
			\\
			&\le
			d\sum_{i=1}^{p+1} \lVert g_{i} \rVert_{L^2([T_1,T_2])}^2\left(\lim_{n\to\infty}\sum_{i=1}^{p+1}\sum_{j=1}^{p+1}\int_{[T_1,T_2]^2} \left| K_{i,j}(s,t)-K^n_{i,j}(s,t)\right|^2dsdt\right)^{\frac12}=0,
		\end{align*}
		for some constant $d$ that depends only on $p$ and where $K_{i,j}$ and $K^n_{i,j}$ are the elements on row $i$ and column $j$ of the matrices $K_{A}$ and $K_{A_n}$ respectively. Therefore
		\begin{align*}
			&\int_{[T_1,T_2]^2} \left(g_1(s),....,g_{p+1}(s)\right)K_{A'}(s,t)\left(g_1(t),....,g_{p+1}(t)\right)^Tdsdt
			\\
			&=
			\lim_{n\to\infty}\int_{[T_1,T_2]^2} \left(g_1(s),....,g_{p+1}(s)\right)K_{A_n}(s,t)\left(g_1(t),....,g_{p+1}(t)\right)^Tdsdt
			\\
			&\le \gamma\int_{[T_1,T_2]^2} \left(g_1(s),....,g_{p+1}(s)\right)K_{A}(s,t)\left(g_1(t),....,g_{p+1}(t)\right)^Tdsdt
		\end{align*}
		which implies $A'\in C^\gamma_{\mathcal{A}}(A)$.
	\end{proof}

	\subsection{Proof of Proposition \ref{Gprop}}
	\begin{proof}
		Let 
		\begin{align*}
			C=&\left\{A'\in \mathcal{A}: \int\int \left(g_1(s),....,g_{p+1}(s)\right)K_{A'}(s,t)\left(g_1(t),....,g_{p+1}(t)\right)^Tdsdt
			\right.\nonumber
			\\
			&\left.\le \gamma\int\int \left(g_1(s),....,g_{p+1}(s)\right)K_{A}(s,t)\left(g_1(t),....,g_{p+1}(t)\right)^Tdsdt,\forall g_1,....,g_{p+1}\in \mathcal{G}
			\right\}.
		\end{align*}
		Since $\mathcal{G}\subseteq L^2([T_1,T_2])$ we trivially have that $C^\gamma_{\mathcal{A}}(A)\subseteq C$. To show $C\subseteq C^\gamma_{\mathcal{A}}(A)$, let $\textbf{g}(s)=\left(g_{1}(s),....,g_{p+1}(s)\right)$, with $g_i\in L^2([T_1,T_2])$ and let $\textbf{g}_n(s)=\left(g_{1,n}(s),....,g_{p+1,n}(s)\right)$, with $g_i\in \mathcal{G}$ be such that $\lVert \textbf{g}_n-\textbf{g} \rVert_{\left(L^2([T_1,T_2])\right)^{p+1}}\to 0$. Taking $A'\in C$ we have (skipping some steps that are in common with \eqref{ApproxKg})
		\begin{align}\label{Aprg}
			&\left|\int_{[T_1,T_2]^2}  \textbf{g}(s)\left(\gamma K_{A}(s,t)-K_{A'}(s,t)\right)\textbf{g}(t)^Tdsdt
			-
			\int\int_{[T_1,T_2]^2}  \textbf{g}_n(s)\left(\gamma K_{A}(s,t)-K_{A'}(s,t)\right)\textbf{g}_n(t)^Tdsdt\right|\nonumber
			\\
			&\le
			\left|\int\int_{[T_1,T_2]^2}   \left(\textbf{g}_n(s)-\textbf{g}(s)\right)\left(\gamma K_{A}(s,t)-K_{A'}(s,t)\right)\textbf{g}(t)^Tdsdt\right|\nonumber
			\\
			&+
			\left|\int\int_{[T_1,T_2]^2}   \textbf{g}(s)\left(\gamma K_{A}(s,t)-K_{A'}(s,t)\right)\left(\textbf{g}_n(t)-\textbf{g}(t)\right)^Tdsdt\right|\nonumber
			\\
			&\le
			2\int\int_{[T_1,T_2]^2} \left(\sum_{i=1}^{p+1}\left(g_{i}(s)-g_{i,n}(s)\right)^2\right)^{\frac12}\left(\sum_{i=1}^{p+1}g_{i}(t)^2\right)^{\frac12} \left\lVert \gamma K_{A}(s,t)-K_{A'}(s,t)\right\rVert_{\mathit{l}^2}dsdt\nonumber
			\\
			&\le
			2T\left(\int_{[T_1,T_2]}\sum_{i=1}^{p+1}\left(g_{i}(s)-g_{i,n}(s)\right)^2ds\int_{[T_1,T_2]} \sum_{i=1}^{p+1}g_{i}(t)^2dt\right)^{\frac12} \left(\int_{[T_1,T_2]^2} \left\lVert K_{A'}(s,t)-K^m_{A'}(s,t)\right\rVert_{\mathit{l}^2}^2dsdt\right)^{\frac12}\nonumber
			\\
			&\le
			dT\lVert \textbf{g}_n-\textbf{g} \rVert_{\left(L^2([T_1,T_2])\right)^{p+1}}\lVert \textbf{g} \rVert_{\left(L^2([T_1,T_2])\right)^{p+1}}
			\left(\sum_{i=1}^{p+1}\sum_{j=1}^{p+1}\int_{[T_1,T_2]^2} \left| K_{i,j}(s,t)-K^m_{i,j}(s,t)\right|^2dsdt\right)^{\frac12},
		\end{align}
		which converges to zero. By assumption
		$$\int\int_{[T_1,T_2]^2}  \textbf{g}_n(s)\left(\gamma K_{A}(s,t)-K_{A'}(s,t)\right)\textbf{g}_n(t)^Tdsdt\ge 0, n\in\N$$
		and therefore
		$$0\le \int\int_{[T_1,T_2]^2}  \textbf{g}(s)\left(\gamma K_{A}(s,t)-K_{A'}(s,t)\right)\textbf{g}(t)^Tdsdt
		= \lim_{n\to\infty}\int\int_{[T_1,T_2]^2}  \textbf{g}_n(s)\left(\gamma K_{A}(s,t)-K_{A'}(s,t)\right)\textbf{g}_n(t)^Tdsdt,$$
		implying $A'\in C^\gamma_{\mathcal{A}}(A)$, i.e. $C\subseteq C^\gamma_{\mathcal{A}}(A)$, as was to be shown.
	\end{proof}
	
	\subsection{Proof of Proposition \ref{WSS}}
	\begin{proof}
		Denote 
		$$ C=\left\{A'\in \mathcal{A}: \gamma \hat{K}_{A}(\omega) -\hat{K}_{A'}(\omega) \textit{ is positive semidefinite a.e.}(\omega) \right\}.$$
		for $1\le i,j\le p+1$ let $K_{i,j}$ be the element on row $i$ and column $j$ of the matrix $\gamma K_{A} -K_{A'}$. By the Plancherel theorem, we have for any $f,g\in L^2([T_1,T_2])$
		\begin{align*}
			\int_{[T_1,T_2]^2} g(s) K_{i,j}(s-t) f(t)dsdt
			&=
			\int_{\R}\int_{\R} g(s) K_{i,j}(s-t) f(t)dsdt
			\\
			&=
			\int_{\R} g(s)(K_{i,j}*f)(s)dsdt
			\\
			&= 
			\frac{1}{2\pi}\int_{\R} \hat{g}(\omega) \widehat{K_{i,j}*f}(\omega)d\omega
			\\
			&=
			\frac{1}{2\pi}\int_{\R} \hat{g}(\omega) \hat{K}_{i,j}(\omega)\hat{f}(\omega)d\omega.
		\end{align*}
		Therefore
		\begin{align*}
			&\int_{[T_1,T_2]^2} \left(g_1(s),\ldots,g_{p+1}(s)\right) \left( \gamma K_{A}(s-t) -K_{A'}(s-t)\right) \left(g_1(t),\ldots,g_{p+1}(t)\right)^Tdsdt
			\\
			=
			&\frac{1}{2\pi}\int_{[T_1,T_2]^2} \left(g_1(s),\ldots,g_{p+1}(s)\right) \left( \gamma K_{A}(s-t) -K_{A'}(s-t)\right) \left(g_1(t),\ldots,g_{p+1}(t)\right)^Tdsdt
			\\
			=
			&\frac{1}{2\pi}\int_{\R}\left(\hat{g}_1(\omega),\ldots,\hat{g}_{p+1}(\omega)\right) \left( \gamma \hat{K}_{A}(\omega) -\hat{K}_{A'}(\omega)\right) \left(\hat{g}_1(\omega),\ldots,\hat{g}_{p+1}(\omega)\right)^Td\omega.
		\end{align*}
		Hence if $A'\in C$ then $A'\in C^\gamma_{\mathcal{A}}(A)$ (i.e. $C\subseteq C^\gamma_{\mathcal{A}}(A)$). In the other direction, suppose that $A'\in C^c$. Denote $\lambda_{p+1}(\omega)$, the smallest eigenvalue of $ \gamma \hat{K}_{A}(\omega) -\hat{K}_{A'}(\omega)$. By assumption there exists a set $D$ of positive Lebesgue measure where $\lambda_{p+1}(\omega)<0$. Take any point $\omega'\in D$ and let $x$ be an eigenvector corresponding to $\lambda_{p+1}(\omega')$. As the entries of $ \gamma \hat{K}_{A}(\omega) -\hat{K}_{A'}(\omega)$ are continuous, it follows that if we let $x$ be an eigenvector corresponding to $\lambda_{p+1}(\omega')$ then  
		$$ x\left(\gamma \hat{K}_{A}(\omega) -\hat{K}_{A'}(\omega)\right)x^T<0,$$
		in some neighbourhood $(\omega'-\epsilon,\omega'+\epsilon)$ for some $\epsilon>0$. For any $\delta>0$, we may now take some mollifier function $\psi(\omega)$ such that $0\le \psi(\omega)\le 1$, $\psi(\omega)=1$ on $[\omega'-\epsilon/2,\omega'+\epsilon/2]$ and $\psi(\omega)=0$ on the complement of $[\omega'-\epsilon/2-\delta,\omega'+\epsilon/2+\delta]$. By choosing $\delta$ sufficiently small we have that 
		$$\int_{\R}\psi(\omega)^2 x \left( \gamma \hat{K}_{A}(\omega) -\hat{K}_{A'}(\omega)\right) x^Td\omega<0. $$
		By the Plancherel theorem
		\begin{align*}
			&\int_{\R}\psi(\omega)^2 x \left( \gamma \hat{K}_{A}(\omega) -\hat{K}_{A'}(\omega)\right) x^Td\omega
			\\
			=2\pi &\int_{[T_1,T_2]^2} \check{\psi}(s) x\left( \gamma K_{A}(s-t) -K_{A'}(s-t)\right)x^T\check{\psi}(t)dsdt,
		\end{align*}
		where $\check{\psi}$ denotes the inverse Fourier transform of $\psi$ (note that $\check{\psi}$ does not have compact support by the uncertainty principle, but since $\gamma K_{A} -K_{A'}$ does it will be restricted to the support of this kernel). Setting $g_i(s)=x(i) \check{\psi}(s) ^{\frac{1}{p+1}}1_{[T_1,T_2]}(s)\in L^2([T_1,T_2])$ yields
		$$ \int_{[T_1,T_2]^2} \left(g_1(s),\ldots,g_{p+1}(s)\right) \left( \gamma K_{A}(s-t) -K_{A'}(s-t)\right) \left(g_1(t),\ldots,g_{p+1}(t)\right)^Tdsdt<0 $$
		and hence $A'\not\in C^\gamma_{\mathcal{A}}(A)$, implying $C^\gamma_{\mathcal{A}}(A)\subseteq C$. This concludes the proof.
	\end{proof}
	
	\subsection{Proof of Theorem \ref{WR1}}
    We begin with the proof of the supporting Lemma
    	\subsection{Proof of Lemma \ref{PropEnvCont}}
		\begin{proof}
		For two fixed $A',A''\in \mathcal{V}$ we have,
		\begin{align}\label{RDelta}
			\left| R_{A'}(\beta)-R_{A''}(\beta)\right|
			&\le
			\left|\E_{A'}\left[\int_{[T_1,T_2]}\left(Y^{A'}_t\right)^2dt\right]
			-
			\E_{A''}\left[\int_{[T_1,T_2]}\left(Y^{A''}_tdt\right)^2\right]\right|\nonumber
			\\
			&+
			2\sum_{i=1}^p\left|\E_{A'}\left[\int_{[T_1,T_2]}Y^{A'}_t \int_{[T_1,T_2]^2}(\beta(t,\tau))(i) X^{A'}_\tau(i) d\tau dt\right]
			\right.\nonumber
			\\
			&\left.-
			\E_{A''}\left[\int_{[T_1,T_2]}Y^{A''}_t \int_{[T_1,T_2]^2}(\beta(t,\tau))(i) X^{A''}_\tau d\tau dt\right]
			\right|\nonumber
			\\
			&+
			\sum_{i=1}^p\left|\E_{A'}\left[\int_{[T_1,T_2]}\left(\int_{[T_1,T_2]}(\beta(t,\tau))(i) X^{A'}_\tau(i) d\tau \right)^2dt\right]\nonumber
			\right.
			\\
			&\left.-
			\E_{A''}\left[\int_{[T_1,T_2]}\left(\int_{[T_1,T_2]}(\beta(t,\tau))(i) X^{A''}_\tau(i) d\tau \right)^2dt\right]
			\right|.
		\end{align}
		Let $\mathcal{S}_i=\pi_i \mathcal{S}$ be the projection of $\mathcal{S}$ on its i:th coordinate, i.e. if $g\in L^2([T_1,T_2])^{p+1}$ then $\mathcal{S}g=(\mathcal{S}_1g,\ldots,\mathcal{S}_{i}g,\ldots,\mathcal{S}_{p+1}g)$. We have that $\mathcal{S}_i:L^2([T_1,T_2])^{p+1}\to L^2([T_1,T_2])$ is linear and bounded, indeed, since
		$$\lVert \mathcal{S} \rVert= \sup_{\lVert f\rVert_{L^2([T_1,T_2])^{p+1}}}\sqrt{\sum_{i=1}^{p+1} \lVert \mathcal{S}_i f\rVert_{L^2([T_1,T_2])}^2}
		\ge \sup_{\lVert f\rVert_{L^2([T_1,T_2])^{p+1}}}\lVert \mathcal{S}_i f\rVert_{L^2([T_1,T_2])}, $$
		for every $1\le i\le p+1$, we have $\lVert \mathcal{S}_i \rVert\le \lVert \mathcal{S} \rVert$. 
For every $n\in\N$, let $\{Q^n_j\}_{j\in\N}$ be a partition of $L^2([T_1,T_2])^{p+1}$ into Borel sets of diameter less than $\frac{1}{2n}$, such a partition must exist because $L^2([T_1,T_2])^{p+1}$ is separable. Set $W_{a''}^n=\sum_{j=1}^\infty w_j^n1_{A''\in Q_j^n}$, $W_{a'}^n=\sum_{j=1}^\infty w_j^n1_{A'\in Q_j^n}$, $W_{e''}^n=\sum_{j=1}^\infty w_j^n1_{\epsilon_{ A''}\in Q_j^n}$ and $W_{e'}^n=\sum_{j=1}^\infty w_j^n1_{\epsilon_{ A'}\in Q_j^n}$ where $w_j^n\in Q_j^n$. Then by construction $\lVert W_{a''}^n-A'' \rVert<\frac{1}{2n}$, $\lVert W_{a'}^n-A' \rVert<\frac{1}{2n}$, $\lVert W_{e'}^n-\epsilon_{ A'} \rVert<\frac{1}{2n}$ and $\lVert W_{e''}^n-\epsilon_{ A''} \rVert<\frac{1}{2n}$, therefore 
			\begin{itemize}
				\item[] $\E\left[\lVert W_{a''}^n\rVert_{L^2([T_1,T_2])^{p+1}}^2\right]\le \E\left[\lVert A''\rVert_{L^2([T_1,T_2])^{p+1}}^2\right]+\frac{1}{4n^2}$,
				\item[] $\E\left[\lVert W_{a'}^n\rVert_{L^2([T_1,T_2])^{p+1}}^2\right]\le \E\left[\lVert A'\rVert_{L^2([T_1,T_2])^{p+1}}^2\right]+\frac{1}{4n^2}$,
				\item[] $\E_{A''}\left[\lVert W_{e''}^n\rVert_{L^2([T_1,T_2])^{p+1}}^2\right]\le \E_{A''}\left[\lVert \epsilon_{ A''}\rVert_{L^2([T_1,T_2])^{p+1}}^2\right]+\frac{1}{4n^2}$ and
				\item[] $\E_{A'}\left[\lVert W_{e'}^n\rVert_{L^2([T_1,T_2])^{p+1}}^2\right]\le \E_{A'}\left[\lVert \epsilon_{ A''}\rVert_{L^2([T_1,T_2])^{p+1}}^2\right]+\frac{1}{4n^2}$.
			\end{itemize}
			Since 
			\begin{itemize}
				\item[] $\E\left[ \lVert W_{a''}^n \rVert_{L^2([T_1,T_2])^{p+1}}^2 \right]=\sum_{j=1}^\infty \lVert w_j^n \rVert_{L^2([T_1,T_2])^{p+1}}^2\P\left( A''\in Q_j^n \right)$,
				\item[] $\E\left[ \lVert W_{a'}^n \rVert_{L^2([T_1,T_2])^{p+1}}^2 \right]=\sum_{j=1}^\infty \lVert w_j^n \rVert_{L^2([T_1,T_2])^{p+1}}^2\P\left( A'\in Q_j^n \right)$,
				\item[] $\E_{A''}\left[ \lVert W_{e''}^n \rVert_{L^2([T_1,T_2])^{p+1}}^2 \right]=\sum_{j=1}^\infty \lVert w_j^n \rVert_{L^2([T_1,T_2])^{p+1}}^2,\P_{A''}\left( \epsilon_{A''}\in Q_j^n \right)$ and
				\item[] $\E_{A'}\left[ \lVert W_{e'}^n \rVert_{L^2([T_1,T_2])^{p+1}}^2 \right]=\sum_{j=1}^\infty \lVert w_j^n \rVert_{L^2([T_1,T_2])^{p+1}}^2,\P_{A'}\left( \epsilon^{A'}\in Q_j^n \right)$ 
			\end{itemize}
			it then follows that there exists $N_n\in\N$ such that if we let $A''^n=\sum_{j=1}^{N_n}w_j^n1_{A''\in Q_j^n}$, $A'^n=\sum_{j=1}^{N_n}w_j^n1_{A'\in Q_j^n}$, $\epsilon_{ A'}^n=\sum_{j=1}^{N_n}w_j^n1_{\epsilon^{A'}\in Q_j^n}$ and $\epsilon_{ A''}^n=\sum_{j=1}^{N_n}w_j^n1_{\epsilon_{A''}\in Q_j^n}$, then 
			\begin{itemize}
				\item[] $\E\left[\lVert A''^n-W_{a''}^n \rVert_{L^2([T_1,T_2])^{p+1}}^2\right] =\sum_{j=N_n+1}^{\infty}\n w_j^n\n^2\P\left(A''\in Q_j^n\right)<\frac{1}{2n}$,
				\item[] $\E\left[\lVert A'^n-W_{a'}^n \rVert_{L^2([T_1,T_2])^{p+1}}^2\right] =\sum_{j=N_n+1}^{\infty}\n w_j^n\n^2\P\left(A'\in Q_j^n\right)<\frac{1}{2n}$,
				\item[] $\E_{A'}\left[\lVert \epsilon^{A',n}-W_{e'}^n \rVert_{L^2([T_1,T_2])^{p+1}}^2\right]=\sum_{j=N_n+1}^{\infty}\n w_j^n\n^2\P\left(\epsilon^{A'}\in Q_j^n\right) <\frac{1}{2n}$ and
				\item[] $\E_{A''}\left[\lVert \epsilon_{A''}^n-W_{e''}^n \rVert_{L^2([T_1,T_2])^{p+1}}^2\right]=\sum_{j=N_n+1}^{\infty}\n w_j^n\n^2\P\left(\epsilon_{A''}\in Q_j^n\right) <\frac{1}{2n}$.
			\end{itemize}
			Next,
			$$\E\left[\lVert A''^n-A'' \rVert_{L^2([T_1,T_2])^{p+1}}^2\right]\le 2\E\left[\lVert A''^n-W_{a''}^n \rVert_{L^2([T_1,T_2])^{p+1}}^2\right]+2\E\left[\lVert A''-W_{a''}^n \rVert_{L^2([T_1,T_2])^{p+1}}^2\right], $$
			which converges to zero and similarly
			$$\E\left[\lVert A'^n-A' \rVert_{L^2([T_1,T_2])^{p+1}}^2\right]\le 2\E\left[\lVert A'^n-W_{a'}^n \rVert_{L^2([T_1,T_2])^{p+1}}^2\right]+2\E\left[\lVert A'-W_{a'}^n \rVert_{L^2([T_1,T_2])^{p+1}}^2\right], $$
			$$\E_{A'}\left[\lVert \epsilon^{A',n}-\epsilon^{A'} \rVert_{L^2([T_1,T_2])^{p+1}}^2\right]\le 2\E_{A'}\left[\lVert \epsilon^{A',n}-W_{e'}^n \rVert_{L^2([T_1,T_2])^{p+1}}^2\right]+2\E_{A'}\left[\lVert \epsilon^{A'}-W_{e'}^n \rVert_{L^2([T_1,T_2])^{p+1}}^2\right] $$ and
			$$\E_{A''}\left[\lVert \epsilon_{A''}^n-\epsilon_{A''} \rVert_{L^2([T_1,T_2])^{p+1}}^2\right]\le 2\E_{A''}\left[\lVert \epsilon_{A''}^n-W_{e''}^n \rVert_{L^2([T_1,T_2])^{p+1}}^2\right]+2\E_{A''}\left[\lVert \epsilon_{A''}-W_{e''}^n \rVert_{L^2([T_1,T_2])^{p+1}}^2\right], $$
			which also converges to zero. 
We now tackle the three terms of the right-hand side of \eqref{RDelta}.
\\
\textbf{Term 1 of \eqref{RDelta}}
\\ 
Applying the Cauchy-Schwarz inequality twice gives us,
\begin{align}\label{integralbound}
\E_{A''}\left[\left|\int_{[T_1,T_2]}\mathcal{S}_1\left(A''-A''^n\right)_t\mathcal{S}_1\left(\epsilon_{A''}^n\right)_tdt\right|\right]
&\le 
\E_{A''}\left[\left(\int_{[T_1,T_2]}\mathcal{S}_1\left(A''-A''^n\right)_t^2dt\right)^\frac12\left(\int_{[T_1,T_2]}\mathcal{S}_1\left(\epsilon_{A''}^n\right)_t^2dt\right)^\frac12\right]\nonumber
\\
&\le 
\E\left[\int_{[T_1,T_2]}\mathcal{S}_1\left(A''-A''^n\right)_t^2dt\right]^\frac12\E_{A''}\left[\int_{[T_1,T_2]}\mathcal{S}_1\left(\epsilon_{A''}^n\right)_t^2dt\right]^\frac12\nonumber
\\
&=
\E\left[\lVert S_1(A''-A''^n) \rVert_{L^2([T_1,T_2])}^2\right]^\frac12\E_{A''}\left[\lVert S_1(\epsilon_{A''}^n) \rVert_{L^2([T_1,T_2])}^2\right]^\frac12\nonumber
\\
&\le
\lVert S \rVert^2\E\left[\lVert A''-A''^n \rVert_{L^2([T_1,T_2])}^2\right]^\frac12\E_{A''}\left[\lVert \epsilon_{A''}^n \rVert_{L^2([T_1,T_2])}^2\right]^\frac12,
\end{align}		
which converges to zero (the fact that $\E_{A''}\left[\lVert \epsilon_{A''}^n - \epsilon_{A''}\rVert_{L^2([T_1,T_2])}^2\right]$ converges to zero implies that $\E_{A''}\left[\lVert \epsilon_{A''}^n \rVert_{L^2([T_1,T_2])}^2\right]$ is bounded).			
Next,
\begin{align}\label{simple}
			\E_{A''}\left[\int_{[T_1,T_2]}\mathcal{S}_1\left(A''^n\right)_t\mathcal{S}_1\left(\epsilon_{A''}^n\right)_tdt\right]
			&=\sum_{j_1=1}^{N_n}\sum_{j_2=1}^{N_n}
			\E_{A''}\left[\int_{[T_1,T_2]}\mathcal{S}_1\left(w_{j_1}^n\right)_t1_{A''\in Q_{j_1}^n}\mathcal{S}_1\left( w_{j_2}^n\right)_t1_{ \epsilon_{A''}\in Q_{j_2}^n}dt\right]\nonumber
			\\
			&=
			\sum_{j_1=1}^{N_n}\sum_{j_2=1}^{N_n}
			\int_{[T_1,T_2]}\mathcal{S}_1\left(w_{j_1}^n\right)_t\mathcal{S}_1\left( w_{j_2}^n\right)_tdt\P_{A''}\left(\left\{A''\in Q_{j_1}^n\right\}\cap \left\{\epsilon_{A''}\in Q_{j_2}^n\right\}\right)
			\nonumber
			\\
			&=
			\sum_{j_1=1}^{N_n}\sum_{j_2=1}^{N_n}
			\int_{[T_1,T_2]}\mathcal{S}_1\left(w_{j_1}^n\right)_t\mathcal{S}_1\left( w_{j_2}^n\right)_tdt\P\left(A''\in Q_{j_1}^n\right)\P_{A''}\left(\epsilon_{A''}\in Q_{j_2}^n\right) 
						\nonumber
			\\
			&=
			\sum_{j_1=1}^{N_n}\sum_{j_2=1}^{N_n}
			\int_{[T_1,T_2]}\mathcal{S}_1\left(w_{j_1}^n\right)_t\mathcal{S}_1\left( w_{j_2}^n\right)_tdt\P\left(A''\in Q_{j_1}^n\right)\P_{A''}\left(\epsilon^{A'}\in Q_{j_2}^n\right) 
									\nonumber
			\\
			&=
						\E_{A'}\left[\int_{[T_1,T_2]}\mathcal{S}_1\left(A''^n\right)_t\mathcal{S}_1\left(\epsilon^{A',n}\right)_tdt\right].
		\end{align}	
An analogous argument as in \eqref{integralbound} shows also that 
$$\lim_{n\to\infty} \E_{A''}\left[\left|\int_{[T_1,T_2]}\mathcal{S}_1\left(A''\right)_t\mathcal{S}_1\left(\epsilon_{A''}^n-\epsilon_{A''}\right)_tdt\right|\right]=0. $$
Obviously,
\begin{align*}
&\left|\E_{A''}\left[\int_{[T_1,T_2]}\mathcal{S}_1\left(A''\right)_t\mathcal{S}_1\left(\epsilon_{A''}\right)_tdt\right] -\E_{A''}\left[\int_{[T_1,T_2]}\mathcal{S}_1\left(A''^n\right)_t\mathcal{S}_1\left(\epsilon_{A''}^n\right)_tdt\right] \right|
\\
&\le 
\E_{A''}\left[\left|\int_{[T_1,T_2]}\mathcal{S}_1\left(A''-A''^n\right)_t\mathcal{S}_1\left(\epsilon_{A''}^n\right)_tdt\right|\right]
+
\E_{A''}\left[\left|\int_{[T_1,T_2]}\mathcal{S}_1\left(A''\right)_t\mathcal{S}_1\left(\epsilon_{A''}^n-\epsilon_{A''}\right)_tdt\right|\right]
\end{align*}
and therefore we have that 
$$ \lim_{n\to\infty}\E_{A''}\left[\int_{[T_1,T_2]}\mathcal{S}_1\left(A''^n\right)_t\mathcal{S}_1\left(\epsilon_{A''}^n\right)_tdt\right]
=
\E_{A''}\left[\int_{[T_1,T_2]}\mathcal{S}_1\left(A''\right)_t\mathcal{S}_1\left(\epsilon_{A''}\right)_tdt\right]$$
and analogously
$$ \lim_{n\to\infty}\E_{A'}\left[\int_{[T_1,T_2]}\mathcal{S}_1\left(A'^n\right)_t\mathcal{S}_1\left(\epsilon^{A',n}\right)_tdt\right]
=
\E_{A'}\left[\int_{[T_1,T_2]}\mathcal{S}_1\left(A'\right)_t\mathcal{S}_1\left(\epsilon^{A'}\right)_tdt\right].$$
So if we now pass to the limit on both sides of \eqref{simple} it follows that 
\begin{align}\label{law1o}
\E_{A''}\left[\int_{[T_1,T_2]}\mathcal{S}_1\left(A''\right)_t\mathcal{S}_1\left(\epsilon_{A''}\right)_tdt\right]=\E_{A'}\left[\int_{[T_1,T_2]}\mathcal{S}_1\left(A''\right)_t\mathcal{S}_1\left(\epsilon^{A'}\right)_tdt\right].
\end{align}
Therefore 
		\begin{align}\label{law1}
		\E_{A'}\left[\int_{[T_1,T_2]}\mathcal{S}_1\left(A'\right)_t\mathcal{S}_1\left(\epsilon^{A'}\right)_tdt\right]
		&=
		\E_{A'}\left[\int_{[T_1,T_2]}\mathcal{S}_1\left(A'-A''\right)_t\mathcal{S}_1\left(\epsilon^{A'}\right)_tdt\right]
		+
		\E_{A'}\left[\int_{[T_1,T_2]}\mathcal{S}_1\left(A''\right)_t\mathcal{S}_1\left(\epsilon^{A'}\right)_tdt\right]\nonumber
		\\
		&=
		\E_{A'}\left[\int_{[T_1,T_2]}\mathcal{S}_1\left(A'-A''\right)_t\mathcal{S}_1\left(\epsilon^{A'}\right)_tdt\right]
		+
		\E_{A''}\left[\int_{[T_1,T_2]}\mathcal{S}_1\left(A''\right)_t\mathcal{S}_1\left(\epsilon_{A''}\right)_tdt\right].
		\end{align}
Similarly to \eqref{simple},
\begin{align}\label{simpleeps}
			\E_{A''}\left[\int_{[T_1,T_2]}\mathcal{S}_1\left(\epsilon^{A'',n}\right)_t^2dt\right]
			&=
			\sum_{j_1=1}^{N_n}\sum_{j_2=1}^{N_n}
			\int_{[T_1,T_2]}\mathcal{S}_1\left(w_{j_1}^n\right)_t\mathcal{S}_1\left( w_{j_2}^n\right)_tdt\P_{A''}\left(\left\{\epsilon_{A''}\in Q_{j_1}^n\right\}\cap \left\{\epsilon_{A''}\in Q_{j_2}^n\right\}\right)
			\nonumber
			\\
			&=
			\sum_{j_1=1}^{N_n}\sum_{j_2=1}^{N_n}
			\int_{[T_1,T_2]}\mathcal{S}_1\left(w_{j_1}^n\right)_t\mathcal{S}_1\left( w_{j_2}^n\right)_tdt\P_{A'}\left(\epsilon_{A'}\in Q_{j_1}^n\cap Q_{j_2}^n \right) 
						\nonumber
			\\
			&=
			\sum_{j_1=1}^{N_n}\sum_{j_2=1}^{N_n}
			\int_{[T_1,T_2]}\mathcal{S}_1\left(w_{j_1}^n\right)_t\mathcal{S}_1\left( w_{j_2}^n\right)_tdt\P_{A''}\left(\epsilon^{A'}\in Q_{j_1}^n\cap Q_{j_2}^n\right) 
									\nonumber
			\\
			&=
						\E_{A'}\left[\int_{[T_1,T_2]}\mathcal{S}_1\left(\epsilon^{A',n}\right)_t^2dt\right].
		\end{align}	
By passing to the limit we obtain
\begin{align}\label{law2}
\E_{A'}\left[\int_{[T_1,T_2]}\mathcal{S}_1\left(\epsilon^{A'}\right)_t^2dt\right]
=
\E_{A''}\left[\int_{[T_1,T_2]}\mathcal{S}_1\left(\epsilon^{A''}\right)_t^2dt\right].
\end{align}
Returning to the first term of \eqref{RDelta},
		\begin{align*}
			&\left|\E_{A'}\left[\int_{[T_1,T_2]}\left(Y^{A'}_t\right)^2dt\right]-\E_{A''}\left[\int_{[T_1,T_2]}\left(Y^{A''}_tdt\right)^2\right]\right|
			\\
			&=
			\left|\E_{A'}\left[\int_{[T_1,T_2]}\mathcal{S}_1\left(A'+\epsilon^{A'}\right)_t^2dt\right]
-
			\E_{A''}\left[\int_{[T_1,T_2]}\mathcal{S}_1\left(A''+\epsilon_{A''}\right)_t^2dt\right]\right|
			\\
			&=
			\left|\E\left[\int_{[T_1,T_2]}\left(\mathcal{S}_1\left(A'\right)_t^2-\mathcal{S}_1\left(A''\right)_t^2\right) dt\right]
			+
			2\E_{A'}\left[\int_{[T_1,T_2]}\mathcal{S}_1\left(A'\right)_t\mathcal{S}_1\left(\epsilon^{A'}\right)_tdt\right]-
			2\E_{A''}\left[\int_{[T_1,T_2]}\mathcal{S}_1\left(A''\right)_t\mathcal{S}_1\left(\epsilon_{A''}\right)_tdt\right]
			\right.
			\\
			&\left.
			+
			\E_{A'}\left[\int_{[T_1,T_2]}\mathcal{S}_1\left(\epsilon^{A'}\right)_t^2dt\right]
			-
			\E_{A''}\left[\int_{[T_1,T_2]}\mathcal{S}_1\left(\epsilon_{A''}\right)_t^2dt\right]\right|
			\\
			&\le
			\E\left[\left|\int_{[T_1,T_2]} \mathcal{S}_1\left(A'+A''\right)_t\mathcal{S}_1\left(A'-A''\right)_t dt\right|\right]
			+
			\E_{A'}\left[\left|\int_{[T_1,T_2]}\mathcal{S}_1\left(A'-A''\right)_t\mathcal{S}_1\left(\epsilon^{A'}\right)_tdt\right|\right]
			\\
			&\le
			\E\left[\left(\int_{[T_1,T_2]} \left(\mathcal{S}_1\left(A'+A''\right)_t\right)^2 dt\right)^{\frac12}\left(\int_{[T_1,T_2]} \left(\mathcal{S}_1\left(A'-A''\right)_t\right)^2dt\right)^{\frac12}\right]
			\\
			&+
			\E_{A'}\left[\left(\int_{[T_1,T_2]}\mathcal{S}_1\left(A'-A''\right)_t^2dt\right)^\frac12 \left(\int_{[T_1,T_2]}\mathcal{S}_1\left(\epsilon^{A'}\right)_t^2dt\right)^\frac12 \right]
			\\
			&=
			\E\left[\lVert \mathcal{S}_1(A'+A'') \rVert_{L^2([T_1,T_2])^{p+1}}\lVert \mathcal{S}_1(A'-A'') \rVert_{L^2([T_1,T_2])^{p+1}}\right]
			+
			\E_{A'}\left[\lVert \mathcal{S}_1(A'-A'') \rVert_{L^2([T_1,T_2])^{p+1}}\lVert \mathcal{S}_1( \epsilon^{A'}) \rVert_{L^2([T_1,T_2])^{p+1}}\right]
			\\
			&\le 
			\lVert \mathcal{S} \rVert^2 \E\left[\lVert A''-A'\rVert_{L^2([T_1,T_2])^{p+1}}^2\right]^{\frac12}\left(\E\left[\lVert A''+A'\rVert_{L^2([T_1,T_2])^{p+1}}^2\right]^{\frac12} +\E_{A'}\left[\lVert \mathcal{S}_1( \epsilon^{A'}) \rVert_{L^2([T_1,T_2])^{p+1}}^2\right]^\frac12\right)
			\\
			&\le 2\lVert \mathcal{S} \rVert^2\lVert A''-A'\rVert_{\mathcal{V}} \left(\lVert A'\rVert_{\mathcal{V}}+\lVert A''\rVert_{\mathcal{V}}+\lVert \epsilon\rVert_{\mathcal{V}} \right),
		\end{align*}
		where we applied the Cauchy-Schwarz inequality several times and utilized both \eqref{law1} and \eqref{law2} to cancel out terms.
\\
\textbf{Term 2 of} \eqref{RDelta}
\\
For the second term on the right-hand side of \eqref{RDelta} we first make the following expansion,
\begin{align}\label{TERM2}
&\E_{A'}\left[\int_{[T_1,T_2]}Y^{A'}_t \int_{[T_1,T_2]}(\beta(t,\tau))(i) X^{A'}_\tau(i) d\tau dt\right]
			-
			\E_{A''}\left[\int_{[T_1,T_2]}Y^{A''}_t \int_{[T_1,T_2]}(\beta(t,\tau))(i) X^{A''}_\tau(i) d\tau dt\right]\nonumber
			\\
			&=
			\E_{A'}\left[\int_{[T_1,T_2]}\mathcal{S}_1\left(A'+\epsilon^{A'}\right)_t \int_{[T_1,T_2]}(\beta(t,\tau))(i) \mathcal{S}_{i+1}\left(A'+\epsilon^{A'}\right)_\tau d\tau dt\right]\nonumber
			\\
			&-			\E_{A''}\left[\int_{[T_1,T_2]}\mathcal{S}_1\left(A''+\epsilon_{A''}\right)_t \int_{[T_1,T_2]}(\beta(t,\tau))(i) \mathcal{S}_{i+1}\left(A''+\epsilon_{A''}\right)_\tau d\tau dt\right]\nonumber
			\\
			&=
			\E\left[\int_{[T_1,T_2]}\mathcal{S}_1\left(A'-A''\right)_t \int_{[T_1,T_2]}(\beta(t,\tau))(i) \mathcal{S}_{i+1}\left(A'\right)_\tau d\tau dt\right]\nonumber
			\\
			&+
			\E\left[\int_{[T_1,T_2]}\mathcal{S}_1\left(A''\right)_t \int_{[T_1,T_2]}(\beta(t,\tau))(i) \mathcal{S}_{i+1}\left(A'-A''\right)_\tau d\tau dt\right]\nonumber
			\\
			&+
			\E_{A'}\left[\int_{[T_1,T_2]}\mathcal{S}_1\left(\epsilon^{A'}\right)_t \int_{[T_1,T_2]}(\beta(t,\tau))(i) \mathcal{S}_{i+1}\left(\epsilon^{A'}\right)_\tau d\tau dt\right]
			-
			\E_{A''}\left[\int_{[T_1,T_2]}\mathcal{S}_1\left(\epsilon_{A''}\right)_t \int_{[T_1,T_2]}(\beta(t,\tau))(i) \mathcal{S}_{i+1}\left(\epsilon_{A''}\right)_\tau d\tau dt\right]\nonumber
			\\
			&+
			\E_{A'}\left[\int_{[T_1,T_2]}\mathcal{S}_1\left(A'\right)_t \int_{[T_1,T_2]}(\beta(t,\tau))(i) \mathcal{S}_{i+1}\left(\epsilon^{A'}\right)_\tau d\tau dt\right]
			-
			\E_{A''}\left[\int_{[T_1,T_2]}\mathcal{S}_1\left(A''\right)_t \int_{[T_1,T_2]}(\beta(t,\tau))(i) \mathcal{S}_{i+1}\left(\epsilon_{A''}\right)_\tau d\tau dt\right]\nonumber
			\\
			&+
			\E_{A'}\left[\int_{[T_1,T_2]}\mathcal{S}_1\left(\epsilon^{A'}\right)_t \int_{[T_1,T_2]}(\beta(t,\tau))(i) \mathcal{S}_{i+1}\left(A'\right)_\tau d\tau dt\right]
			-
			\E_{A''}\left[\int_{[T_1,T_2]}\mathcal{S}_1\left(\epsilon_{A''}\right)_t \int_{[T_1,T_2]}(\beta(t,\tau))(i) \mathcal{S}_{i+1}\left(A''\right)_\tau d\tau dt\right]
\end{align}
We bound the first term on the right-most side above	
\begin{align}\label{Aprimbiss}
&\left|\E\left[\int_{[T_1,T_2]}\mathcal{S}_1\left(A'-A''\right)_t \int_{[T_1,T_2]}(\beta(t,\tau))(i) \mathcal{S}_{i+1}\left(A'\right)_\tau d\tau dt\right]\right|\nonumber
\\
&\le
\E\left[\left(\int_{[T_1,T_2]^2}\mathcal{S}_1\left(A'-A''\right)_t^2\mathcal{S}_{i+1}\left(A'\right)_\tau^2d\tau dt\right)^\frac12 \left(\int_{[T_1,T_2]}(\beta(t,\tau))^2(i)  d\tau dt\right)^\frac12\right]\nonumber
\\
&=
\E\left[\left(\int_{[T_1,T_2]}\mathcal{S}_1\left(A'-A''\right)_t^2dt \int_{[T_1,T_2]}\mathcal{S}_{i+1}\left(A'\right)_\tau^2d\tau\right)^\frac12 \left(\int_{[T_1,T_2]^2}(\beta(t,\tau))^2(i)  d\tau dt\right)^\frac12\right]\nonumber
\\
&=
\E\left[\lVert \mathcal{S}_1(A'-A'')  \rVert_{L^2([T_1,T_2]} \lVert \mathcal{S}_{i+1}\left(A'\right)\rVert_{L^2([T_1,T_2]} \right]\lVert \beta(i)  \rVert_{L^2([T_1,T_2]^2)}\nonumber
\\
&\le
\lVert \beta(i)  \rVert_{L^2([T_1,T_2]^2)} \lVert \mathcal{S} \rVert^2\lVert A'-A'' \rVert_{\mathcal{V}}\lVert A' \rVert_{\mathcal{V}}
\end{align}
analogously for the second term we have
\begin{align*}
\left|\E\left[\int_{[T_1,T_2]}\mathcal{S}_1\left(A''\right)_t \int_{[T_1,T_2]}(\beta(t,\tau))(i) \mathcal{S}_{i+1}\left(A'-A''\right)_\tau d\tau dt\right]\right|
\le
\lVert \beta(i)  \rVert_{L^2([T_1,T_2]^2)} \lVert \mathcal{S} \rVert^2\lVert A'-A'' \rVert_{\mathcal{V}}\lVert A'' \rVert_{\mathcal{V}}.
\end{align*}
For terms three and four of \eqref{TERM2}, analogously to \eqref{Aprimbiss},
\begin{align*}
\left|\E_{A'}\left[\int_{[T_1,T_2]}\mathcal{S}_1\left(\epsilon^{A',n}-\epsilon^{A'}\right)_t \int_{[T_1,T_2]}(\beta(t,\tau))(i) \mathcal{S}_{i+1}\left(\epsilon^{A',n}-\epsilon^{A'}\right)_\tau d\tau dt\right]\right|
\le \lVert \beta(i)  \rVert_{L^2([T_1,T_2]^2)} \lVert \mathcal{S} \rVert^2\lVert \epsilon^n-\epsilon\rVert_{\mathcal{V}}^2,
\end{align*}	
which converges to zero, implying
\begin{align*}
\lim_{n\to\infty}&\E_{A'}\left[\int_{[T_1,T_2]}\mathcal{S}_1\left(\epsilon^{A',n}-\epsilon^{A'}\right)_t \int_{[T_1,T_2]}(\beta(t,\tau))(i) \mathcal{S}_{i+1}\left(\epsilon^{A',n}-\epsilon^{A'}\right)_\tau d\tau dt\right]
\\
=
&\E_{A'}\left[\int_{[T_1,T_2]}\mathcal{S}_1\left(\epsilon^{A'}\right)_t \int_{[T_1,T_2]}(\beta(t,\tau))(i) \mathcal{S}_{i+1}\left(\epsilon^{A'}\right)_\tau d\tau dt\right].
\end{align*}
Meanwhile, arguing as in \eqref{simpleeps} shows that
\begin{align*}
\E_{A'}\left[\int_{[T_1,T_2]}\mathcal{S}_1\left(\epsilon^{A',n}\right)_t \int_{[T_1,T_2]}(\beta(t,\tau))(i) \mathcal{S}_{i+1}\left(\epsilon^{A',n}\right)_\tau d\tau dt\right]
=
\E_{A''}\left[\int_{[T_1,T_2]}\mathcal{S}_1\left(\epsilon_{A''}^n\right)_t \int_{[T_1,T_2]}(\beta(t,\tau))(i) \mathcal{S}_{i+1}\left(\epsilon_{A''}^n\right)_\tau d\tau dt\right].
\end{align*}	
Similarly to \eqref{law1o} this leads to
\begin{align*}
\E_{A'}\left[\int_{[T_1,T_2]}\mathcal{S}_1\left(\epsilon^{A'}\right)_t \int_{[T_1,T_2]}(\beta(t,\tau))(i) \mathcal{S}_{i+1}\left(\epsilon^{A'}\right)_\tau d\tau dt\right]
=
\E_{A''}\left[\int_{[T_1,T_2]}\mathcal{S}_1\left(\epsilon_{A''}\right)_t \int_{[T_1,T_2]}(\beta(t,\tau))(i) \mathcal{S}_{i+1}\left(\epsilon_{A''}\right)_\tau d\tau dt\right]
,
\end{align*}
hence terms three and four of \eqref{TERM2} cancel. Terms five and six of \eqref{TERM2} are handled with similar techniques,
\begin{align*}
\E_{A'}\left[\int_{[T_1,T_2]}\mathcal{S}_1\left(A'\right)_t \int_{[T_1,T_2]}(\beta(t,\tau))(i) \mathcal{S}_{i+1}\left(\epsilon^{A'}\right)_\tau d\tau dt\right]
=
\E_{A''}\left[\int_{[T_1,T_2]}\mathcal{S}_1\left(A'\right)_t \int_{[T_1,T_2]}(\beta(t,\tau))(i) \mathcal{S}_{i+1}\left(\epsilon_{A''}\right)_\tau d\tau dt\right]
\end{align*}
implying
\small\begin{align*}
\E_{A'}\left[\int_{[T_1,T_2]}\mathcal{S}_1\left(A'\right)_t \int_{[T_1,T_2]}(\beta(t,\tau))(i) \mathcal{S}_{i+1}\left(\epsilon^{A'}\right)_\tau d\tau dt\right]
&=
\E_{A'}\left[\int_{[T_1,T_2]}\mathcal{S}_1\left(A'-A''\right)_t \int_{[T_1,T_2]}(\beta(t,\tau))(i) \mathcal{S}_{i+1}\left(\epsilon^{A'}\right)_\tau d\tau dt\right]
\\
&+
\E_{A''}\left[\int_{[T_1,T_2]}\mathcal{S}_1\left(A''\right)_t \int_{[T_1,T_2]}(\beta(t,\tau))(i) \mathcal{S}_{i+1}\left(\epsilon_{A''}\right)_\tau d\tau dt\right].
\end{align*}\normalsize
Similarly for terms seven and eight of \eqref{TERM2},
\small\begin{align*}
\E_{A'}\left[\int_{[T_1,T_2]}\mathcal{S}_1\left(\epsilon^{A'}\right)_t \int_{[T_1,T_2]}(\beta(t,\tau))(i) \mathcal{S}_{i+1}\left(A'\right)_\tau d\tau dt\right]
&=
\E_{A'}\left[\int_{[T_1,T_2]}\mathcal{S}_1\left(\epsilon^{A'}\right)_t \int_{[T_1,T_2]}(\beta(t,\tau))(i) \mathcal{S}_{i+1}\left(A'-A''\right)_\tau d\tau dt\right]
\\
&+
\E_{A''}\left[\int_{[T_1,T_2]}\mathcal{S}_1\left(\epsilon_{A''}\right)_t \int_{[T_1,T_2]}(\beta(t,\tau))(i) \mathcal{S}_{i+1}\left(A''\right)_\tau d\tau dt\right]
\end{align*}\normalsize
This leads to
\begin{align*}
&\left|\E_{A'}\left[\int_{[T_1,T_2]}Y^{A'}_t \int_{[T_1,T_2]}(\beta(t,\tau))(i) X^{A'}_\tau(i) d\tau dt\right]
			-
			\E_{A''}\left[\int_{[T_1,T_2]}Y^{A''}_t \int_{[T_1,T_2]}(\beta(t,\tau))(i) X^{A''}_\tau(i) d\tau dt\right]\right|
			\\
			&\le \left|\E\left[\int_{[T_1,T_2]}\mathcal{S}_1\left(A'-A''\right)_t \int_{[T_1,T_2]}(\beta(t,\tau))(i) \mathcal{S}_{i+1}\left(A'\right)_\tau d\tau dt\right]\right|
						\\
			&+
			\left|\E\left[\int_{[T_1,T_2]}\mathcal{S}_1\left(A''\right)_t \int_{[T_1,T_2]}(\beta(t,\tau))(i) \mathcal{S}_{i+1}\left(A'-A''\right)_\tau d\tau dt\right]\right|
			\\
			&+
			\left|\E_{A'}\left[\int_{[T_1,T_2]}\mathcal{S}_1\left(\epsilon^{A'}\right)_t \int_{[T_1,T_2]}(\beta(t,\tau))(i) \mathcal{S}_{i+1}\left(A'-A''\right)_\tau d\tau dt\right]\right|
			\\
			&+
						\left|\E_{A'}\left[\int_{[T_1,T_2]}\mathcal{S}_1\left(A'-A''\right)_t \int_{[T_1,T_2]}(\beta(t,\tau))(i) \mathcal{S}_{i+1}\left(\epsilon^{A'}\right)_\tau d\tau dt\right]\right|
			\\
			&\le	
			\lVert \beta(i)  \rVert_{L^2([T_1,T_2]^2)} \lVert \mathcal{S} \rVert^2\lVert A'-A'' \rVert_{\mathcal{V}}\left(\lVert A' \rVert_{\mathcal{V}}+\lVert A'' \rVert_{\mathcal{V}}\right)	
+
2\lVert \beta(i)  \rVert_{L^2([T_1,T_2]^2)}\lVert \mathcal{S} \rVert^2\lVert A'-A'' \rVert_{\mathcal{V}}\lVert \epsilon \rVert_{\mathcal{V}}.
\end{align*}
\textbf{Term 3 of} \eqref{RDelta}
\\	
We expand the third term of \eqref{RDelta},
		\begin{align}\label{T3}
		&\E_{A'}\left[\int_{[T_1,T_2]}\left(\int_{[T_1,T_2]}(\beta(t,\tau))(i) X^{A'}_\tau(i) d\tau \right)^2dt\right]-
			\E_{A''}\left[\int_{[T_1,T_2]}\left(\int_{[T_1,T_2]}(\beta(t,\tau))(i) X^{A''}_\tau(i) d\tau \right)^2dt\right]\nonumber
			\\
			&=
			\E_{A'}\left[\int_{[T_1,T_2]}\left(\int_{[T_1,T_2]}(\beta(t,\tau))(i) S_{i+1}\left( A'+\epsilon^{A'}\right)_\tau(i) d\tau \right)^2dt\right]\nonumber
			\\
			&-
			\E_{A''}\left[\int_{[T_1,T_2]}\left(\int_{[T_1,T_2]}(\beta(t,\tau))(i) S_{i+1}\left( A''+\epsilon_{A''}\right)_\tau(i) d\tau \right)^2dt\right]\nonumber
			\\
			&=
			\E\left[\int_{[T_1,T_2]}\left(\int_{[T_1,T_2]} \mathcal{S}_{i+1}\left(A''\right)_\tau(\beta(t,\tau))(i) d\tau\right)^2dt\right]
-
\E\left[\int_{[T_1,T_2]}\left(\int_{[T_1,T_2]} \mathcal{S}_{i+1}\left(A'\right)_\tau(\beta(t,\tau))(i)d\tau \right)^2dt\right]\nonumber
\\
&+\E_{A'}\left[\int_{[T_1,T_2]}\left(\int_{[T_1,T_2]} \mathcal{S}_{i+1}\left(\epsilon^{A'}\right)_\tau(\beta(t,\tau))(i) d\tau\right)^2dt\right]
-\E_{A''}\left[\int_{[T_1,T_2]}\left(\int_{[T_1,T_2]} \mathcal{S}_{i+1}\left(\epsilon_{A''}\right)_\tau(\beta(t,\tau))(i) d\tau\right)^2dt\right]\nonumber
\\
&+\E_{A'}\left[\int_{[T_1,T_2]}\int_{[T_1,T_2]} \mathcal{S}_{i+1}\left(\epsilon^{A'}\right)_\tau(\beta(t,\tau))(i) d\tau \int_{[T_1,T_2]} \mathcal{S}_{i+1}\left(A'\right)_\tau(\beta(t,\tau))(i) d\tau dt\right]\nonumber
\\
&-\E_{A''}\left[\int_{[T_1,T_2]}\int_{[T_1,T_2]} \mathcal{S}_{i+1}\left(\epsilon_{A''}\right)_\tau(\beta(t,\tau))(i) d\tau \int_{[T_1,T_2]} \mathcal{S}_{i+1}\left(A''\right)_\tau(\beta(t,\tau))(i) d\tau dt\right].
		\end{align}
For the first term difference the right-hand side of \eqref{T3}, note that
\begin{align*}
&\left|\E\left[\int_{[T_1,T_2]}\left(\int_{[T_1,T_2]} \mathcal{S}_{i+1}\left(A''\right)_\tau(\beta(t,\tau))(i) d\tau\right)^2dt\right]
-
\E\left[\int_{[T_1,T_2]}\left(\int_{[T_1,T_2]} \mathcal{S}_{i+1}\left(A'\right)_\tau(\beta(t,\tau))(i)d\tau \right)^2dt\right]\right|
\\
=
&\left|\E\left[\int_{[T_1,T_2]}\left(\int_{[T_1,T_2]} \mathcal{S}_{i+1}\left(A''-A''\right)_\tau(\beta(t,\tau))(i)d\tau \right)\left(\int_{[T_1,T_2]} \mathcal{S}_{i+1}\left(A''+A''\right)_\tau(\beta(t,\tau))(i)d\tau \right)dt\right]
\right|
\\
\le
&\E\left[\left(\int_{[T_1,T_2]}\left(\int_{[T_1,T_2]} \mathcal{S}_{i+1}\left(A''-A''\right)_\tau(\beta(t,\tau))(i)d\tau \right)^2dt\right)^\frac12 
\right.
\\
&\left.\left(\int_{[T_1,T_2]}\left(\int_{[T_1,T_2]} \mathcal{S}_{i+1}\left(A''+A''\right)_\tau(\beta(t,\tau))(i)d\tau \right)^2dt\right)^\frac12\right]
\\
\le
&\E\left[\left(\int_{[T_1,T_2]}\int_{[T_1,T_2]} \mathcal{S}_{i+1}\left(A''-A''\right)_\tau^2d\tau \int_{[T_1,T_2]} (\beta(t,\tau))^2(i)d\tau dt\right)^\frac12 
\right.
\\
&\left.
\left(\int_{[T_1,T_2]}\int_{[T_1,T_2]} \mathcal{S}_{i+1}\left(A''-A''\right)_\tau^2d\tau  \int_{[T_1,T_2]} (\beta(t,\tau))^2(i)d\tau dt\right)^\frac12 \right]
\\
\le
&\E\left[\int_{[T_1,T_2]}\left(\int_{[T_1,T_2]} \mathcal{S}_{i+1}\left(A''-A''\right)_\tau^2d\tau \right)^\frac12 \left(\int_{[T_1,T_2]} (\beta(t,\tau))^2(i)d\tau \right)^\frac12 dt\right]^\frac12 
\\
&=
\lVert  \beta(i) \rVert_{L^2([T_1,T_2]^2)}^2\E\left[ \lVert \mathcal{S}_{i+1}\left(A''-A''\right) \rVert_{L^2([T_1,T_2])}\lVert \mathcal{S}_{i+1}\left(A''+A''\right) \rVert_{L^2([T_1,T_2])}\right] 
\\
&\le
\lVert  \beta(i) \rVert_{L^2([T_1,T_2]^2)}^2\E\left[ \lVert \mathcal{S}_{i+1}\left(A''-A''\right) \rVert_{L^2([T_1,T_2])}^2\right]^\frac12\E\left[ \lVert \mathcal{S}_{i+1}\left(A''+A''\right) \rVert_{L^2([T_1,T_2])}^2\right]^\frac12
\\
&\le
\lVert  \beta(i) \rVert_{L^2([T_1,T_2]^2)}^2 \lVert \mathcal{S} \rVert^2 \lVert A''-A' \rVert_{\mathcal{V}}\left(\lVert A''\rVert_{\mathcal{V}}+\lVert A'\rVert_{\mathcal{V}}\right).
\end{align*}	
For the second difference on the right-hand side of \eqref{T3}, 
\begin{align*}
\E_{A''}\left[\int_{[T_1,T_2]}\left(\int_{[T_1,T_2]} \mathcal{S}_{i+1}\left(\epsilon_{A''}^n-\epsilon_{A''}\right)_\tau(\beta(t,\tau))(i) d\tau\right)^2dt\right]
&\le
\lVert \beta(i) \rVert_{L^2([T_1,T_2])}^2\E_{A''}\left[\lVert \mathcal{S}_{i+1}\left(\epsilon_{A''}^n-\epsilon_{A''}\right) \rVert_{L^2([T_1,T_2])}^2\right]
\\
&\le
\lVert \beta(i) \rVert_{L^2([T_1,T_2])}^2 \E_{A''}\left[\lVert \mathcal{S}_{i+1}\left(\epsilon_{A''}^n-\epsilon_{A''}\right) \rVert_{L^2([T_1,T_2])}^2\right]
\\
&\le
\lVert  \beta(i) \rVert_{L^2([T_1,T_2]^2)}^2 \lVert \mathcal{S} \rVert^2 \lVert \epsilon_{A''}^n-\epsilon_{A''} \rVert_{\mathcal{V}}^2,
\end{align*}
which converges to zero, implying that
$$\lim_{n\to\infty}\E_{A''}\left[\int_{[T_1,T_2]}\left(\int_{[T_1,T_2]} \mathcal{S}_{i+1}\left(\epsilon_{A''}^n\right)_\tau(\beta(t,\tau))(i) d\tau\right)^2dt\right] 
=
\E_{A''}\left[\int_{[T_1,T_2]}\left(\int_{[T_1,T_2]} \mathcal{S}_{i+1}\left(\epsilon_{A''}\right)_\tau(\beta(t,\tau))(i) d\tau\right)^2dt\right]$$
and analogously
$$\lim_{n\to\infty}\E_{A''}\left[\int_{[T_1,T_2]}\left(\int_{[T_1,T_2]} \mathcal{S}_{i+1}\left(\epsilon_{A''}^n\right)_\tau(\beta(t,\tau))(i) d\tau\right)^2dt\right] 
=
\E_{A'}\left[\int_{[T_1,T_2]}\left(\int_{[T_1,T_2]} \mathcal{S}_{i+1}\left(\epsilon^{A'}\right)_\tau(\beta(t,\tau))(i) d\tau\right)^2dt\right].$$
Meanwhile, analogously to \eqref{simpleeps},
$$\E_{A''}\left[\int_{[T_1,T_2]}\left(\int_{[T_1,T_2]} \mathcal{S}_{i+1}\left(\epsilon_{A''}^n\right)_\tau(\beta(t,\tau))(i) d\tau\right)^2dt\right]
=
\E_{A'}\left[\int_{[T_1,T_2]}\left(\int_{[T_1,T_2]} \mathcal{S}_{i+1}\left(\epsilon^{A',n}\right)_\tau(\beta(t,\tau))(i) d\tau\right)^2dt\right], $$
which therefore leads to
$$\E_{A''}\left[\int_{[T_1,T_2]}\left(\int_{[T_1,T_2]} \mathcal{S}_{i+1}\left(\epsilon_{A''}\right)_\tau(\beta(t,\tau))(i) d\tau\right)^2dt\right]
=
\E_{A'}\left[\int_{[T_1,T_2]}\left(\int_{[T_1,T_2]} \mathcal{S}_{i+1}\left(\epsilon^{A'}\right)_\tau(\beta(t,\tau))(i) d\tau\right)^2dt\right]. $$
Hence the second difference in \eqref{T3} is cancelled.
For the third difference in \eqref{T3},
\begin{align*}
&\left|\E_{A''}\left[\int_{[T_1,T_2]}\int_{[T_1,T_2]} \mathcal{S}_{i+1}\left(\epsilon_{A''}-\epsilon_{A''}^n\right)_\tau(\beta(t,\tau))(i) d\tau \int_{[T_1,T_2]} \mathcal{S}_{i+1}\left(A''\right)_\tau(\beta(t,\tau))(i) d\tau dt\right]\right|
\\
&\le
\int_{[T_1,T_2]} \int_{[T_1,T_2]}(\beta(t,\tau))^2(i) d\tau  dt\E_{A''}\left[\left(\int_{[T_1,T_2]} \mathcal{S}_{i+1}\left(\epsilon_{A''}-\epsilon_{A''}^n\right)_\tau^2d\tau\right)^\frac12 \left(\int_{[T_1,T_2]} \mathcal{S}_{i+1}\left(A''\right)_\tau^2d\tau\right)^\frac12 \right]
\\
&=
\lVert  \beta(i) \rVert_{L^2([T_1,T_2]^2)}^2
\E_{A''}\left[\lVert  \mathcal{S}_{i+1}\left(\epsilon_{A''}-\epsilon_{A''}^n\right) \rVert_{L^2([T_1,T_2])} \lVert  \mathcal{S}_{i+1}\left(A''\right) \rVert_{L^2([T_1,T_2])} \right]
\\
&\le
\lVert  \beta(i) \rVert_{L^2([T_1,T_2]^2)}^2
\E_{A''}\left[\lVert  \mathcal{S}_{i+1}\left(\epsilon_{A''}-\epsilon_{A''}^n\right) \rVert_{L^2([T_1,T_2])}^2\right]^\frac12
\E_{A''}\left[\lVert  \mathcal{S}_{i+1}\left(A''\right) \rVert_{L^2([T_1,T_2])}^2\right]^\frac12
\\
&\le
\lVert  \beta(i) \rVert_{L^2([T_1,T_2]^2)}^2
\lVert  S \rVert^2 \lVert  \epsilon_{A''}-\epsilon_{A''}^n \rVert_{\mathcal{V}}\lVert  A'' \rVert_{\mathcal{V}},
\end{align*}
which converges to zero, implying that
\begin{align*}
\lim_{n\to\infty}&\E_{A''}\left[\int_{[T_1,T_2]}\int_{[T_1,T_2]} \mathcal{S}_{i+1}\left(\epsilon_{A''}^n\right)_\tau(\beta(t,\tau))(i) d\tau \int_{[T_1,T_2]} \mathcal{S}_{i+1}\left(A''\right)_\tau(\beta(t,\tau))(i) d\tau dt\right]
\\
=
&\E_{A''}\left[\int_{[T_1,T_2]}\int_{[T_1,T_2]} \mathcal{S}_{i+1}\left(\epsilon_{A''}\right)_\tau(\beta(t,\tau))(i) d\tau \int_{[T_1,T_2]} \mathcal{S}_{i+1}\left(A''\right)_\tau(\beta(t,\tau))(i) d\tau dt\right]
\end{align*}
and analogously
\begin{align*}
\lim_{n\to\infty}&\E_{A'}\left[\int_{[T_1,T_2]}\int_{[T_1,T_2]} \mathcal{S}_{i+1}\left(\epsilon^{A',n}\right)_\tau(\beta(t,\tau))(i) d\tau \int_{[T_1,T_2]} \mathcal{S}_{i+1}\left(A'\right)_\tau(\beta(t,\tau))(i) d\tau dt\right]
\\
=
&\E_{A'}\left[\int_{[T_1,T_2]}\int_{[T_1,T_2]} \mathcal{S}_{i+1}\left(\epsilon^{A'}\right)_\tau(\beta(t,\tau))(i) d\tau \int_{[T_1,T_2]} \mathcal{S}_{i+1}\left(A'\right)_\tau(\beta(t,\tau))(i) d\tau dt\right]
\end{align*}
Meanwhile
\begin{align*}
&\E_{A'}\left[\int_{[T_1,T_2]}\int_{[T_1,T_2]} \mathcal{S}_{i+1}\left(\epsilon^{A',n}\right)_\tau(\beta(t,\tau))(i) d\tau \int_{[T_1,T_2]} \mathcal{S}_{i+1}\left(A'\right)_\tau(\beta(t,\tau))(i) d\tau dt\right]
\\
=
&\E_{A''}\left[\int_{[T_1,T_2]}\int_{[T_1,T_2]} \mathcal{S}_{i+1}\left(\epsilon_{A''}^n\right)_\tau(\beta(t,\tau))(i) d\tau \int_{[T_1,T_2]} \mathcal{S}_{i+1}\left(A'\right)_\tau(\beta(t,\tau))(i) d\tau dt\right]
\end{align*}
which therefore leads to
\begin{align*}
&\E_{A'}\left[\int_{[T_1,T_2]}\int_{[T_1,T_2]} \mathcal{S}_{i+1}\left(\epsilon^{A'}\right)_\tau(\beta(t,\tau))(i) d\tau \int_{[T_1,T_2]} \mathcal{S}_{i+1}\left(A'\right)_\tau(\beta(t,\tau))(i) d\tau dt\right]
\\
=
&\E_{A''}\left[\int_{[T_1,T_2]}\int_{[T_1,T_2]} \mathcal{S}_{i+1}\left(\epsilon_{A''}\right)_\tau(\beta(t,\tau))(i) d\tau \int_{[T_1,T_2]} \mathcal{S}_{i+1}\left(A'\right)_\tau(\beta(t,\tau))(i) d\tau dt\right].
\end{align*}
This implies, 
		\begin{align}\label{lawX}
		&\E_{A'}\left[\int_{[T_1,T_2]}\int_{[T_1,T_2]} \mathcal{S}_{i+1}\left(\epsilon^{A'}\right)_\tau(\beta(t,\tau))(i) d\tau \int_{[T_1,T_2]} \mathcal{S}_{i+1}\left(A'\right)_\tau(\beta(t,\tau))(i) d\tau dt\right]\nonumber
		\\
		&=
		\E_{A'}\left[\int_{[T_1,T_2]}\int_{[T_1,T_2]} \mathcal{S}_{i+1}\left(\epsilon^{A'}\right)_\tau(\beta(t,\tau))(i) d\tau \int_{[T_1,T_2]} \mathcal{S}_{i+1}\left(A'-A''\right)_\tau(\beta(t,\tau))(i) d\tau dt\right]\nonumber
		\\
		&+
		\E_{A''}\left[\int_{[T_1,T_2]}\int_{[T_1,T_2]} \mathcal{S}_{i+1}\left(\epsilon^{A'}\right)_\tau(\beta(t,\tau))(i) d\tau \int_{[T_1,T_2]} \mathcal{S}_{i+1}\left(A''\right)_\tau(\beta(t,\tau))(i) d\tau dt\right],
		\end{align}
which allows us to compute the following bound for the third difference in \eqref{T3},
\begin{align}\label{lawXX}
		&\left|\E_{A'}\left[\int_{[T_1,T_2]}\int_{[T_1,T_2]} \mathcal{S}_{i+1}\left(\epsilon^{A'}\right)_\tau(\beta(t,\tau))(i) d\tau \int_{[T_1,T_2]} \mathcal{S}_{i+1}\left(A'\right)_\tau(\beta(t,\tau))(i) d\tau dt\right]\right.\nonumber
		\\
		&\left.-
				\E_{A''}\left[\int_{[T_1,T_2]}\int_{[T_1,T_2]} \mathcal{S}_{i+1}\left(\epsilon^{A'}\right)_\tau(\beta(t,\tau))(i) d\tau \int_{[T_1,T_2]} \mathcal{S}_{i+1}\left(A''\right)_\tau(\beta(t,\tau))(i) d\tau dt\right]\right|\nonumber
		\\
		&=\left| \E_{A'}\left[\int_{[T_1,T_2]}\int_{[T_1,T_2]} \mathcal{S}_{i+1}\left(\epsilon^{A'}\right)_\tau(\beta(t,\tau))(i) d\tau \int_{[T_1,T_2]} \mathcal{S}_{i+1}\left(A'-A''\right)_\tau(\beta(t,\tau))(i) d\tau dt\right]\right|\nonumber
		\\
		&\le \lVert  \beta(i) \rVert_{L^2([T_1,T_2]^2)}^2 \lVert \mathcal{S} \rVert^2 \lVert A''-A' \rVert_{\mathcal{V}}\lVert \epsilon\rVert_{\mathcal{V}}.
		\end{align}
We may now bound the third term of \eqref{RDelta},
\begin{align*}
&\left|\E_{A'}\left[\int_{[T_1,T_2]}\left(\int_{[T_1,T_2]}(\beta(t,\tau))(i) X^{A'}_\tau(i) d\tau \right)^2dt\right]-
			\E_{A''}\left[\int_{[T_1,T_2]}\left(\int_{[T_1,T_2]}(\beta(t,\tau))(i) X^{A''}_\tau(i) d\tau \right)^2dt\right] \right|
			\\
			&\le 
			\lVert  \beta(i) \rVert_{L^2([T_1,T_2]^2)}^2 \lVert \mathcal{S} \rVert^2 \lVert A''-A' \rVert_{\mathcal{V}}\left(\lVert A''\rVert_{\mathcal{V}}+\lVert A'\rVert_{\mathcal{V}}+\lVert \epsilon\rVert_{\mathcal{V}}\right)
\end{align*}
It now follows that there exists some $E\in \R^+$ (depending on $\lVert \beta\rVert^2_{L^2([T_1,T_2]^2)^p}$ and $\lVert \mathcal{S}\rVert^2$) such that
		\begin{align}\label{Rineq1}
			\left| R_{A'}(\beta)-R_{A''}(\beta)\right|\le E \left(\lVert A' \rVert_{\mathcal{V}}+\lVert A'' \rVert_{\mathcal{V}}+\lVert \epsilon\rVert_{\mathcal{V}}\right)\lVert A'-A'' \rVert_{\mathcal{V}}
		\end{align}
		from which the result follows.
	\end{proof}
    We now proceed with the proof of the main Theorem.
	\begin{proof}[Proof of Theorem \ref{WR1}]
		\textbf{Step 1: Establish regularity/summability properties for the target and the covariates and expand the risk function in terms of scores.}
		\\
		Take $A'\in\mathcal{V}$, clearly,
		$$(Y^{A'},X^{A'})=\mathcal{S}\left({A'}+\epsilon^{A'}\right) .$$
		Therefore, 
		\begin{align*}
			\E_{A'}\left[\int_{[T_1,T_2]} (Y_t^{A'})^2dt\right]
			&=
			\E_{A'}\left[\int_{[T_1,T_2]} (Y_t^{A'})^2dt\right]
			\\
			&=
			\E_{A'}\left[\int_{[T_1,T_2]} \pi_1\left(\mathcal{S}\left(A'+\epsilon^{A'}\right)\right)^2_tdt\right]
			\\
			&=
			\E_{A'}\left[\n \left(\pi_1\left(\mathcal{S}\left(A'+\epsilon^{A'}\right)\right),0\ldots,0\right)\n_{L^2([T_1,T_2])^{p+1}}^2\right]
			\\
			&\le
			\n \mathcal{S} \n^2 \E_{A'}\left[\n A'+\epsilon^{A'}\n_{L^2([T_1,T_2])^{p+1}}^2\right]
			\\
			&\le
			\n \mathcal{S} \n^2 \E_{A'}\left[2\n A'\n_{L^2([T_1,T_2])^{p+1}}^2+2\n\epsilon^{A'}\n_{L^2([T_1,T_2])^{p+1}}^2\right]<\infty,
		\end{align*}
		which also implies $Y^{A'}\in L^2([T_1,T_2])$ a.s.. Analogously we have that $\E\left[\int_{[T_1,T_2]} (X_t(i)^{A'})^2dt\right]<\infty$ and $X^{A'}(i)\in L^2([T_1,T_2])$ a.s., for $1\le i\le p$.  Consider $\mathcal{K}_{X^{A'}(i)} :L^2([T_1,T_2])\to L^2([T_1,T_2])$, defined by 
$$(\mathcal{K}_{X^{A'}(i)} f)(t)=\int_{[T_1,T_2]}K_{X^{A'}(i)}(s,t)f(s)ds,$$
for $f\in L^2([T_1,T_2])$. By the Cauchy-Schwarz inequality
\begin{align*}
\int_{[T_1,T_2]^2}K_{X^{A'}(i)}(s,t)^2dsdt
&=
\int_{[T_1,T_2]^2}\E_{A'}\left[X^{A'}_s(i)X^{A'}_t(i)\right]^2dsdt
\\
&\le
\int_{[T_1,T_2]^2}\E_{A'}\left[\left(X_s^{A'}(i)\right)^2\right]\E_{A'}\left[\left(X_t^{A'}(j)\right)^2\right]dsdt
\\
&=
\left(\int_{[T_1,T_2]}\E_{A'}\left[\left(X_t^{A'}(i)\right)^2\right]dt\right)^2,
\end{align*}
which is finite by the assumption on $A$ and $\epsilon$.  It follows that $\mathcal{K}_{X^{A'}(i)}$ induces a compact self-adjoint operator on $L^2([T_1,T_2])^p$ and therefore, by the Hilbert-Schmidt theorem, $\mathcal{K}_{X^{A'}(i)}$ has an eigendecomposition, $\mathcal{K}_{X^{A'}(i)}f=\sum_{k=1}^\infty\alpha_{i,k}\langle \psi_{i,k},f\rangle\psi_{i,k}$, for $f\in L^2([T_1,T_2])$, where its eigenfunctions, $\{\psi_{i,k}\}_{k\in\N}$, are orthonormal in $L^2([T_1,T_2])$ and the series converges in $L^2([T_1,T_2])$. If $\mathsf{Ker}\left(\mathcal{K}_{X^{A'}(i)}\right)\not=\{0\}$ then there exists an ON-basis, $\{\eta_{i,l}\}_{l\in\N}$ for $\mathsf{Ker}\left(\mathcal{K}_{X^{A'}}\right)$. Let $\mathcal{L}_1=\overline{\mathsf{span}}\left(\{\psi_{i,k}\}_{k\in\N}\right)$ and $\mathcal{L}_2:=\mathcal{L}_1^{\perp}=\mathsf{Ker}\left(\mathcal{K}_{X^{A'}}\right)$. Then $\mathcal{L}_2$ is a closed subspace of $L^2([T_1,T_2])$ and is therefore separable, which implies it has a countable ON-basis, $\{\eta_k\}_k$ and moreover $L^2([T_1,T_2])=\mathcal{L}_1\oplus \mathcal{L}_2$ as well as $L^2([T_1,T_2])=\overline{\mathsf{span}}\left(\{\psi_{i,k}\}_k\cup \{\eta_{i,k}\}_k\right)$. Let $\{\tilde{\psi}_{i,k}\}_k$ be an enumeration of the ON-system $\{\psi_{i,k}\}_{k\in\N}\cup \{\eta_{i,k}\}_{k\in\N}$, which is a basis for $L^2([T_1,T_2])$. Take some arbitrary ON-basis of $L^2([T_1,T_2])$, $\{\phi_{n}\}_{n\in\N}$ and define 
		\begin{itemize}
			\item $Z^{A'}_k=\int_{[T_1,T_2]}Y^{A'}_t\phi_k(t)dt$ and
			\item  $\chi^{A'}_k(i)=\int_{[T_1,T_2]}X^{A'}_t(i)\tilde{\psi}_{i,k}(t)dt$.
		\end{itemize}
Then it follows that $X^{A'}(i)\in L^2([T_1,T_2])$ a.s. and hence if we let $S_n^{i*}(t) = 
\sum_{k=1}^n\left(\chi^{(1),A'}_k(i)\psi_{i,k}(t)+\chi^{(2),A'}_k(i)\eta_{i,k}(t)\right)$ then $S_n^{i*}\xrightarrow{L^2([T_1,T_2])}X^{A'}(i)$ a.s. (since $\{\tilde{\psi}_{i,k}\}_k$ is an ON-basis for $L^2([T_1,T_2])$). Moreover, note that
\begin{align*}
\E_{A'}\left[\left(\chi^{(2),A'}_{l}\right)^2\right]
&=
\E_{A'}\left[\int_0^T \int_0^T \eta_{i,l}(s)X^{A'}_s(i)X^{A'}_t(i)\eta_{i,l}(t)dtds\right]
\\
&=
\int_0^T \int_0^T \eta_{i,l}(s)\E_{A'}\left[X_s^{A'}(i)X_t^{A'}(i)\right]\eta_{i,l}(t)dtds
\\
&=\int_0^T \int_0^T \eta_{i,l}(s)K_{X^{A'}(i)}(t,s)\eta_{i,l}(t)dtds
\\
&=\int_0^T (\mathcal{K}_{X^{A'}(i)}\eta_{i,l})(t)\eta_{i,l}(t)^Tdtds
\\
&=\lim_{N\to\infty}\int_0^T \sum_{n=1}^N \alpha_{i,n}\langle\psi_{i,n},\eta_{i,l}\rangle_{L^2([T_1,T_2])}\psi_{i,n}(t)\eta_{i,l}(t)dt
\\
&=\lim_{N\to\infty} \sum_{n=1}^N \alpha_{i,n}\langle\psi_{i,n},\eta_{i,l}\rangle_{L^2([T_1,T_2])} \int_0^T  \psi_{i,n}(t) \eta_{i,l}(t)dt
=0.
\end{align*}
This implies that if we let $S_n^{X^{A'}(i)}(t)=\sum_{k=1}^n\chi^{(1),A'}_k(i)\psi_{i,k}(t)$, then $S_n^{X^{A'}(i)}\xrightarrow{L^2([T_1,T_2])}X^{A'}(i)$ a.s.. We will therefore denote $\chi_l^{A'}(i)=\chi^{(1),A'}_l(i)$ from now on. Since we expand $X^{A'}(i)$ in the basis given by the eigenfunctions of $K_{X^{A'}(i)}$ we also have that the sequence $\{\chi_l^{A'}(i)\}_{l\in\N}$ is orthogonal,
\begin{align}\label{orthscores}
\E\left[\chi_{l_1}^{A'}(i)\chi_{l_2}^{A'}(i)\right]
&=
\E\left[\int_{[T_1,T_2]} X^{A'}_s(i) \psi_{i,l_1}(s) ds\int_{[T_1,T_2]} X^{A'}_t(i)\psi_{i,l_2}(t)dt\right]\nonumber
\\
&=
\int_{[T_1,T_2]} \int_{[T_1,T_2]} \psi_{l_1}(s)\E\left[X^{A'}_s(i)X^{A'}_t(i)\right]\psi_{l_2}(t)dtds\nonumber
\\
&=\int_{[T_1,T_2]} \int_{[T_1,T_2]} \psi_{i,l_1}(s)K_{X^{A'}(i)}(t,s)\psi_{i,l_2}(t)dtds\nonumber
\\
&=\lim_{N\to\infty} \int_{[T_1,T_2]} \int_{[T_1,T_2]} \sum_{n=1}^N \alpha_n\langle \psi_{i,n},\psi_{i,l_1}\rangle_{L^2([T_1,T_2])}\psi_{i,n}(t) \psi_{i,l_2}(t)dtds\nonumber
\\
&=\lim_{N\to\infty} \sum_{n=1}^N  \alpha_n\langle \psi_{i,n},\psi_{i,l_1}\rangle_{L^2([T_1,T_2])}\langle \psi_{i,n},\psi_{i,l_2}\rangle_{L^2([T_1,T_2])}\nonumber
\\
&=\lim_{N\to\infty} \sum_{n=1}^N  \alpha_n\delta_{n,l_1}\delta_{n,l_2}
=\alpha_{l_1}\delta_{l_1,l_2}.
\end{align}
If also let $S_n^{Y^{A'}}(t)=\sum_{k=1}^nZ^{A'}_k\phi_k(t)$ then, since $\{\phi_k\}_{k\in\N}$ is an ON-basis, $S_n^{Y^{A'}}\xrightarrow{L^2([T_1,T_2])}Y^{A'}$ and $S_n^{X^{A'}(i)}\xrightarrow{L^2([T_1,T_2])}X^{A'}(i)$. Next, by monotone convergence
		\begin{align*}
			\int_{[T_1,T_2]} \E_{A'}\left[(X^{A'}_t(i))^2\right]dt 
			&=
			\E_{A'}\left[ \int_{[T_1,T_2]} (X^{A'}_t(i))^2dt \right]
			\\
			&=
			\E_{A'}\left[ \lim_{n\to\infty}\int_{[T_1,T_2]} \left(\sum_{k=1}^n\chi^{A'}_k(i)\psi_{i,k}(t)\right)^2dt \right]
			\\
			&=
			\E_{A'}\left[ \lim_{n\to\infty} \sum_{k=1}^n\sum_{l=1}^n\chi^{A'}_k(i)\chi^{A'}_l(i)\int_{[T_1,T_2]} \psi_{i,l}(t)\psi_{i,k}(t)dt \right]
						\\
			&=
			\E_{A'}\left[ \lim_{n\to\infty} \sum_{k=1}^n\sum_{l=1}^n\chi^{A'}_k(i)\chi^{A'}_l(i)\delta_{l,k} \right]
			\\
			&=
			\E_{A'}\left[ \sum_{k=1}^\infty(\chi^{A'}_k(i))^2 \right]
			=\sum_{k=1}^\infty\E_{A'}\left[ (\chi^{A'}_k(i))^2 \right],
		\end{align*}
		which implies $\sum_{k=1}^\infty\E\left[ (\chi_k^{A'}(i))^2 \right]<\infty$. Similarly $\sum_{k=1}^\infty\E\left[ (Z^{A'}_k)^2 \right]<\infty$.  Also, $S_n^{X^{A'}(i)}\xrightarrow{L^2(dt\times d\P_{A'})}X^A(i)$ ,
\begin{align}\label{XL2}
			\lim_{n\to\infty}\E_{A'}\left[ \int_{[T_1,T_2]} (S_n^{X^{A'}(i)}-X^{A'}_t(i))^2dt \right]\nonumber
			&=
			\lim_{n\to\infty}\E_{A'}\left[ \lim_{N\to\infty}\int_{[T_1,T_2]} \left(\sum_{k=1}^n\chi^{A'}_k(i)\psi_{i,k}(t)-\sum_{k=1}^N\chi^{A'}_k(i)\psi_{i,k}(t)\right)^2dt \right]\nonumber
			\\
			&=
			\lim_{n\to\infty}\E_{A'}\left[ \lim_{N\to\infty}\int_{[T_1,T_2]} \left(\sum_{k=n+1}^N\chi^{A'}_k(i)\psi_{i,k}(t)\right)^2dt \right]\nonumber
			\\
			&=
			\lim_{n\to\infty}\E_{A'}\left[ \lim_{N\to\infty}\sum_{k=n+1}^N(\chi^{A'}_k(i))^2 \right]\nonumber
			\\
			&=\lim_{n\to\infty}\sum_{k=n+1}^\infty\E_{A'}\left[ (\chi^{A'}_k(i))^2 \right]=0,
		\end{align}
		by monotone convergence, and analogously $S_n^{Y^{A'}}\xrightarrow{L^2(dt\times d\P_{A'})}Y^{A'}$. As $L^2([T_1,T_2]^2)$ is separable and $\{\phi_k\}_{k\in\N}$ and $\{\tilde{\psi}_{i,k}\}_{k\in\N}$ are both ON-bases for $L^2([T_1,T_2])$, it follows that $\{\phi_k\otimes\tilde{\psi}_{i,l}\}_{k,l}$ is an ON-basis for $L^2([T_1,T_2]^2)$. Therefore
		$$ L^2([T_1,T_2]^2)=\left\{ \sum_{k=1}^\infty \sum_{l=1}^\infty  \lambda_{k,l} \phi_k\otimes\tilde{\psi}_{i,l}: \sum_{k=1}^\infty \sum_{l=1}^\infty  \lambda_{k,l}^2<\infty\right\}.$$ 
Let 
$$S_n^{\beta(i)}(t,\tau)=\sum_{k=1}^n\sum_{l=1}^n \lambda_{k,l}^{\beta(i),1}\psi_{i,l}(\tau)\phi_k(t)+\sum_{k=1}^n\sum_{l=1}^n \lambda_{k,l}^{\beta(i),2}\eta_{i,l}(\tau)\phi_k(t),$$
then $S_n^{\beta(i)}\xrightarrow{L^2([T_1,T_2]^2)}\beta(i)$. 
By the Cauchy-Schwarz inequality,
\begin{align}\label{Qbound}
\int_{[T_1,T_2]} \left( \int_{[T_1,T_2]}(\beta(t,\tau))(i)X^{A'}_\tau(i)d\tau \right)^2dt
&\le
\int_{[T_1,T_2]} \left( \int_{[T_1,T_2]}\left|(\beta(t,\tau))(i)X^{A'}_\tau(i)\right|d\tau \right)^2dt\nonumber
\\
&\le
\int_{[T_1,T_2]}  \left(\int_{[T_1,T_2]}\left|(\beta(t,\tau))(i)\right|^2d\tau \int_{[T_1,T_2]}\left|X^{A'}_\tau(i)\right|^2d\tau\right) dt\nonumber
\\
&=
\int_{[T_1,T_2]}\left|X^{A'}_\tau(i)\right|^2d\tau 
\int_{[T_1,T_2]}  \int_{[T_1,T_2]}\left|(\beta(t,\tau))(i)\right|^2d\tau dt,
\end{align}
which is a.s. finite. Therefore if we let $Q_i(t)=\int_{[T_1,T_2]}(\beta(t,\tau))(i)X^{A'}_\tau(i)d\tau$ and $$S_N^{\int_i}(t)=\sum_{n=1}^N \left\langle \int_{[T_1,T_2]}(\beta(t,.))(i)X^{A'}_.(i)d\tau,\phi_k\right\rangle_{L^2([T_1,T_2])} \phi_k(t)$$ 
then $S_N^{\int_i}\xrightarrow{L^2([T_1,T_2])}Q_i$ a.s.. Since
\begin{align*}
\left\langle \int_{[T_1,T_2]}(\beta(.,\tau))(i)X^{A'}_\tau(i)d\tau,\phi_k\right\rangle_{L^2([T_1,T_2])}
&=
\lim_{n\to\infty}\sum_{k'=1}^n \sum_{l=1}^n \sum_{m=1}^n \lambda_{k,l}^{\beta(i),1}\chi_m^{A'}(i) \int_{[T_1,T_2]}\int_{[T_1,T_2]} \psi_{i,l}(\tau)\psi_{i,m}(\tau)\phi_{k'}(t)\phi_k(t)d\tau dt
\\
&+
\lim_{n\to\infty}\sum_{k'=1}^n \sum_{l=1}^n \sum_{m=1}^n \lambda_{k,l}^{\beta(i),2}\chi_m^{A'}(i) \int_{[T_1,T_2]}\int_{[T_1,T_2]} \eta_{i,l}(\tau)\psi_{i,m}(\tau)\phi_{k'}(t)\phi_k(t)d\tau dt
\\
&=
\lim_{n\to\infty}\sum_{k'=1}^n \sum_{l=1}^n \sum_{m=1}^n \lambda_{k,l}^{\beta(i),1}\chi_m^{A'}(i)\delta_{l,m}\delta_{k',k}
\\
&= \sum_{l=1}^\infty \lambda_{k,l}^{\beta(i),1} \chi_l^{A'}(i)
\end{align*}
we get,
\begin{align*}
S_n^{\int_i}(t)
&=
\sum_{k=1}^n\sum_{l=1}^\infty \lambda_{k,l}^{\beta(i)} \chi_l^{A'}(i)  \phi_k(t),
\end{align*}
where we denote $\lambda_{k,l}^{\beta(i)}=\lambda_{k,l}^{\beta(i),1}$. 
Next, utilizing \eqref{orthscores},
\begin{align}\label{series1}
&\E_{A'}\left[\int_{[T_1,T_2]}\left(\sum_{i=1}^p\sum_{k=1}^n\sum_{l=1}^n \lambda^{\beta(i)}_{k,l}\chi^{A'}_l(i)\phi_k(t)- \sum_{i=1}^p\int_{[T_1,T_2]} (\beta(i))(t,\tau)X^{A'}_\tau(i) d\tau \right)^2dt \right]\nonumber
\\
\le
2^p\sum_{i=1}^p &\E_{A'}\left[\int_{[T_1,T_2]}\left(S_n^{\int_i}(t)-\sum_{k=1}^n\sum_{l=n+1}^\infty \lambda_{k,l}^{\beta(i)} \chi_l^{A'}(i)  \phi_k(t)- Q_i(t) \right)^2dt \right]\nonumber
\\
\le
2^{p+1}\sum_{i=1}^p &\E_{A'}\left[\int_{[T_1,T_2]}\left(S_n^{\int_i}(t)- Q_i(t) \right)^2dt \right]
+
2^{p+1}\sum_{i=1}^p \E_{A'}\left[\int_{[T_1,T_2]}\left(\sum_{k=1}^n\sum_{l=n+1}^\infty \lambda_{k,l}^{\beta(i)} \chi_l^{A'}(i)  \phi_k(t) \right)^2dt \right]
\nonumber
\\
\le
2^{p+1}\sum_{i=1}^p &\E_{A'}\left[\int_{[T_1,T_2]}\left(S_n^{\int_i}(t)- Q_i(t) \right)^2dt \right]
+
2^{p+1}\sum_{i=1}^p\sum_{k=1}^n\sum_{l=n+1}^\infty  \left( \lambda_{k,l}^{\beta(i)} \right)^2\E_{A'}\left[\left( \chi_l^{A'}(i)  \right)^2 \right]
\nonumber
\\
\le
2^{p+1}\sum_{i=1}^p &\E_{A'}\left[\int_{[T_1,T_2]}\left(S_n^{\int_i}(t)- Q_i(t) \right)^2dt \right]
+
2^{p+1}\sum_{i=1}^p\sup_{m\ge n+1}\E_{A'}\left[\left( \chi_m^{A'}(i)  \right)^2 \right]\lVert\beta(i) \rVert_{L^2([T_1,T_2]^2)}
\end{align}
where the second on the right-most side tends to zero since $\sum_{m=1}^\infty\E_{A'}\left[\left( \chi_m^{A'}(i)  \right)^2 \right]<\infty$. Fixing $1\le i\le p$, we now bound the integral terms inside the expectation on the right-most side above using \eqref{Qbound},

\begin{align}
M_n:=\int_{[T_1,T_2]}\left(S_n^{\int_i}(t)- Q_i(t) \right)^2dt\nonumber
&\le
2\int_{[T_1,T_2]} Q_i(t)^2dt
+
2\int_{[T_1,T_2]} \left(S_n^{\int_i}(t)\right)^2dt\nonumber
\\
&\le
2\int_{[T_1,T_2]}\left|X^{A'}_\tau(i)\right|^2d\tau 
\int_{[T_1,T_2]}  \int_{[T_1,T_2]}\left|(\beta(t,\tau))(i)\right|^2d\tau dt
\nonumber
\\
&+
2\int_{[T_1,T_2]}\left(\sum_{k=1}^n\sum_{l=1}^\infty \lambda_{k,l}^{\beta(i)} \chi_l^{A'}(i)  \phi_k(t) \right)^2dt\nonumber
\\
&=2\lVert X^{A'}(i)\rVert_{L^2([T_1,T_2])}^2\lVert \beta\rVert_{L^2([T_1,T_2])^p}^2
+
2\sum_{k=1}^n\left(\sum_{l=1}^\infty \lambda_{k,l}^{\beta(i)} \chi_l^{A'}(i)  \right)^2
\end{align}
and therefore
$$M_n\le 2\lVert X^{A'}(i)\rVert_{L^2([T_1,T_2])}^2\lVert \beta\rVert_{L^2([T_1,T_2])^p}^2
+
2\sum_{k=1}^\infty\left(\sum_{l=1}^\infty \lambda_{k,l}^{\beta(i)} \chi_l^{A'}(i)  \right)^2:=M.$$
Utilizing \eqref{orthscores}, we get
\begin{align*}
\E_{A'}\left[M\right]
&=
2\lVert X^{A'}\rVert_{\mathcal{V}}^2\lVert \beta\rVert_{L^2([T_1,T_2])^p}^2
+
2\sum_{k=1}^\infty\E_{A'}\left[\sum_{l=1}^\infty \left(\lambda_{k,l}^{\beta(i)} \chi_l^{A'}(i)\right)^2  \right]
\\
&=
2\lVert X^{A'}\rVert_{\mathcal{V}}^2\lVert \beta\rVert_{L^2([T_1,T_2])^p}^2
+
2\sum_{k=1}^\infty\sum_{l=1}^\infty \left(\lambda_{k,l}^{\beta(i)} \right)^2\E_{A'}\left[ \left( \chi_l^{A'}\right)^2  \right]
\\
&\le 4\lVert X^{A'}\rVert_{\mathcal{V}}^2\lVert \beta\rVert_{L^2([T_1,T_2])^p}^2<\infty
\end{align*}
where we utilized 
$$\sum_{k=1}^\infty\sum_{l=1}^\infty \left(\lambda_{k,l}^{\beta(i)} \right)^2\E_{A'}\left[ \left( \chi_l^{A'}(i)\right)^2  \right]
\le 
\sum_{l=1}^\infty \E_{A'}\left[ \left( \chi_l^{A'}(i)\right)^2  \right] \sum_{k=1}^\infty\sum_{l=1}^\infty \left(\lambda_{k,l}^{\beta(i)} \right)^2
\le
\lVert X^{A'}\rVert_{\mathcal{V}}^2\lVert \beta\rVert_{L^2([T_1,T_2])^p}^2.$$
Since $\{M_n\}_{n\in\N}$ converges to zero $\P_{A'}$-a.s. and $0\le M_n\le M$ it follows from the dominated convergence theorem that $$\lim_{n\to\infty}\E_{A'}\left[\int_{[T_1,T_2]}\left(S_n^{\int_i}(t)- Q_i(t) \right)^2dt \right]=0
$$ 
and therefore due to \eqref{series1} we get 
\begin{align}\label{series}
\lim_{n\to\infty}\E_{A'}\left[\int_{[T_1,T_2]}\left(\sum_{i=1}^p\sum_{k=1}^n\sum_{l=1}^n \lambda^{\beta(i)}_{k,l}\chi^{A'}_l(i)\phi_k(t)- \sum_{i=1}^p\int_{[T_1,T_2]} (\beta(i))(t,\tau)X^{A'}_\tau(i) d\tau \right)^2dt \right]
=0.
\end{align}
Now we expand the risk function, for $A'\in\mathcal{V}$,  utilizing the fact that $S_n^{Y^{A'}}\xrightarrow{L^2(dt\times d\P_{A'})}Y^{A'}$,
		\begin{align*}
			&R_{A'}(\beta)=\E_{A'}\left[\int_{[T_1,T_2]}\left( Y^{A'}_t-\int_{[T_1,T_2]}\sum_{i=1}^p(\beta(i))(t,\tau)X^{A'}(i)_\tau d\tau\right)^2dt \right]^\frac12
			\\
			&\le
			\lim_{n\to\infty}\E_{A'}\left[\int_{[T_1,T_2]}\left( S_n^{Y^{A'}}(t)-Y^{A'}_t \right)^2dt \right]^\frac12 
			+ \lim_{n\to\infty}\E_{A'}\left[\int_{[T_1,T_2]}\left( S_n^{Y^{A'}}(t) - \sum_{i=1}^p\sum_{k=1}^n\sum_{l=1}^n \lambda^{\beta(i)}_{k,l}\chi^{A'}_l(i)\phi_k(t) \right)^2dt \right]^\frac12  
			\\
			&+ \lim_{n\to\infty}\E_{A'}\left[\int_{[T_1,T_2]}\left(\sum_{i=1}^p\sum_{k=1}^n\sum_{l=1}^n \lambda^{\beta(i)}_{k,l}\chi^{A'}_l(i)\phi_k(t)- \sum_{i=1}^p\int_{[T_1,T_2]} (\beta(i))(t,\tau)X^{A'}_\tau(i) d\tau \right)^2dt \right]^\frac12 
		\end{align*}
and since the first and third term  (due to \eqref{series}) converges to zero on the right-hand side above, we get
		\begin{align}\label{RAlim}
			R_{A'}(\beta)
			&=
			\lim_{n\to\infty}\E_{A'}\left[\int_{[T_1,T_2]}\left(\sum_{k=1}^n Z_k^{A'}\phi_k(t)-\sum_{i=1}^p\sum_{k=1}^n\sum_{l=1}^n \lambda^{\beta(i)}_{k,l}\chi^{A'}_l(i)\phi_k(t)\right)^2dt\right]\nonumber
			\\
			&=
			\lim_{n\to\infty}\E_{A'}\left[\sum_{k=1}^n \left(Z_k^{A'}- \sum_{i=1}^p\sum_{l=1}^n \lambda_{k,l}^{\beta(i)}\chi_l^{A'}(i)\right)^2\right].
		\end{align}
From this point onwards we will fix our choice of ON-basis for $L^2([T_1,T_2])^{p+1}$, let 
\begin{itemize}
\item[]$\phi_{1,m}=\phi_m$ and 
\item[]$\phi_{i,m}=\tilde{\psi}_{i-1,m}$, for $2\le i\le p+1$ and $m\in\N$.
\end{itemize}	
\textbf{Step 2: Reformulate the the integrals appearing in \eqref{ShiftSet} for relevant subspaces}\\
Denote $F_{i,k}(W)=\int_{[T_1,T_2]}W_t\phi_{i,k}(t)dt$, for $W\in L^2([T_1,T_2])$, $1\le i\le p+1$ and for $A'\in\mathcal{V}$
\begin{align*}
F_{1:n}(A')=&\left(\int_{[T_1,T_2]}A'_t(1)\phi_{1,1}(t)dt,\ldots,\int_{[T_1,T_2]}A'(1)\phi_{1,n}(t)dt,\ldots,\right.
\\
&\left.\int_{[T_1,T_2]}A'_t(p+1)\phi_{p+1,1}(t)dt,\ldots,\int_{[T_1,T_2]}A'_t(p+1)\phi_{p+1,n}(t)dt\right).
\end{align*}
		Since $K_{A'(i),A'(j)},\in L^2([T_1,T_2]^2)$, for $1\le i,j \le p+1$ (with the notation $K_{A'(i),A'(i)}=K_{A'(i)}$), if we denote
		\scriptsize
		\begin{equation}\label{SEMAn}
			K_{A'}^n(s,t)
			=\begin{bmatrix}
				\sum_{k=1}^n\sum_{l=1}^nF_{1,k}(A'(1))F_{1,l}(A'(1))\phi_{1,k}(s)\phi_{1,l}(t) & \hdots & \sum_{k=1}^n\sum_{l=1}^nF_{1,k}(A'(1))F_{p+1,l}(A'(p+1))\phi_{1,k}(s)\phi_{p+1,l}(t)\\
				\vdots & \ddots & \vdots\\
				\sum_{k=1}^n\sum_{l=1}^nF_{1,k}(A'(1))F_{p+1,l}(A'(p+1))\phi_{1,k}(s)\phi_{p+1,l}(t) & \hdots & \sum_{k=1}^n\sum_{l=1}^nF_{p+1,k}(A'(p+1))F_{p+1,l}(A'(p+1))\phi_{p+1,k}(s)\phi_{p+1,l}(t)
			\end{bmatrix}, 	
		\end{equation}
		\normalsize
		then all the elements of the matrix $K_{A'}^n$ converges in $L^2([T_1,T_2]^2)$ to the corresponding elements of $K_{A'}$. With a bit of abuse of notation let $K_{i,j}$ and $K^n_{i,j}$ denote the element on row $i$ and column $j$ of the matrix $K_{A'}$ and $K_{A'}^n$ respectively. Fix $n\in\N$ and let $\textbf{g}(s)=\left(g_{1}(s),....,g_{p+1}(s)\right)$ where $g_i\in span\left\{\phi_{i,1},\ldots,\phi_{i,n}\right\}$, for $1\le i\le p+1$. Then
		\begin{align}\label{ApproxKg}
			&\lim_{m\to\infty}\left|\int_{[T_1,T_2]^2}  \textbf{g}(s)\left(K_{A'}(s,t)-K^m_{A'}(s,t)\right)\textbf{g}(t)^Tdsdt\right|\nonumber
			\\
			=&\lim_{m\to\infty}\left|\int_{[T_1,T_2]^2}  \langle \textbf{g}(s), \textbf{g}(t)\left(K_{A'}(s,t)-K^m_{A'}(s,t)\right) \rangle_{\R^{p+1}}dsdt\right|\nonumber
			\\
			\le
			&\lim_{m\to\infty}\int_{[T_1,T_2]^2} \left|\langle \textbf{g}(s), \textbf{g}(t)\left(K_{A'}(s,t)-K^m_{A'}(s,t)\right) \rangle_{\R^{p+1}}\right|dsdt\nonumber
			\\
			&\le
			\lim_{m\to\infty}\int_{[T_1,T_2]^2} \left(\sum_{i=1}^{p+1}g_{i}(s)^2\right)^{\frac12}\left(\sum_{i=1}^{p+1}g_{i}(t)^2\right)^{\frac12} \left\lVert K_{A'}(s,t)-K^m_{A'}(s,t)\right\rVert_{\mathit{l}^2}dsdt\nonumber
			\\
			&\le
			\lim_{m\to\infty}T\left(\int_{[T_1,T_2]^2} \sum_{i=1}^{p+1}g_{i}(s)^2\sum_{i=1}^{p+1}g_{i}(t)^2dsdt\right)^{\frac12} \left(\int_{[T_1,T_2]^2} \left\lVert K_{A'}(s,t)-K^m_{A'}(s,t)\right\rVert_{\mathit{l}^2}^2dsdt\right)^{\frac12}\nonumber
			\\
			&\le
			\lim_{m\to\infty}dT\sum_{i=1}^{p+1}\int g_{i}(t)^2dt
			\left(\sum_{i=1}^{p+1}\sum_{j=1}^{p+1}\int_{[T_1,T_2]^2} \left| K_{i,j}(s,t)-K^m_{i,j}(s,t)\right|^2dsdt\right)^{\frac12}\nonumber
			\\
			&=
			\lim_{m\to\infty}dT\sum_{i=1}^{p+1} \lVert g_{i} \rVert_{L^2([T_1,T_2])}^2\left(\lim_{m\to\infty}\sum_{i=1}^{p+1}\sum_{j=1}^{p+1}\int_{[T_1,T_2]^2} \left| K_{i,j}(s,t)-K^m_{i,j}(s,t)\right|^2dsdt\right)^{\frac12}=0,
		\end{align}
		for some constant $d$ that only depends on $p$ and where we utilized the Cauchy-Schwarz inequality (both on $\R^{p+1}$ as well for the product integral) and that
		$$ 
		\int_{[T_1,T_2]^2} \left(\sum_{i=1}^{p+1}g_{i}(s)^2\sum_{i=1}^{p+1}g_{i}(t)^2\right)^{\frac12} dsdt
		\le
		T\left(\int_{[T_1,T_2]^2} \sum_{i=1}^{p+1}g_{i}(s)^2\sum_{i=1}^{p+1}g_{i}(t)^2dsdt\right)^{\frac12},$$
		which follows from Jensen's inequality applied to the measure $\frac{1}{T^2}dsdt$ on $[T_1,T_2]^2$. Utilizing \eqref{ApproxKg} and the orthogonality,
		\begin{align*}
			&\int_{[T_1,T_2]^2} \textbf{g}(s)\left(K_{A'}(s,t)-K_{A'}^n(s,t)\right)\textbf{g}(t)^Tdsdt
			\\
			&=
			\lim_{m\to\infty} \int_{[T_1,T_2]^2}\left(\sum_{k=1}^nF_{1,k}(g_{1})\phi_{1,k}(s),...,\sum_{k=1}^nF_{p+1,k}(g_{p+1})\phi_{p+1,k}(s)\right)\left((K_{A'}^m(s,t)-K_{A'}^n(s,t))\right)
			\\
			&\times
			\left(\sum_{k=1}^nF_{1,k}(g_{1})\phi_{1,k}(t),...,\sum_{k=1}^nF_{p+1,k}(g_{p+1})\phi_{p+1,k}(t)\right)^Tdsdt
\\
&=	\lim_{m\to\infty}\sum_{i=1}^{p+1}\sum_{j=1}^{p+1}\sum_{k=1}^n\sum_{l=1}^nF_{i,k}(g_i)F_{j,l}(g_j)	\int_{[T_1,T_2]^2}	\left(\phi_{j,l}(s)\phi_{i,k}(t)\right.
\\
&\left.\left( \sum_{k'=1}^m\sum_{l'=1}^mF_{i,k'}(A'(i))F_{j,l'}(A'(j))\phi_{j,l'}(s)\phi_{i,k'}(t)-\sum_{k'=1}^n\sum_{l'=1}^nF_{i,k'}(A'(i))F_{j,l'}(A'(j))\phi_{j,l'}(s)\phi_{i,k'}(t)\right)\right)dsdt
			=0.
		\end{align*}
		Therefore
		\begin{align*}
			&\int_{[T_1,T_2]^2} \left(g_{1}(s),....,g_{p+1}(s)\right)K_{A'}(s,t)\left(g_{1}(t),....,g_{p+1}(t)\right)^Tdsdt
			\\
			&=
			\int_{[T_1,T_2]^2} \left(g_{1}(s),....,g_{p+1}(s)\right)K_{A'}^n(s,t)\left(g_{1}(t),....,g_{p+1}(t)\right)^Tdsdt.
		\end{align*}
\\	
		\textbf{Step 3: Compute a finite dimensional approximation of the target and the covariates and the corresponding error to this approximation}
		\\
		Let $V_n=\mathsf{span}\left(\left\{(\phi_{1,i_1},0,\ldots,0),(0,\phi_{2,i_2},0,\ldots,0),\ldots,(0,\ldots,0,\phi_{p+1,i_k})\right\}_{1\le i_1,...,i_k\le n}\right)$ and $P_n$ denote projection on the space $V_n$. For any $a\in L^2([T_1,T_2])^{p+1}$ we have 
		\begin{align}\label{PnS}
			\lVert P_n\mathcal{S}a-P_n\mathcal{S}P_n a \rVert
			&=
			\lVert P_n\mathcal{S}\rVert\lVert a-P_na\rVert\nonumber
			\\
			&\le 
			\lVert \mathcal{S}\rVert\lVert a-P_na\rVert, 
		\end{align}
		which converges to zero, since $\mathcal{H}^{p+1}\left(\{\phi_{i,n}\}_{n\in\N, 1\le i\le p+1}\right)$ is a basis for $L^2([T_1,T_2])^{p+1}$ and due to the definition of $V_n$ (this just comes down to convergence of the partial sums). Enumerate $\mathcal{H}^{p+1}\left(\{\phi_{i,n}\}_{n\in\N, 1\le i\le p+1}\right)$ so that $e_1=\left(\phi_{1,1},0,\ldots,0\right)$ and $e_{n(p+1)}=\left(0,\ldots,0,\phi_{p+1,n}\right)$. Let $x=\sum_{k=1}^{n(p+1)} a_ke_k\in V_n$, then
		\begin{align*}
			P_n\mathcal{S}x=P_n\mathcal{S}\left(\sum_{k=1}^{n(p+1)} a_ke_k\right)&= \sum_{k=1}^{n(p+1)}a_kP_n\mathcal{S}\left(e_k\right)
			\\
			&= \sum_{m=1}^{n(p+1)}\left(\sum_{k=1}^{n(p+1)}a_k\langle\mathcal{S}\left(e_k\right),e_m \rangle\right) e_m,
		\end{align*}
		so the coordinates of $P_n\mathcal{S}x$ in the basis $\mathcal{H}^{p+1}\left(\{\phi_{i,n}\}_{n\in\N, 1\le i\le p+1}\right)$ are given by 
		$$\left(\left(\sum_{k=1}^{n(p+1)}a_k\langle\mathcal{S}\left(e_k\right),e_1 \rangle\right),\ldots,\left(\sum_{k=1}^{n(p+1)}a_k\langle\mathcal{S}\left(e_k\right),e_{n(p+1)} \rangle\right)\right).$$
		So by defining the following $n(p+1)\times n(p+1)$ matrix 
		\begin{equation}
			B^n=\begin{bmatrix}
				\langle \mathcal{S}\left(\phi_{1,1},0,\ldots,0\right), \left(\phi_{1,1},0,\ldots,0\right) \rangle & \hdots & \langle \mathcal{S}\left(0,\ldots,0,\phi_{p+1,n}\right), \left(\phi_{1,1},0,\ldots,0\right) \rangle \\
				\vdots & \ddots & \vdots  \\
				\langle \mathcal{S}\left(\phi_{1,1},0,\ldots,0\right), \left(0,\ldots,0,\phi_{p+1,n}\right) \rangle & \hdots & \langle \mathcal{S}\left(0,\ldots,0,\phi_{p+1,n}\right), \left(0,\ldots,0,\phi_{p+1,n}\right) \rangle 
			\end{bmatrix},
		\end{equation}
we have that if $x\in V_n$ and we let $H(x)$ denote coordinates of $x$ in the basis $\mathcal{H}^{p+1}\left(\{\phi_{i,n}\}_{n\in\N, 1\le i\le p+1}\right)$, then $H\left(P_n\mathcal{S}x\right)=B^nH(x)$. For $A'\in\mathcal{V}$, let $\mathbf{\chi}_n=\left(\chi_1^{A'}(1),\ldots,\chi_n^{A'}(1),\ldots ,\chi_n^{A'}(p) \right)$ and $\mathbf{Z}^n=\left(Z_1^{A'},\ldots,Z_n^{A'} \right)$. We have $(\mathbf{Z}^n,\mathbf{\chi}^n)=H\left(P_n(Y^{A'},X^{A'})\right)$ and $B^n\left(F_{1:n}(A')+F_{1:n}(\epsilon^{A'})\right)=H(P_n\mathcal{S}P_n(A'+\epsilon^{A'}))$ (since $F_{1:n}(A')+F_{1:n}(\epsilon^{A'})= P_n(A'+\epsilon^{A'})$). Therefore,
		\begin{align*}
			\lVert (\mathbf{Z}^n,\mathbf{\chi}^n)-B^n(F_{1:n}(A')+F_{1:n}(\epsilon^{A'})) \rVert_{\left(\mathit{l^2}\right)^{p+1}}
			&=
			\lVert P_n(Y^{A'},X^{A'})-P_n\mathcal{S}P_n(A'+\epsilon^{A'}) \rVert_{L^2([T_1,T_2])^{p+1}}
			\\
			&=
			\lVert P_n\mathcal{S}(A'+\epsilon^{A'})-P_n\mathcal{S}P_n(A'+\epsilon^{A'}) \rVert_{L^2([T_1,T_2])^{p+1}}
			,
		\end{align*}
		which converges path-wise (per $\omega$) to zero as $n\to\infty$, due to \eqref{PnS}.
		This implies
		\begin{equation}\label{SEMA}
			\begin{bmatrix}
				Z_1^{A'}\\
				\vdots\\
				Z_n^{A'}\\
				\chi_1^{A'}(1)\\
				\vdots\\
				\chi_n^{A'}(1)\\
				\vdots\\
				\chi_n^{A'}(p)\
			\end{bmatrix} = 
			B^n\cdot \left( \begin{bmatrix}
				F_{1,}(A'(1))\\
				\vdots\\
				F_{1,n}(A'(1))\\
				F_{2,1}(A'(2))\\
				\vdots\\
				F_{2,n}(A'(2))\\
				\vdots\\
				F_{p+1,n}(A'(p+1))
			\end{bmatrix}	+
			\begin{bmatrix}
				F_{1,1}(\epsilon^{A'}(1))\\
				\vdots\\
				F_{1,n}(\epsilon^{A'}(1))\\
				F_{2,1}(\epsilon^{A'}(2))\\
				\vdots\\
				F_{2,n}(\epsilon^{A'}(2))\\
				\vdots\\
				F_{p+1,n}(\epsilon^{A'}(p+1))
			\end{bmatrix}	
			\right)+\delta_n(A'),
		\end{equation}
		where 
		\begin{align}\label{delta}
			\lVert \delta_n(A') \rVert_{\left(\mathit{l^2}\right)^{p+1}}
			&= 
			\lVert P_n\mathcal{S}(A'+\epsilon^{A'})-P_n\mathcal{S}P_n(A'+\epsilon^{A'}) \rVert_{L^2([T_1,T_2])^{p+1}}\nonumber
			\\
			&\le
			\lVert \mathcal{S}\rVert\left( \lVert A' - P_nA' \rVert_{L^2([T_1,T_2])^{p+1}}+\lVert \epsilon^{A'}-P_n\epsilon^{A'} \rVert_{L^2([T_1,T_2])^{p+1}} \right)
			\nonumber\\
			&\le
			2\lVert\mathcal{S}\rVert\left( \lVert A' \rVert_{L^2([T_1,T_2])^{p+1}}+\lVert \epsilon^{A'} \rVert_{L^2([T_1,T_2])^{p+1}} \right).
		\end{align}
		Therefore  
		\begin{align*}
			\lVert \delta_n(A') \rVert_{\left(\mathit{l^2}\right)^{p+1}}^2 \le 4\lVert\mathcal{S}\rVert^2\left( \lVert A' \rVert_{L^2([T_1,T_2])^{p+1}}^2 +\lVert \epsilon^{A'} \rVert_{L^2([T_1,T_2])^{p+1}}^2  \right).
		\end{align*}
		By dominated convergence, this implies $\E_{A'}\left[\lVert \delta_n(A') \rVert_{\left(\mathit{l^2}\right)^{p+1}}^2\right]\to 0$.
		\\
		\textbf{Step 4:  Approximate the risk using the finite dimensional approximation from the previous step}
		\\
		Let $\textbf{v}_n=\sum_{k=1}^nB^n_{k,.}-\sum_{k=1}^{n}\sum_{l=1}^{n}\sum_{i=1}^p\lambda_{k,l}^{\beta(i)}B^n_{in+l,.}$. From \eqref{RAlim} we have that	for any $A''\in \mathcal{V}$,
		\begin{align}\label{R_A}
			R_{A''}(\beta) 
			&=
			\lim_{n\to\infty}\sum_{k=1}^n\E_{A''}\left[ \left(Z_k^{A''}- \sum_{l=1}^n\sum_{i=1}^p \lambda_{k,l}^{\beta(i)}\chi_l^{A''}(i)\right)^2\right]\nonumber
			\\
			&=
			\lim_{n\to\infty}\sum_{k=1}^n\E_{A''}\left[ \left(B^n_{k,.}(F_{1:n}(A'')^T+F_{1:n}(\epsilon_{A''})^T)+(\delta_n(A''))(k)
			\right.\right.\nonumber
			\\
			&\left.\left.- \sum_{i=1}^p\sum_{l=1}^n\lambda_{k,l}^{\beta(i)}\left(B_{in+l,.}(F_{1:n}(A'')^T+F_{1:n}(\epsilon_{A''})^T) +(\delta_n(A''))(in+l)\right) \right)^2\right]\nonumber
			\\
			&=
			\lim_{n\to\infty} \left(\textbf{v}_{n}\E_{A''}\left[ \left(F_{1:n}(A'')+F_{1:n}(\epsilon_{A''})\right)^T\left(F_{1:n}(A'')+F_{1:n}(\epsilon_{A''})\right) \right] \textbf{v}_{n}^T\right.\nonumber
			\\
			&\left.+
			\sum_{k=1}^n\E_{A''}\left[ \left((\delta_n(A''))(k)+\sum_{l=1}^n\sum_{i=1}^p\lambda_{k,l}^{\beta(i)}(\delta_n(A''))(in+l)\right)^2 \right]\right.\nonumber
			\\
			&\left.+
			2\sum_{k=1}^n\E_{A''}\left[ \left(B^n_{k,.}(F_{1:n}(A'')^T+F_{1:n}(\epsilon_{A''})^T)- \sum_{i=1}^p\sum_{l=1}^n\lambda_{k,l}^{\beta(i)}\left(B_{in+l,.}(F_{1:n}(A'')^T+F_{1:n}(\epsilon_{A''})^T)\right) \right)\nonumber
			\right.\right.
			\\
			&\left.\left.\cdot\left((\delta_n(A''))(k)+\sum_{l=1}^n\sum_{i=1}^p\lambda_{k,l}^{\beta(i)}(\delta_n(A''))(in+l)\right)\right]\right).
		\end{align}
		The term
		\begin{align*}
			2\sum_{k=1}^n\E_{A''}&\left[ \left(B^n_{k,.}(F_{1:n}(A'')^T+F_{1:n}(\epsilon_{A''})^T)- \sum_{i=1}^p\sum_{l=1}^n\lambda_{k,l}^{\beta(i)}\left(B_{in+l,.}(F(A'')^T+F(\epsilon_{A''})^T)\right) \right)
			\right.
			\\
			&\left.\cdot \left((\delta_n(A''))(k)+\sum_{l=1}^n\sum_{i=1}^p\lambda_{k,l}^{\beta(i)}(\delta_n(A''))(in+l)\right)\right]
		\end{align*}
		is readily dominated by (using the Cauchy Schwarz-inequality, first for the expectation and then for the sum)
		\begin{align*}
			&2\sum_{k=1}^n\E_{A''}\left[ \left(B^n_{k,.}(F_{1:n}(A'')^T+F_{1:n}(\epsilon_{A''})^T)-\sum_{i=1}^p\sum_{l=1}^n\lambda_{k,l}^{\beta(i)}\left(B_{in+l,.}(F_{1:n}(A'')^T+F_{1:n}(\epsilon_{A''})^T)\right) \right)^2 \right]^{\frac12} 
			\\
			&\cdot\E_{A''}\left[ \left((\delta_n(A''))(k)+\sum_{l=1}^n\sum_{i=1}^p\lambda_{k,l}^{\beta(i)}(\delta_n(A''))(in+l)\right)^2 \right]^{\frac12}
			\\
			&\le
			2\left(\sum_{k=1}^n\E_{A''}\left[ \left(B^n_{k,.}(F(A'')^T+F(\epsilon_{A''})^T)- \sum_{i=1}^p\sum_{l=1}^n\lambda_{k,l}^{\beta(i)}\left(B_{in+l,.}(F_{1:n}(A'')^T+F_{1:n}(\epsilon_{A''})^T)\right) \right)^2 \right]\right)^{\frac12} 
			\\
			&\times\left(\sum_{k=1}^n\E_{A''}\left[ \left((\delta_n(A''))(k)+\sum_{l=1}^n\sum_{i=1}^p\lambda_{k,l}^{\beta(i)}(\delta_n(A''))(in+l)\right)^2 \right]\right)^{\frac12}.
		\end{align*}
		This term will converge to zero since, as we will see,  
		\begin{align}\label{deltatermen}
			\sum_{k=1}^n\E_{A''}\left[ \left((\delta_n(A''))(k)+\sum_{l=1}^n\sum_{i=1}^p\lambda_{k,l}^{\beta(i)}(\delta_n(A''))(in+l)\right)^2 \right]
		\end{align}
		converges to zero, while we will show that the term 
		\begin{align}\label{bdd}
			\sum_{k=1}^n\E_{A''}\left[ \left(B^n_{k,.}(F_{1:n}(A'')^T+F_{1:n}(\epsilon_{A''})^T)- \sum_{i=1}^p\sum_{l=1}^n\lambda_{k,l}^{\beta(i)}\left(B_{in+l,.}(F_{1:n}(A'')^T+F_{1:n}(\epsilon_{A''})^T)\right) \right)^2 \right],
		\end{align}
		is bounded. First, we will show that \eqref{deltatermen} converges to zero. Expanding the squares we find
		\begin{align}\label{smalldelta}
			&\sum_{k=1}^n\E_{A''}\left[ \left((\delta_n(A''))(k)+\sum_{l=1}^n\sum_{i=1}^p\lambda_{k,l}^{\beta(i)}(\delta_n(A''))(in+l)\right)^2 \right]\nonumber
			\\
			&=
			\E_{A''}\left[ \lVert (\delta_n(A''))\rVert_{\mathit{l}^2}^2 \right]\nonumber
			+
			2\sum_{k=1}^n\sum_{l=1}^n\sum_{i=1}^p\lambda_{k,l}^{\beta(i)}\E_{A''}\left[\delta_n(k)(\delta_n(A''))(in+l) \right]\nonumber
			\\
			&+\sum_{i_1=1}^p\sum_{i_2=1}^p\sum_{k=1}^n\sum_{l_1=1}^n\sum_{l_2=1}^n\lambda_{k,l_1}^{\beta(i_1)}\lambda_{k,l_2}^{\beta(i_1)}\lambda_{k,l_1}^{\beta(i_2)}\lambda_{k,l_2}^{\beta(i_2)}\E_{A''}\left[(\delta_n(A''))(in+l_1) (\delta_n(A''))(in+l_2) \right],
		\end{align}
		where we already know that the first term on the right-most side will vanish. We then bound the second term in \eqref{smalldelta},
		\begin{align*}
			&\left|\sum_{k=1}^n\sum_{l=1}^n\sum_{i=1}^p\lambda_{k,l}^{\beta(i)}\E_{A''}\left[(\delta_n(A''))(k)(\delta_n(A''))(in+l) \right]\right|
			\\
			\le
			&\left(\sum_{k=1}^n\sum_{l=1}^n\sum_{i=1}^p\left(\lambda_{k,l}^{\beta(i)}\right)^2\right)^{\frac12} \left(\sum_{k=1}^n\sum_{l=1}^n\sum_{i=1}^p \E_{A''}\left[(\delta_n(A''))(k)(\delta_n(A''))(in+l) \right]^2\right)^{\frac12}
			\\
			\le
			&\left(\sum_{i=1}^p\lVert \beta(i) \rVert_{L^2([T_1,T_2]^2)}^2\right)^{\frac12} \left(\sum_{k=1}^n\sum_{l=1}^n\sum_{i=1}^p \E_{A''}\left[(\delta_n(A''))(k)^2\right]\E_{A''}\left[(\delta_n(A''))(in+l)^2\right]\right)^{\frac12}
			\\
			=
			&\left(\sum_{k=1}^n\E_{A''}\left[(\delta_n(A''))(k)^2 \right]\right)^{\frac12}  \lVert \beta \rVert_{L^2([T_1,T_2]^2)^{p+1}}  \left(\sum_{i=1}^p\sum_{l=1}^n\E_{A''}\left[(\delta_n(A''))(in+l)^2 \right]\right)^{\frac12}
			\\
			\le &\E_{A''}\left[\lVert (\delta_n(A'')) \rVert_{\left(\mathit{l^2}\right)^{p+1}}^2\right]\lVert \beta \rVert_{\left(L^2([T_1,T_2]^2)\right)^p},
		\end{align*}
		which converges to zero. For the third term of \eqref{smalldelta}, for fixed $1\le i_1,i_2\le p$
		\begin{align*}
			&\left|\sum_{k=1}^n\sum_{l_1=1}^n\sum_{l_2=1}^n\lambda_{k,l_1}^{\beta(i_1)}\lambda_{k,l_2}^{\beta(i_1)}\lambda_{k,l_1}^{\beta(i_2)}\lambda_{k,l_2}^{\beta(i_2)}\E_{A''}\left[(\delta_n(A''))(in+l_1)(\delta_n(A''))(in+l_2) \right]\right|
			\\
			&\le 
			\sum_{k=1}^n\sum_{l_1=1}^n\sum_{l_2=1}^n\left|\lambda_{k,l_1}^{\beta(i_1)}\right|\left|\lambda_{k,l_2}^{\beta(i_1)}\right|\left|\lambda_{k,l_1}^{\beta(i_2)}\right|\left|\lambda_{k,l_2}^{\beta(i_2)}\right|\E_{A''}\left[\left|(\delta_n(A''))(in+l_1) (\delta_n(A''))(in+l_2) \right|\right]
			\\
			&\le
			\sum_{k=1}^n\sum_{l_1=1}^n\sum_{l_2=1}^n\left|\lambda_{k,l_1}^{\beta(i_1)}\right|\left|\lambda_{k,l_2}^{\beta(i_1)}\right|\left|\lambda_{k,l_1}^{\beta(i_2)}\right|\left|\lambda_{k,l_2}^{\beta(i_2)}\right|\E_{A'}\left[(\delta_n(A''))(in+l_1)^2 \right]^{\frac12}\E_{A''}\left[(\delta_n(A''))(in+l_2)^2 \right]^{\frac12}
			\\
			&=
			\sum_{k=1}^n\left( \sum_{l=1}^n\left|\lambda_{k,l}^{\beta(i_1)}\right|\left|\lambda_{k,l}^{\beta(i_2)}\right|\E_{A''}\left[(\delta_n(A''))(in+l)^2 \right]^{\frac12}\right)^2
			\\
			&\le
			\sum_{k=1}^n\sum_{l=1}^n\left|\lambda_{k,l}^{\beta(i_1)}\right|^2\left|\lambda_{k,l}^{\beta(i_2)}\right|^2\sum_{l=1}^n\E_{A''}\left[ (\delta_n(A''))(in+l)^2\right]
			\\
			&\le
			\E_{A''}\left[\lVert \delta_n(A'') \rVert_{\left(\mathit{l^2}\right)^{p+1}}^2\right]\left(\sum_{k=1}^n\sum_{l=1}^n\left|\lambda_{k,l}^{\beta(i_1)}\right|^4\right)^{\frac12}\left(\sum_{k=1}^n\sum_{l=1}^n\left|\lambda_{k,l}^{\beta(i_2)}\right|^4\right)^{\frac12}
			\\
			&\le
			\E_{A''}\left[\lVert \delta_n(A'') \rVert_{\left(\mathit{l^2}\right)^{p+1}}^2\right]\left(\sum_{k=1}^n\sum_{l=1}^n\left(\lambda_{k,l}^{\beta(i_1)}\right)^2\right)\left(\sum_{k=1}^n\sum_{l=1}^n\left(\lambda_{k,l}^{\beta(i_1)}\right)^2\right)
			\\
			&=
			\E_{A''}\left[\lVert \delta_n(A'') \rVert_{\left(\mathit{l^2}\right)^{p+1}}^2\right]\lVert \beta(i_1) \rVert_{\mathit{L}^2([T_1,T_2]^2)}^2\lVert \beta(i_2) \rVert_{\mathit{L}^2([T_1,T_2]^2)}^2,
		\end{align*}
		which also converges to zero and where we used the Cauchy-Schwarz inequality for expectations as well as sums and the fact that the $\mathit{l}^4$-norm is dominated by the $\mathit{l}^2$-norm. Summing over $i_1$ and $i_2$, this will still converge to zero. We shall now establish that \eqref{bdd} is indeed bounded. We have,
		\begin{align*}
			&\sum_{k=1}^n\E_{A''}\left[ \left(B^n_{k,.}(F_{1:n}(A'')+F_{1:n}(\epsilon_{A''}))- \sum_{i=1}^p\sum_{l=1}^n\lambda_{k,l}^{\beta(i)}\left(B_{in+l,.}(F_{1:n}(A'')^T+F_{1:n}(\epsilon_{A''})^T)\right) \right)^2 \right]
			\\
			&\le \sum_{k=1}^n \E_{A''}\left[ 2\left(B^n_{k,.}(F_{1:n}(A'')^T+F_{1:n}(\epsilon_{A''})^T) \right)^2 \right]
			+
			2\sum_{k=1}^n\E_{A''}\left[ \left(\sum_{i=1}^p\sum_{l=1}^n \lambda_{k,l}^{\beta(i)} \left\langle B_{in+l,.}, \left((F_{1:n}(A'')^T+F_{1:n}(\epsilon_{A''})^T)\right) \right\rangle  \right)^2 \right] 
			\\
			&\le 
			\sum_{k=1}^n4\E_{A''}\left[\left(Z_k^{A''}\right)^2+ (\delta_n(A''))(k)^2 \right]
			+\sum_{k=1}^n\sum_{l=1}^n\sum_{i=1}^p2\left(\lambda_{k,l}^{\beta(i)}\right)^2\E_{A''}\left[ \left\lVert B^n_{in+1:(i+1)n,.}(F_{1:n}(A'')^T+F_{1:n}(\epsilon_{A''})^T) \right\rVert_{\mathit{l}^2}^2 \right]
			\\
			&\le 
			\sum_{k=1}^n4\E_{A''}\left[\left(Z_k^{A''}\right)^2\right]+4\E_{A''}\left[\lVert \delta_n(A'') \rVert_{\left(\mathit{l^2}\right)^{p+1}}^2\right]
			+2\lVert \beta\rVert_{\left(L^2([T_1,T_2]^2)\right)^p}^2\sum_{i=1}^p\E_{A''}\left[ \sum_{l=1}^n\left(\chi^{A''}_l(i)-(\delta_n(A''))(in+l)\right)^2 \right],
		\end{align*}
		where we utilized \eqref{SEMA}. As 
		$$\E_{A''}\left[ \sum_{l=1}^n\left(\chi^{A''}_l(i)-(\delta_n(A''))(in+l)\right)^2 \right] 
		\le 
		2\sum_{l=1}^n\E_{A''}\left[\left(\chi^{A''}_l(i)\right)^2  \right]+2\E_{A''}\left[ \left\lVert \delta_n(A'') \right\rVert_{\left(\mathit{l^2}\right)^{p+1}}^2 \right]<\infty,
		$$
		we then readily see that \eqref{bdd} is indeed bounded. Returning to \eqref{R_A} we now have
		for any $A''\in \mathcal{V}$,
		\begin{align}\label{R_Afinal}
			R_{A''}(\beta) 
			&=
			\lim_{n\to\infty} \textbf{v}_{n}\E_{A''}\left[ \left(F_{1:n}(A'')+F_{1:n}(\epsilon_{A''})\right)\left(F_{1:n}(A'')+F_{1:n}(\epsilon_{A''})\right)^T \right] \textbf{v}_{n}^T.
		\end{align}
		\\
		\textbf{Step 5: Verify that the cross terms between the shifts and the noise vanish:}
		\\
				\begin{claim}\label{claimepsA}
			Let $G_1,G_2\in (L^2([T_1,T_2])^{p+1})^*$  then we have that
			$$\E_{A'}\left[G_1(A')G_2(\epsilon^{A'})\right]=\E_{A''}\left[G_1(A'')G_2(\epsilon_{A''})\right]=0 $$
		\end{claim}
		\begin{proof}[Proof of Claim\ref{claimepsA}]
			We show $\E_{A''}\left[G_1(A'')G_2(\epsilon_{A''})\right]=0$, $\E_{A'}\left[G_1(A')G_2(\epsilon^{A'})\right]=0$ is analogous. Let $\epsilon_{ A'}^n$, $\epsilon_{ A''}^n$, $A'^n$ and $A''^n$ be defined as in the proof of Lemma \ref{PropEnvCont}.  Since
			\begin{align*}
				\left| G_1(A''^n)G_2(\epsilon_{A''}^n)-G_1(A'')G_2(\epsilon_{A''}) \right| 
				&\le 
				\left| G_1(A''^n)G_2(\epsilon_{A''}^n)-G_1(A''^n)G_2(\epsilon_{A''}) \right| 
				+
				\left| G_1(A''^n)G_2(\epsilon_{A''})-G_1(A'')G_2(\epsilon_{A''}) \right| 
				\\
				&\le
				\lVert G_1 \rVert\lVert G_2 \rVert \lVert A''^n \rVert_{L^2([T_1,T_2])^{p+1}}\lVert \epsilon_{A''}^n-\epsilon_{A''} \rVert_{L^2([T_1,T_2])^{p+1}}
				\\
				&+
				\lVert G_1 \rVert\lVert G_2 \rVert \lVert A''^n-A'' \rVert_{L^2([T_1,T_2])^{p+1}}\lVert \epsilon_{A''} \rVert_{L^2([T_1,T_2])^{p+1}}
			\end{align*}
			it follows from the Cauchy-Schwarz inequality that
			\begin{align*}
				\E_{A''}\left[\left| G_1(A''^n)G_2(\epsilon_{A''}^n)-G_1(A'')G_2(\epsilon_{A''}) \right| \right]
				&\le
				\lVert G_1 \rVert\lVert G_2 \rVert\E_{A''}\left[  \lVert A''^n \rVert_{L^2([T_1,T_2])^{p+1}}^2\right]^{\frac12}
				\E_{A''}\left[  \lVert \epsilon_{A''}^n-\epsilon_{A''}\rVert_{L^2([T_1,T_2])^{p+1}}^2\right]^{\frac12}
				\\
				&+
				\lVert G_1 \rVert\lVert G_2 \rVert\E_{A''}\left[  \lVert A''^n -A''\rVert_{L^2([T_1,T_2])^{p+1}}^2\right]^{\frac12}
				\E_{A''}\left[  \lVert \epsilon_{A''}\rVert_{L^2([T_1,T_2])^{p+1}}^2\right]^{\frac12},
			\end{align*}
			which converges to zero. Next,
			\begin{align*}
				\E_{A''}\left[G_1\left(A''^n\right)G_2\left(\epsilon_{A''}^n\right)\right]
				&=
				\sum_{j=1}^{N_n}\sum_{i=1}^{N_n}\P_{A''}\left(\left\{A''\in Q_{j}^n\right\}\cap\left\{\epsilon_{A''}\in Q_{j}^n\right\}\right)G_1\left(w^n_{j}\right)G_2\left(w^n_{j}\right)
				\\
				&=
				\sum_{j=1}^{N_n}\P\left(A''\in Q_{j}^n\right)G_1\left(w^n_{j}\right)\sum_{i=1}^{N_n}\P_{A''}\left(\epsilon_{A''}\in Q_{j}^n\right)G_2\left(w^n_{j}\right)
				\\
				&=
				\E\left[G_1\left(A''^n\right)\right]
				\E_{A''}\left[G_2\left(\epsilon_{A''}^n\right)\right]
				=
				\E\left[G_1\left(A''^n\right)\right]G_2\left(\E_{A''}\left[\epsilon_{A''}^n\right]\right),
			\end{align*}
			Therefore
			\begin{align*}
				\E_{A''}\left[G_1(A'')G_2(\epsilon_{A''})\right]
				&=
				\lim_{n\to\infty}\E_{A''}\left[G_1(A''^n)G_2(\epsilon_{A''}^n)\right]
				\\
				&=
				\lim_{n\to\infty}\E_{A''}\left[G_1(A''^n)\right]\E_{A''}\left[G_2(\epsilon_{A''}^n)\right]
				\\
				&=
				\E_{A''}\left[G_1(A'')\right]\lim_{n\to\infty}G_2\left(\E_{A''}\left[\epsilon_{A''}^n\right]\right)
				\\
				&=
				\E_{A''}\left[G_1(A'')\right]G_2\left(\lim_{n\to\infty}\E_{A''}\left[\epsilon_{A''}^n\right]\right)
				\\
				&=
				\E_{A''}\left[G_1(A'')\right]G_2\left(\E_{A''}\left[\epsilon_{A''}\right]\right)=0,
			\end{align*}
			where the fact that $G_2\left(\E_{A''}\left[\epsilon_{A''}^n\right]\right)=\E_{A''}\left[G_2\left(\epsilon_{A''}^n\right)\right]$ follows from the linearity of $G_2$.
		\end{proof}
We now claim that $F_k\circ S_i\in (L^2([T_1,T_2])^{p+1})^*$ for any $k\in\N, 1\le i\le p+1$. Taking $g=\left(g_1,\ldots,g_{p+1}\right)\in L^2([T_1,T_2])^{p+1}$ we have by the Cauchy-Schwarz inequality
		\begin{align*}
			\left|F_{i,k}\left( S_i g \right)\right|
			&=\left|\int_{[T_1,T_2]}g_i(t)\phi_{i,k}(t)dt \right|
			\\
			&\le
			\lVert g_i \rVert_{L^2([T_1,T_2])}
			\le
			\lVert g \rVert_{L^2([T_1,T_2])^{p+1}}.
		\end{align*}
		Therefore, by Claim \ref{claimepsA} with $G_1=F_{i,k}\circ S_i$ and $G_2=F_{j,l}\circ S_j$, $\E_{A''}\left[F_{i,k}(A''(i))F_{j,l}(\epsilon_{A''}(j))\right]=0$ for all $k,l\in\N$ and $1\le i,j,\le p+1$. This implies
		\begin{align}\label{Ae}
			\E_{A''}\left[F_{1:n}(A'')F_{1:n}(\epsilon_{A''})^T\right]=0, \hspace{2mm} \forall n\in\N.
		\end{align}
		\\
		\textbf{Step 6: Isolate the pure observational term}	
		\begin{claim}\label{claimeps}
			Let $G_1,G_2\in (L^2([T_1,T_2])^{p+1})^*$  (i.e. a bounded linear functional on $L^2([T_1,T_2])^{p+1}$). Then
			$$\E_{A''}\left[G_1(\epsilon_{ A''})G_2(\epsilon_{ A''})\right]=\E\left[ G_1(\epsilon)G_2(\epsilon)\right]=\E_{A'}\left[ G_1(\epsilon^{A'})G_2(\epsilon^{A'})\right]. $$
		\end{claim}
		
		\begin{proof}[Proof of Claim \ref{claimeps}]
			We will show the first equality, the second one is analogous. Let $\epsilon_{A''}^n$ be as above and let $\epsilon^n=\sum_{j=1}^{N_n}w_j^n1_{\epsilon\in Q_j^n}$, where $W_{e}^n=\sum_{j=1}^\infty w_j^n1_{\epsilon\in Q_j^n}$ and we now choose $N_n$ as in the proof of Lemma \ref{PropEnvCont} but also large enough to assure that
			$$\E\left[\lVert \epsilon^n-W_{e}^n \rVert_{L^2([T_1,T_2])^{p+1}}^2\right] =\sum_{j=N_n+1}^{\infty}\n w_j^n\n^2\P\left(\epsilon\in Q_j^n\right)<\frac{1}{2n}.$$
			This will also imply $\E\left[\lVert \epsilon^n-\epsilon \rVert_{L^2([T_1,T_2])^{p+1}}^2\right]\to 0$. We note that
			\begin{align}\label{G}
				\E_{A''}\left[ G_1\left(\epsilon_{A''}^n\right)G_2\left(\epsilon_{A''}^n\right)\right]
				&=
				\sum_{i=1}^{N_n}\P_{A''}\left(\epsilon_{A''}\in Q_{i}^n\right) G_1\left(w^n_{i}\right)G_2\left(w^n_{i}\right)\nonumber
				\\
				&=
				\sum_{i=1}^{N_n}\P\left(\epsilon\in Q_{i}^n\right)\prod_{j=1}^2 G_1\left(w^n_{i}\right)G_2\left(w^n_{i}\right)
				=
				\E\left[G_1\left(\epsilon^n\right)G_2\left(\epsilon^n\right)\right].
			\end{align}
			Since
			\begin{align}\label{AG1}
				&\left| \E_{A''}\left[G_1\left(\epsilon_{A''}^n\right)G_2\left(\epsilon_{A''}^n\right)\right]- \E_{A''}\left[G_1\left(\epsilon_{A''}\right)G_2\left(\epsilon_{A''}\right)\right]\right|\nonumber
				\\
				&\le
				\E_{A''}\left[\left|G_2\left(\epsilon_{A''}^n\right)G_1\left(\epsilon_{A''}^n-\epsilon_{A''}\right)\right|\right]
				+ 
				\E_{A''}\left[\left|G_1\left(\epsilon_{A''}\right)G_2\left(\epsilon_{A''}^n-\epsilon_{A''}\right)\right|\right]\nonumber
				\\
				&\le
				\lVert G_2\rVert\lVert G_1\rVert \E_{A''}\left[\lVert \epsilon_{A''}^n\rVert_{\left(L^2([T_1,T_2])\right)^{p+1}}\lVert \epsilon_{A''}^n-\epsilon_{A''}\rVert_{\left(L^2([T_1,T_2])\right)^{p+1}}\right]\nonumber
				\\
				&+
				\lVert G_2\rVert\lVert G_1\rVert \E_{A''}\left[\lVert \epsilon_{A''}\rVert_{\left(L^2([T_1,T_2])\right)^{p+1}}\lVert \epsilon_{A''}^n-\epsilon_{A''}\rVert_{\left(L^2([T_1,T_2])\right)^{p+1}}\right]\nonumber
				\\
				&\le
				\lVert G_2\rVert\lVert G_1\rVert\left(\E_{A''}\left[\lVert \epsilon_{A''}\rVert_{\left(L^2([T_1,T_2])\right)^{p+1}}^2\right]^{\frac12}+\E_{A''}\left[\lVert \epsilon_{A''}^n\rVert_{\left(L^2([T_1,T_2])\right)^{p+1}}^2\right]^{\frac12}\right)\E_{A''}\left[\lVert \epsilon_{A''}^n-\epsilon_{A''}\rVert_{\left(L^2([T_1,T_2])\right)^{p+1}}^2\right]^{\frac12}
			\end{align}
			and analogously
			\begin{align}\label{AG2}
				&\left| \E\left[G_1\left(\epsilon^n\right)G_2\left(\epsilon^n\right)\right]- \E\left[G_1\left(\epsilon\right)G_2\left(\epsilon\right)\right]\right|\nonumber
				\\
				&\le
				\lVert G_2\rVert\lVert G_1\rVert\left(\E\left[\lVert \epsilon\rVert_{\left(L^2([T_1,T_2])\right)^{p+1}}^2\right]^{\frac12}
				+
				\E\left[\lVert \epsilon^n\rVert_{\left(L^2([T_1,T_2])\right)^{p+1}}^2\right]^{\frac12}\right)\E\left[\lVert \epsilon^n-\epsilon\rVert_{\left(L^2([T_1,T_2])\right)^{p+1}}^2\right]^{\frac12}.
			\end{align}
			Next,
			\begin{align*}
				\left| \E_{A''}\left[G_1\left(\epsilon_{A''}\right)G_2\left(\epsilon_{A''}\right)\right]- \E\left[G_1\left(\epsilon\right)G_2\left(\epsilon\right)\right]\right|
				&\le
				\left| \E_{A''}\left[G_1\left(\epsilon_{A''}^n\right)G_2\left(\epsilon_{A''}^n\right)\right]- \E_{A''}\left[G_1\left(\epsilon_{A''}\right)G_2\left(\epsilon_{A''}\right)\right]\right|
				\\
				&+ 
				\left| \E\left[G_1\left(\epsilon^n\right)G_2\left(\epsilon^n\right)\right]- \E\left[G_1\left(\epsilon\right)G_2\left(\epsilon\right)\right]\right|
				\\
				&+
				\left| \E_{A''}\left[G_1\left(\epsilon_{A''}^n\right)G_2\left(\epsilon_{A''}^n\right)\right]- \E\left[G_1\left(\epsilon^n\right)G_2\left(\epsilon^n\right)\right]\right|,
			\end{align*}
			which converges to zero due to \eqref{G}, \eqref{AG1} and \eqref{AG2}.
		\end{proof}
		Fix $1\le i,j\le p+1$ and $k,l\in\N$. Let $G_1(f)=F_{i,k}(\pi_i f)$ and $G_2(f)=F_{j,l}(\pi_j f)$ for $f\in\left(L^2([T_1,T_2])\right)^{p+1}$. Then $| G_1(f)|\le \lVert f\rVert_{\left(L^2([T_1,T_2])\right)^{p+1}}$ implying $\lVert G_1\rVert\le 1$ and similarly $\lVert G_2\rVert\le 1$. By Claim \ref{claimeps}, it follows that 
		\begin{align*}
			\E_{A''}\left[F_{i,k}(\epsilon_{A''}(i))F_{j,l}(\epsilon_{A''}(j))\right]
			=
			\E_O\left[F_{i,k}(\epsilon_{O}(i))F_{j,l}(\epsilon_{O}(j))\right],
		\end{align*}
		which in turn implies that
		$$\E_{O}\left[ F_{1:n}(\epsilon_O)F_{1:n}(\epsilon_O)^T \right]=\E_{A''}\left[ F_{1:n}(\epsilon_A'')F_{1:n}(\epsilon_A'')^T \right]. $$
		\\
		\textbf{Step 7: Optimize over the shifts}
		\\
				Recall the definition of $\textbf{v}_n$ from step 4. We utilize \eqref{Ae} when we now return to \eqref{R_Afinal},
		\begin{align*}
			R_{A''}(\beta)
			&=
			\textbf{v}_{n}\E_{A''}\left[ \left(F_{1:n}(A'') +F_{1:n}(\epsilon_{A''}) \right)\left(F_{1:n}(A'') +F_{1:n}(\epsilon_{A''})\right)^T \right] \textbf{v}_{n}^T
			\\
			&=
			\lim_{n\to\infty}\textbf{v}_{n}\E\left[ F_{1:n}(A'')F_{1:n}(A'')^T \right] \textbf{v}_{n}^T+\lim_{n\to\infty}\textbf{v}_{n}\E_{A''}\left[ F_{1:n}(\epsilon_{A''})F_{1:n}(\epsilon_{A''})^T \right] \textbf{v}_{n}^T
			\\
			&=\lim_{n\to\infty}\textbf{v}_{n}\E\left[ F_{1:n}(A'')F_{1:n}(A'')^T \right] \textbf{v}_{n}^T+\lim_{n\to\infty}\textbf{v}_{n}\E_{O}\left[ F_{1:n}(\epsilon_O)F_{1:n}(\epsilon_O)^T \right] \textbf{v}_{n}^T
		\end{align*}
		and simlarly we have
		\begin{align}\label{RAformula}
			R_{A}(\beta)
			=\lim_{n\to\infty}\textbf{v}_{n}\E\left[ F_{1:n}(A)F_{1:n}(A)^T \right] \textbf{v}_{n}^T+\lim_{n\to\infty}\textbf{v}_{n}\E_{O}\left[ F_{1:n}(\epsilon_O)F_{1:n}(\epsilon_O)^T \right] \textbf{v}_{n}^T.
		\end{align}
		Take a sequence $\{A_n\}_{n\in\N}\subset C^\gamma_{\mathcal{A}}(A)$ such that $\lim_{n\to\infty}R_{A_n}(\beta)=\sup_{A'\in C^\gamma_{\mathcal{A}}(A)}R_{A'}(\beta)$. Fix any $\Delta>0$. Let $\tilde{A}_\Delta\in C^\gamma_{\mathcal{A}}(A)$ be such that $\lVert \tilde{A}_\Delta-\sqrt{\gamma}A\rVert_{\mathcal{V}}<\eta$ where $\eta$ is chosen such that $\left|R_{\tilde{A}_\Delta}(\beta)-R_{\sqrt{\gamma}A}(\beta)\right|<\Delta$, which is possible due to Lemma \ref{PropEnvCont} and the fact that $\sqrt{\gamma}A\in\bar{\mathcal{A}}$. Define the sets
		$$C_m=\{\tilde{A}_\Delta\}\cup\left(\bigcup_{k=1}^m\{A_m\}\right), m\in\N.$$
		Fix $m\in\N$. Since there are only finitely many elements in $C_m$, we have that $\lim_{n\to\infty}\textbf{v}_{n}\E\left[ F_{1:n}(A'')F_{1:n}(A'')^T \right]\textbf{v}_{n}^T$ uniformly over all $A''\in C_m$, so we may take $N\in\N$ such that 
		$$\left|\lim_{n\to\infty}\textbf{v}_{n}\E\left[ F_{1:n}(A'')F_{1:n}(A'')^T \right]\textbf{v}_{n}^T -\textbf{v}_{N}\E\left[ F_{1:N}(A'')F_{1:N}(A'')^T \right]\textbf{v}_{N}^T\right|<\Delta, \forall A''\in C_m$$
		and
		$$\left|\lim_{n\to\infty}\textbf{v}_{n}\E\left[ F_{1:n}(A)F_{1:n}(A)^T \right]\textbf{v}_{n}^T -\textbf{v}_{N}\E\left[ F_{1:N}(A)F_{1:N}(A)^T \right]\textbf{v}_{N}^T\right|<\Delta.$$
		Let $g_i(s)=\sum_{k=1}^N\textbf{v}_N((i-1)N+k)\phi_{i,k}(s)$ for $1\le i\le p+1$. Then clearly $g_i\in L^2([T_1,T_2])$. We will now utilize our results from step 2. Borrowing similar notation from step 2, note that
		\small
		\begin{align}\label{truncen}
			&\int_{[T_1,T_2]^2} \left( g_1(s),\ldots,g_{p+1}(s) \right) K^N_{A''}(s,t) \left(g_1(t),\ldots,g_{p+1}(t) \right)^Tdsdt\nonumber
			\\
			&=
			\sum_{i=1}^{p+1}\sum_{j=1}^{p+1}\int_{[T_1,T_2]^2}g_i(s)g_j(t)K^N_{(i,j)}(s,t)dsdt\nonumber
			\\
			&=
			\sum_{i=1}^{p+1}\sum_{j=1}^{p+1}\int_{[T_1,T_2]^2}g_i(s)g_j(t)K^N_{(j,i)}(s,t)dsdt\nonumber
			\\
			&=
			\sum_{i=1}^{p+1}\sum_{j=1}^{p+1}\sum_{m=1}^{N}\sum_{v=1}^{N}\int_{[T_1,T_2]^2}\textbf{v}_N((i-1)N+m)\phi_{i,m}(s)\E\left[\sum_{k=1}^{N}\sum_{l=1}^{N}F_{i,k}(A''(i))F_{j,l}(A''(j))\phi_{i,k}(s)\phi_{j,l}(t)\right]\textbf{v}_N((j-1)N+v)\phi_{j,v}(s)dsdt\nonumber
			\\
			&=
			\sum_{i=1}^{p+1}\sum_{j=1}^{p+1}\sum_{k=1}^{N}\sum_{l=1}^{N}\textbf{v}_N((i-1)N+k)\E\left[F_{i,k}(A''(i))F_{j,l}(A''(j))\right]\textbf{v}_N((j-1)N+l)\nonumber
			\\
			&=
			\sum_{k=1}^{N(p+1)}\sum_{l=1}^{N(p+1)}\textbf{v}_N(k)\E\left[ F_{1:N}(A'')F_{1:N}(A'')^T \right]_{k,l}\textbf{v}_N(l)
			=\textbf{v}_{N}\E\left[ F_{1:N}(A'')F_{1:N}(A'')^T \right]\textbf{v}_{N}^T,
		\end{align}
		\normalsize
		where we utilized the symmetry of $K^N_{A''}$ to swap $i$ and $j$ in the second equality, as well as the fact that 
		\begin{align*}
			\E\left[ F_{1:N}(A'')F_{1:N}(A'')^T \right]_{k,l}
			&=\E\left[ \left(F_{1:N}(A'')F_{1:N}(A'')^T\right)_{k,l} \right]
			\\
			&=\E\left[\left(F_{1:N}(A'')F_{1:N}(A'')^T\right)_{(i-1)N+k',(j-1)N+l'}\right]
			\\
			&=
			\E\left[F_{i,k'}(A''(i))F_{j,l'}(A''(j))\right],
		\end{align*}
		if $k=(i-1)N+k'$ and $l=(j-1)N+l'$, for $1\le i,j\le p+1$ and $1\le k',l'\le N$. By the definition of $C^\gamma_{\mathcal{A}}(A)$, \eqref{truncen} and the fact that $A''\in C_m\subset  C^\gamma_{\mathcal{A}}(A)$  we have,
\begin{align*}
			\textbf{v}_{N}\E\left[ F_{1:N}(A'')F_{1:N}(A'')^T \right]\textbf{v}_{N}^T
			&=
			\int_{[T_1,T_2]} \int_{[T_1,T_2]}\left( g_1(s),\ldots,g_{p+1}(s) \right) K^N_{A''}(s,t) \left(g_1(t),\ldots,g_{p+1}(t) \right)^Tdsdt
			\\
			&=
			\int_{[T_1,T_2]^2} \left( g_1(s),\ldots,g_{p+1}(s) \right) K_{A''}(s,t) \left(g_1(t),\ldots,g_{p+1}(t) \right)^Tdsdt
			\\
			&\le 
			\gamma\int_{[T_1,T_2]^2} \left( g_1(s),\ldots,g_{p+1}(s) \right) K_{A}(s,t) \left(g_1(t),\ldots,g_{p+1}(t) \right)^Tdsdt
			\\
			&=
			\gamma\int_{[T_1,T_2]^2} \left( g_1(s),\ldots,g_{p+1}(s) \right) K^N_{A}(s,t) \left(g_1(t),\ldots,g_{p+1}(t) \right)^Tdsdt
			\\
			&=
			\gamma\textbf{v}_{N}\E\left[ F_{1:N}(A)F_{1:N}(A)^T \right]\textbf{v}_{N}^T
			\\
			&\le 
			\gamma\lim_{n\to\infty}\textbf{v}_{n}\E\left[ F_{1:n}(A)F_{1:n}(A)^T \right]\textbf{v}_{n}^T+\Delta
		\end{align*}
		Since
		\begin{align*}
			\lim_{n\to\infty}\gamma \textbf{v}_{n}\E_{A}\left[ \left(F_{1:n}(A) +F_{1:n}(\epsilon_A) \right)\left(F_{1:n}(A) +F_{1:n}(\epsilon_A)\right)^T \right] \textbf{v}_{n}^T
			=\gamma R_{A}(\beta)
		\end{align*}
		and analogously 
		\begin{align*}
			\lim_{n\to\infty} \textbf{v}_{n}\E_O\left[ F(\epsilon_O)F(\epsilon_O)^T \right] \textbf{v}_{n}^T
			=R_{O}(\beta)
		\end{align*}
		it therefore follows that
		\begin{align*}
			R_{A''}(\beta)
			\le
			\gamma R_{A}(\beta) + (1-\gamma)R_O(\beta)+\Delta
			=\frac12 R_+(\beta)+\left(\gamma -\frac12\right) R_\Delta(\beta)+\Delta.
		\end{align*}
		Since $\tilde{A}_\Delta\in C_m$, for every $m\in\N$, we have
		\begin{align*}
			\max_{A''\in C_m}R_{A''}(\beta)&\ge  R_{\tilde{A}_\Delta}(\beta)
			\\
			&\ge R_{\sqrt{\gamma}A}(\beta)-\Delta
			=\frac12 R_+(\beta)+\left(\gamma -\frac12\right) R_\Delta(\beta)-\Delta.
		\end{align*}
		Hence 
		\begin{align*}
			\left|\max_{A''\in C_m}R_{A''}(\beta)-\left(\frac12 R_+(\beta)+\left(\gamma -\frac12\right) R_\Delta(\beta)\right)\right|<\Delta.
		\end{align*}
		Since $\lim_{n\to\infty}R_{A_n}(\beta)=\sup_{A'\in C^\gamma_{\mathcal{A}}(A)}R_{A'}(\beta)$, and $R_{A_n}(\beta)\le \sup_{A'\in C^\gamma_{\mathcal{A}}(A)}R_{A'}(\beta)$ (since $A_n\in C^\gamma_{\mathcal{A}}(A)$) it follows that $\lim_{m\to\infty}\max_{A''\in C_m}R_{A''}(\beta) =\sup_{A'\in C^\gamma_{\mathcal{A}}(A)}R_{A'}(\beta)$ and therefore there exists $M\in\N$ such that if $m\ge M$, $\left| \max_{A''\in C_m}R_{A''}(\beta) -\sup_{A'\in C^\gamma_{\mathcal{A}}(A)}R_{A'}(\beta)\right|<\Delta$.
		Therefore, for $m\ge M$
		\begin{align*}
			\left| \sup_{A'\in C^\gamma_{\mathcal{A}}(A)}R_{A'}(\beta)-\left(\frac12 R_+(\beta)+\left(\gamma -\frac12\right)R_\Delta(\beta)\right)\right|
			&\le 
			\left| \sup_{A'\in  C^\gamma_{\mathcal{A}}(A)}R_{A'}(\beta)-\max_{A''\in C_m}R_{A''}(\beta)\right|
			\\
			&+
			\left|\max_{A''\in C_m}R_{A''}(\beta)-\left(\frac12 R_+(\beta)+\left(\gamma -\frac12\right) R_\Delta(\beta)\right)\right|< 2\Delta
		\end{align*}
		and by letting $\Delta\to 0$ we get
		$$ \sup_{A'\in C^\gamma_{\mathcal{A}}(A)} R_{A'}(\beta)=\frac12 R_+(\beta)+\left(\gamma -\frac12\right) R_\Delta(\beta),$$
		as was to be shown.
	\end{proof}
	
	\subsection{Proof of Theorem \ref{MinimizerThm}}
	\begin{proof}
In the proof of Theorem \ref{WR1} we saw that, $\sup_{A'\in C^\gamma_{\mathcal{A}}(A)} R_{A'}(\beta)=R_{\sqrt{\gamma}A}(\beta)$. By the Cauchy-Schwarz inequality
\begin{align*}
\int_{[T_1,T_2]^2}K_{X^{\sqrt{\gamma}A}(i)X^{\sqrt{\gamma}A}(j)}(s,t)^2dsdt
&=
\int_{[T_1,T_2]^2}\E_{\sqrt{\gamma}A}\left[X^{\sqrt{\gamma}A}_s(i)X^{\sqrt{\gamma}A}_t(j)\right]^2dsdt
\\
&\le
\int_{[T_1,T_2]^2}\E_{\sqrt{\gamma}A}\left[\left(X_s^{\sqrt{\gamma}A}(i)\right)^2\right]\E_{\sqrt{\gamma}A}\left[\left(X_t^{\sqrt{\gamma}A}(j)\right)^2\right]dsdt
\\
&=
\int_{[T_1,T_2]}\E_{\sqrt{\gamma}A}\left[\left(X_t^{\sqrt{\gamma}A}(i)\right)^2\right]dt
\int_{[T_1,T_2]}\E_{\sqrt{\gamma}A}\left[\left(X_t^{\sqrt{\gamma}A}(j)\right)^2\right]ds,
\end{align*}
which is finite by the assumption on $A$ and $\epsilon$. Let $K^i$ denote the i:th row of $K_{X^{\sqrt{\gamma}A}}$, $K^i_t=K^i(s,t)$ and let $\{e_n\}_{n\in\N}$ be an ON-basis for $L^2([T_1,T_2])^p$. Since 
$$K_{X^{\sqrt{\gamma}A}(i)X^{\sqrt{\gamma}A}(j)}(s,t)^2\le \E_{\sqrt{\gamma}A}\left[\left(X_s^{\sqrt{\gamma}A}(i)\right)^2\right]\E_{\sqrt{\gamma}A}\left[\left(X_t^{\sqrt{\gamma}A}(j)\right)^2\right]$$
and $\E_{\sqrt{\gamma}A}\left[\left(X_t^{\sqrt{\gamma}A}(j)\right)^2\right]<\infty$ a.e. $t$ (since $\int_{[T_1,T_2]}\E_{\sqrt{\gamma}A}\left[\left(X_t^{\sqrt{\gamma}A}(j)\right)^2\right]dt<\infty$),
 every component of $K^i_t$ is a well-defined element of $L^2([T_1,T_2])$ for a.e. $t$. Therefore $K^i_t=\sum_{n=1}^\infty \langle K_t^i,e_n\rangle_{L^2([T_1,T_2])^p}e_n$ and $\lVert K^i_t \rVert_{L^2([T_1,T_2])^p}^2=\sum_{n=1}^\infty \langle K_t^i,e_n\rangle_{L^2([T_1,T_2])^p}^2$. Next,
\begin{align*}
(\mathcal{K}e_n)(t)=\int_{[T_1,T_2]} K^i(s,t)e_n(s)ds=\left(\langle K_t,e_n\rangle_{L^2([T_1,T_2])^p},\ldots,\langle K^p_t,e_n\rangle_{L^2([T_1,T_2])^p}\right).
\end{align*}
By monotone convergence,
\begin{align*}
\sum_{n=1}^\infty \lVert \mathcal{K}e_n \rVert_{L^2([T_1,T_2])^p}^2 =\sum_{n=1}^\infty\sum_{i=1}^p\int_{[T_1,T_2]} \langle K_t^i,e_n\rangle_{L^2([T_1,T_2])^p}^2dt
=
\sum_{i=1}^p\int_{[T_1,T_2]} \sum_{n=1}^\infty\langle K_t^i,e_n\rangle_{L^2([T_1,T_2])^p}^2dt.
\end{align*}
This implies
\begin{align*}
\sum_{n=1}^\infty \lVert \mathcal{K}e_n \rVert_{L^2([T_1,T_2])^p}^2 
&=
\sum_{i=1}^p\int_{[T_1,T_2]} \lVert K^i_t \rVert_{L^2([T_1,T_2])^p}^2dt
\\
&=
\sum_{i=1}^p\sum_{j=1}^p\int_{[T_1,T_2]^2}K_{X^{\sqrt{\gamma}A}(i)X^{\sqrt{\gamma}A}(j)}(s,t)^2dsdt<\infty.
\end{align*}
It follows that $\mathcal{K}$ is a Hilbert Schmidt operator and therefore compact on $L^2([T_1,T_2])^p$. Since $K_{X^{\sqrt{\gamma}A}}(s,t)=K_{X^{\sqrt{\gamma}A}}(t,s)$, this operator is also self-adjoint operator. By the Hilbert-Schmidt theorem, $\mathcal{K}$ has an eigendecomposition, $\mathcal{K}=\sum_{k=1}^\infty\alpha_k\langle \psi_k,.\rangle\psi_k$, where its eigenfunctions, $\{\psi_k\}_k$, are orthonormal in $L^2([T_1,T_2])^p$. We then have that
\begin{align*}
K_{X^{\sqrt{\gamma}A}}(s,t)
&=
\E_{\sqrt{\gamma}A}\left[\left(X^{\sqrt{\gamma}A}_s\right)^TX^{\sqrt{\gamma}A}_t\right]\\
&=
\E_{\sqrt{\gamma}A}\left[\mathcal{S}_{2:p}\left(\sqrt{\gamma}A+\epsilon_{\sqrt{\gamma}A}\right)_s^T \mathcal{S}_{2:p}\left(\sqrt{\gamma}A+\epsilon_{\sqrt{\gamma}A}\right)_t\right]\\
&=
\gamma\E\left[\mathcal{S}_{2:p}\left(A\right)_s^T \mathcal{S}_{2:p}\left(A\right)_t\right]
+
\sqrt{\gamma}\left(\E_{\sqrt{\gamma}A}\left[\mathcal{S}_{2:p}\left(A\right)_s^T \mathcal{S}_{2:p}\left(\epsilon_{\sqrt{\gamma} A}\right)_t\right]
+
\E_{\sqrt{\gamma}A}\left[\mathcal{S}_{2:p}\left(\epsilon_{\sqrt{\gamma} A}\right)_s^T \mathcal{S}_{2:p}\left(A\right)_t\right] \right)
\\
&+
\E_{\sqrt{\gamma}A}\left[\mathcal{S}_{2:p}\left(\epsilon_{\sqrt{\gamma} A}\right)_s^T \mathcal{S}_{2:p}\left(\epsilon_{\sqrt{\gamma} A}\right)_t\right]
\end{align*}
Take some arbitrary basis of $L^2([T_1,T_2]^2)$, say the basis provided in the theorem statement, $\{\phi_n\}_{n\in\N}$, and let $V_n=\mathsf{span}\left(\{\phi_{m_1}\otimes\phi_{m_2}\}_{1\le m_1,m_2\le n}\right)$ and $f\in V_n$ so that $f=\sum_{k=1}^n\sum_{l=1}^n a_{k,l}\phi_k\otimes \phi_l$ for some $\{a_{k,l}\}_{1\le k,l\le n}$. We have for $2\le k,l\le p+1$ by Claim \ref{claimepsA},
\begin{align*}
&\int_{[T_1,T_2]^2}\E_{\sqrt{\gamma}A}\left[\mathcal{S}_{k}\left(\sqrt{\gamma}A\right)_s \mathcal{S}_{l}\left(\epsilon_{\sqrt{\gamma}A}\right)_t\right]f(s,t)
\\
=
&\sum_{k=1}^n\sum_{l=1}^n a_{k,l}\E_{\sqrt{\gamma}A}\left[\int_{[T_1,T_2]}\mathcal{S}_k\left(\sqrt{\gamma}A\right)_s\phi_k(s)ds \int_{[T_1,T_2]}\mathcal{S}_{l}\left(\epsilon_{\sqrt{\gamma}A}\right)_t\phi_l(t)dt\right] 
=0.
\end{align*}
Letting $P_n$ denote projection on the space $V_n$, so that $\lVert P_nf-f\rVert_{L^2([T_1,T_2]^2)}\to 0$ it then follows that
\begin{align*}
&\left|\int_{[T_1,T_2]^2}\E_{\sqrt{\gamma}A}\left[\mathcal{S}_{k}\left(\sqrt{\gamma}A\right)_s \mathcal{S}_{l}\left(\epsilon_{\sqrt{\gamma}A}\right)_t\right]f(s,t) -\int_{[T_1,T_2]^2}\E_{\sqrt{\gamma}A}\left[\mathcal{S}_{k}\left(\sqrt{\gamma}A\right)_s \mathcal{S}_{l}\left(\epsilon_{\sqrt{\gamma}A}\right)_t\right](P_nf)(s,t)\right|
\\
&\le
\left(\int_{[T_1,T_2]^2}\E_{\sqrt{\gamma}A}\left[\mathcal{S}_{k}\left(\sqrt{\gamma}A\right)_s \mathcal{S}_{l}\left(\epsilon_{\sqrt{\gamma}A}\right)_t\right]^2dsdt\right)^{\frac12}\lVert P_nf-f\rVert_{L^2([T_1,T_2]^2)}
\\
&\le
\left(\int_{[T_1,T_2]^2}\E_{\sqrt{\gamma}A}\left[\mathcal{S}_{k}\left(\sqrt{\gamma}A\right)_s^2\right]\E_{\sqrt{\gamma}A}\left[\mathcal{S}_{l}\left(\epsilon_{\sqrt{\gamma}A}\right)_t^2\right]dsdt\right)^{\frac12}\lVert P_nf-f\rVert_{L^2([T_1,T_2]^2)}
\\
&\le
\gamma\lVert \mathcal{S}\left(A\right)\rVert_{\mathcal{V}} \lVert \mathcal{S}\left(\epsilon\right)\rVert_{\mathcal{V}} \lVert P_nf-f\rVert_{L^2([T_1,T_2]^2)},
\end{align*}
which converges to zero. Since $g(s,t)=\E_{\sqrt{\gamma}A}\left[\mathcal{S}_{k}\left(\sqrt{\gamma}A\right)_s \mathcal{S}_{l}\left(\epsilon_{\sqrt{\gamma}A}\right)_t\right]\in L^2([T_1,T_2]^2)$ and $\lVert gf\rVert_{L^2([T_1,T_2]^2)}=0$, $\forall f\in L^2([T_1,T_2]^2)$ this implies $
g=0$ a.e. in $[T_1,T_2]^2$, implying $\E_{\sqrt{\gamma}A}\left[\mathcal{S}_{2:p}\left(\sqrt{\gamma}A\right)_s^T \mathcal{S}_{2:p}\left(\epsilon_{\sqrt{\gamma}A}\right)_t\right]=0$ a.e. in $[T_1,T_2]^2$. Analogous arguments shows that 
$$\E_{\sqrt{\gamma}A}\left[\mathcal{S}_{2:p}\left(\epsilon_{\sqrt{\gamma}A}\right)_s^T \mathcal{S}_{2:p}\left(\sqrt{\gamma}A\right)_t\right]=0,$$ 
and that a.e. in $[T_1,T_2]^2$,  
$$\E_{\sqrt{\gamma}A}\left[\mathcal{S}_{2:p}\left(\epsilon_{\sqrt{\gamma} A}\right)_s^T \mathcal{S}_{2:p}\left(\epsilon_{\sqrt{\gamma} A}\right)_t\right]
=
\E_{O}\left[\mathcal{S}_2\left(\epsilon_{O}\right)_s \mathcal{S}_2\left(\epsilon_{O}\right)_t\right].
$$  
\\
We now prove that the formula \eqref{argminformula} is valid given \eqref{summationkrit}. If $\mathcal{K}$ is injective then $W=\{0\}$, which will imply the uniqueness. Let $\mathcal{L}_1=\overline{\mathsf{span}}\left(\{\psi_k\}_k\right)$ and $\mathcal{L}_2:=\mathcal{L}_1^{\perp}$. Then $\mathcal{L}_2$ is a closed subspace of $L^2([T_1,T_2])^p$ and is therefore separable, which implies it has a countable ON-basis, $\{\eta_k\}_k$ and moreover $L^2([T_1,T_2])^p=\mathcal{L}_1\oplus \mathcal{L}_2$ as well as $L^2([T_1,T_2])^p=\overline{\mathsf{span}}\left(\{\psi_k\}_k\cup \{\eta_k\}_k\right)$. Let $\{\tilde{\psi}_k\}_k$ be an enumeration of the ON-system $\{\psi_k\}_k\cup \{\eta_k\}_k$, which is a basis for $L^2([T_1,T_2])^p$. We will now show that $\{\tilde{\psi}_l\otimes\phi_k \}_{k,l\in\N}$ is an ON-basis for $L^2([T_1,T_2]^2)^p$. Denote the components of $\tilde{\psi}_{l}$ as $\tilde{\psi}_{l}=\left(\tilde{\psi}_{l,1},\ldots,\tilde{\psi}_{l,p} \right)$, where $\tilde{\psi}_{l,i}\in L^2([T_1,T_2])$. Then
\begin{align*}
\langle \tilde{\psi}_{l_1}\otimes \phi_{k_1}, \tilde{\psi}_{l_2}\otimes\phi_{k_2}\rangle_{L^2([T_1,T_2]^2)^p} 
&=
\sum_{i=1}^p\langle \tilde{\psi}_{l,i}\otimes \phi_{k_1}, \tilde{\psi}_{l_2,i}\otimes\phi_{k_2}\rangle_{L^2([T_1,T_2]^2)}
\\
&=
\sum_{i=1}^p\langle \tilde{\psi}_{l_1,i}, \tilde{\psi}_{l_2,i}\rangle_{L^2([T_1,T_2])}\langle \phi_{k_1}, \phi_{k_2}\rangle_{L^2([T_1,T_2])}
\\
&=
\langle \tilde{\psi}_{l_1}, \tilde{\psi}_{l_2}\rangle_{L^2([T_1,T_2])^p}\langle \phi_{k_1}, \phi_{k_2}\rangle_{L^2([T_1,T_2])}=\delta_{l_1,l_2}\delta_{k_1,k_2},
\end{align*}
which shows the orthogonality. For the completeness of this basis, let $f\in \left(\mathsf{span}\{\tilde{\psi}_l\otimes\phi_k \}_{k,l\in\N}\right)^\perp$ so that
$$ \langle f, \tilde{\psi}_{l}\otimes\phi_{k}\rangle_{L^2([T_1,T_2]^2)^p}=0, $$
for all $k,l\in\N$. Explicitly this means,
$$ \int_{[T_1,T_2]}\left(\int_{[T_1,T_2]}f(t,\tau) \tilde{\psi}_{l}(\tau)d\tau\right) \phi_k(t)dt=0.$$
If we let $g_l(t)=\int_{[T_1,T_2]}f(t,\tau) \tilde{\psi}_{l}(\tau)d\tau$ then, by the Cauchy-Schwarz inequality (applied component-wise), $g_l\in L^2([T_1,T_2])$. Moreover, $\langle g, \phi_{k}\rangle_{L^2([T_1,T_2])}=0$, $\forall k\in \N$ implying 
$$g_l=\sum_{k=1}^\infty \langle g_l, \phi_{k}\rangle_{L^2([T_1,T_2])}\phi_k=0,$$
in $L^2([T_1,T_2])$, further implying that $g_l=0$ a.e. in $[T_1,T_2]$. Let $F_l=\{t:g_l(t)\not=0\}$ and $F=\bigcup_{l\in\N}F_l$. On $F^c$,
$$ \int_{[T_1,T_2]}f(t,\tau) \tilde{\psi}_{l}(\tau)d\tau=0 , \forall l\in\N,$$
implying $f(t,\tau)=0$ outside a zero-set $F'\in [T_1,T_2]$. By the Fubini theorem,
$$ \int_{[T_1,T_2]}\int_{[T_1,T_2]}|f(t,\tau)| d\tau dt
=
\int_{[T_1,T_2]\setminus F}\left(\int_{[T_1,T_2]\setminus F'}|f(t,\tau)| d\tau\right) dt=0.$$
This implies that $f=0$ a.e. in $[T_1,T_2]^2$, which leads to the conclusion that $\left(\mathsf{span}\{\tilde{\psi}_l\otimes\phi_k \}_{k,l\in\N}\right)^\perp=\{0\}$, implying that $\{\tilde{\psi}_l\otimes\phi_k \}_{k,l\in\N}$ is indeed a basis. We therefore have the following representation,
\begin{align*}
L^2([T_1,T_2]^2)&=
\left\{ \sum_{k=1}^\infty \sum_{l=1}^\infty  \lambda_{k,l} \tilde{\psi}_k\otimes\phi_l: \sum_{k=1}^\infty \sum_{l=1}^\infty  \lambda_{k,l}^2<\infty\right\}\\
&=
\left\{ \sum_{k=1}^\infty \sum_{l=1}^\infty  \lambda_{k,l}^{(1)} \psi_k\otimes\phi_l + \sum_{k=1}^\infty \sum_{l=1}^\infty  \lambda_{k,l}^{(2)} \eta_k\otimes\phi_l: \sum_{k=1}^\infty \sum_{l=1}^\infty  (\lambda_{k,l}^{(1)})^2 + \sum_{k=1}^\infty \sum_{l=1}^\infty  (\lambda_{k,l}^{(2)})^2<\infty\right\}.
\end{align*}
Moreover $X^{\sqrt{\gamma}A}\in L^2([T_1,T_2])$ a.s. and hence if we let $S_n(t,\omega)=\sum_{k=1}^n\left(\chi^{(1)}_k(\omega)\psi_k(t)+\chi^{(2)}_k(\omega)\eta_k(t)\right)$ then $S_n\xrightarrow{L^2([T_1,T_2])}X^{\sqrt{\gamma}A}$ a.s.. Note however that
\begin{align*}
\E\left[\left(\chi^{(2)}_{l}\right)^2\right]
&=
\E\left[\int_{[T_1,T_2]} \int_{[T_1,T_2]} \eta_{l}(s)\left(X_s^{\sqrt{\gamma}A}\right)^TX_t^{\sqrt{\gamma}A}\eta_{l}(t)^Tdtds\right]
\\
&=
\int_{[T_1,T_2]} \int_{[T_1,T_2]} \eta_{l}(s)\E\left[\left(X_s^{\sqrt{\gamma}A}\right)^TX_t^{\sqrt{\gamma}A}\right]\eta_{l}(t)^Tdtds
\\
&=\int_{[T_1,T_2]} \int_{[T_1,T_2]} \eta_{l}(s)K_{X^{\sqrt{\gamma}A}}(t,s)\eta_{l}(t)^Tdtds
\\
&=\int_{[T_1,T_2]} (\mathcal{K}\eta_l)(t)\eta_{l}(t)^Tdtds
\\
&=\lim_{N\to\infty}\int_{[T_1,T_2]} \sum_{n=1}^N \alpha_n\langle\psi_n,\eta_l\rangle_{L^2([T_1,T_2])^p}\psi_n(t)\eta_{l}(t)^Tdt
\\
&=\lim_{N\to\infty} \sum_{n=1}^N \alpha_n\langle\psi_n,\eta_l\rangle_{L^2([T_1,T_2])^p} \int_{[T_1,T_2]}  \psi_n(t) \eta_{l}(t)dt
=0.
\end{align*}
This implies that in fact $S_n(t,\omega)=\sum_{k=1}^n\chi^{(1)}_k(\omega)\psi_k(t)\xrightarrow{L^2([T_1,T_2])}X^{\sqrt{\gamma}A}$ a.s.. We will therefore denote $\chi_l=\chi^{(1)}_l$ from now on. Since we expand $X^{\sqrt{\gamma}A}$ in the basis given by the eigenfunctions of $K_{X^{\sqrt{\gamma}A}}$ we also have that the sequence $\{\chi_l\}_l$ is orthogonal,
\begin{align*}
\E\left[\chi_{l_1}\chi_{l_2}\right]
&=
\E\left[\int_{[T_1,T_2]} \int_{[T_1,T_2]} \psi_{l_1}(s) \left( X_s^{\sqrt{\gamma}A}\right)^TX_t^{\sqrt{\gamma}A}\psi_{l_2}(t)^Tdtds\right]
\\
&=
\int_{[T_1,T_2]} \int_{[T_1,T_2]} \psi_{l_1}(s)\E\left[\left(X_s^{\sqrt{\gamma}A}\right)^TX_t^{\sqrt{\gamma}A}\right]\psi_{l_2}(t)^Tdtds
\\
&=\int_{[T_1,T_2]} \int_{[T_1,T_2]} \psi_{l_1}(s)K_{X^{\sqrt{\gamma}A}}(t,s)\psi_{l_2}(t)^Tdtds
\\
&=\lim_{N\to\infty} \int_{[T_1,T_2]} \int_{[T_1,T_2]} \sum_{n=1}^N \alpha_n\langle \psi_n,\psi_{l_1}\rangle_{L^2([T_1,T_2])^p}\psi_n(t) \psi_{l_2}(t)dtds
\\
&=\lim_{N\to\infty} \sum_{n=1}^N  \alpha_n\langle \psi_n,\psi_{l_1}\rangle_{L^2([T_1,T_2])^p}\langle \psi_n,\psi_{l_2}\rangle_{L^2([T_1,T_2])^p}
\\
&=\lim_{N\to\infty} \sum_{n=1}^N  \alpha_n\delta_{n,l_1}\delta_{n,l_2}
=\alpha_{l_1}\delta_{l_1,l_2}.
\end{align*}
Let
$$ V=\left\{\int \beta(t,\tau)X^{\sqrt{\gamma}A}_\tau d\tau:  \beta\in L^2([T_1,T_2]^2)^p \right\}.$$
Setting $\delta=\inf_{\beta\in L^2([T_1,T_2]^2)^p}\int_{[T_1,T_2]} \E\left[\left(Y_t-\int_{[T_1,T_2]} \beta(t,\tau)X^{\sqrt{\gamma}A}_\tau d\tau\right)^2\right]dt$, it follows that there exists a sequence $\{\beta_n\}_n\in L^2([T_1,T_2]^2)^p$ such that $\delta=\lim_{n\to\infty}\int_{[T_1,T_2]} \E\left[\left(Y_t-\int_{[T_1,T_2]} \beta_n(t,\tau)X^{\sqrt{\gamma}A}_\tau d\tau\right)^2\right]dt$. If we let $W_n(t)=\int_{[T_1,T_2]} \beta_n(t,\tau)X^{\sqrt{\gamma}A}_\tau d\tau$ then $\{W_n\}_n\subset V$ and $\delta=\lim_{n\to\infty}\int_{[T_1,T_2]} \E\left[\left(Y_t-W_n(t)\right)^2\right]dt$ which implies that $\inf_{W\in V}\int_{[T_1,T_2]} \E\left[\left(Y_t-W(t)\right)^2\right]dt\le \delta$. Therefore
\begin{align}\label{argmin}
\inf_{\beta\in L^2([T_1,T_2]^2)^p}\int \E\left[\left(Y_t-\int \beta(t,\tau)X^{\sqrt{\gamma}A}_\tau d\tau\right)^2\right]dt
&\ge
\inf_{W\in V}\int \E\left[\left(Y_t-W(t)\right)^2\right]dt\nonumber
\\
&\ge
\inf_{W\in \bar{V}}\int \E\left[\left(Y_t-W(t)\right)^2\right]dt.
\end{align}
This implies that if $\arg\min_{W\in \bar{V}}\int \E\left[\left(Y_t-W(t)\right)^2\right]dt\in V$ then there exists $\tilde{\beta}\in L^2([T_1,T_2]^2)^p$ such that \\$\arg\min_{W\in \bar{V}}\int \E\left[\left(Y_t-W(t)\right)^2\right]dt=\int \tilde{\beta}(t,\tau)X^{\sqrt{\gamma}A}_\tau d\tau$ and
$$
\arg\min_{\beta\in L^2([T_1,T_2]^2)^p}\int_{[T_1,T_2]} \E\left[\left(Y_t-\int_{[T_1,T_2]} \beta(t,\tau)X^{\sqrt{\gamma}A}(\tau)d\tau\right)^2\right]dt
=
\tilde{\beta}.
$$
Consider an arbitrary element in $V$,
\begin{align*}
h(t)
=
\int_{[T_1,T_2]} \beta(t,\tau)X^{\sqrt{\gamma}A}_\tau d\tau.
\end{align*}
Let 
$$S_n^{\beta}(t,\tau)=\sum_{k=1}^n\sum_{l=1}^n \lambda_{(1),k,l}^{\beta}\psi_l(\tau)\phi_k(t)+\sum_{k=1}^n\sum_{l=1}^n \lambda_{(2),k,l}^{\beta}\eta_l(\tau)\phi_k(t),$$
then $S_n^{\beta}\xrightarrow{L^2([T_1,T_2]^2)^p}\beta$ and  $\lVert S_n^{\beta} \rVert_{L^2([T_1,T_2]^2)^p}= \sum_{k=1}^n\sum_{l=1}^n\left(\left(\lambda_{(1),k,l}^{\beta}\right)^2+\left(\lambda_{(2),k,l}^{\beta}\right)^2\right)$.
Similarly to the proof of Theorem \ref{WR1},
\begin{align*}
\int_{[T_1,T_2]} \left( \int_{[T_1,T_2]}(\beta(t,\tau))X^{\sqrt{\gamma} A}_\tau d\tau \right)^2dt
&=
\int_{[T_1,T_2]} \left( \sum_{i=1}^p\int_{[T_1,T_2]}(\beta(t,\tau))(i)X^{\sqrt{\gamma} A}_\tau(i)d\tau \right)^2dt
\\
&\le
2^p\sum_{i=1}^p\int_{[T_1,T_2]}\left|X^{\sqrt{\gamma} A}_\tau(i)\right|^2d\tau 
\int_{[T_1,T_2]}  \int_{[T_1,T_2]}\left|(\beta(t,\tau))(i)\right|^2d\tau dt,
\end{align*}
Therefore if we let $Q(t)=\int_{[T_1,T_2]}(\beta(t,\tau))X^{\sqrt{\gamma}A}_\tau d\tau$ and 
$$S_N^{\int}(t)=\sum_{n=1}^N \left\langle \int_{[T_1,T_2]}\beta(.,\tau)X^{\sqrt{\gamma} A}_\tau d\tau,\phi_k\right\rangle_{L^2([T_1,T_2])} \phi_k(t)$$ 
Since
\begin{align*}
\left\langle \int_{[T_1,T_2]}\beta(.,\tau)X^{\sqrt{\gamma}A}_\tau d\tau,\phi_k\right\rangle_{L^2([T_1,T_2])}
&=
\lim_{n\to\infty}\sum_{k'=1}^n \sum_{l=1}^n \sum_{m=1}^n \lambda_{(1)k,l}^{\beta}\chi_m^{\sqrt{\gamma}A} \int_{[T_1,T_2]}\int_{[T_1,T_2]} \psi_l(\tau)^T\psi_m(\tau)\phi_{k'}(t)\phi_k(t)d\tau dt
\\
&+\lim_{n\to\infty}\sum_{k'=1}^n \sum_{l=1}^n \sum_{m=1}^n \lambda_{(2)k,l}^{\beta}\chi_m^{\sqrt{\gamma}A} \int_{[T_1,T_2]}\int_{[T_1,T_2]} \eta_l(\tau)^T\psi_m(\tau)\phi_{k'}(t)\phi_k(t)d\tau dt
\\
&=
\lim_{n\to\infty}\sum_{k'=1}^n \sum_{l=1}^n \sum_{m=1}^n \lambda_{(1),k,l}^{\beta}\chi_m^{\sqrt{\gamma}A}\delta_{l,m}\delta_{k',k}
= \sum_{l=1}^\infty \lambda_{(1)k,l}^{\beta} \chi_l^{\sqrt{\gamma}A}
\end{align*}
we get,
\begin{align*}
S_n^{\int}(t)
&=
\sum_{k=1}^n\sum_{l=1}^\infty \lambda_{(1),k,l}^{\beta} \chi_l^{\sqrt{\gamma}A} \phi_k(t).
\end{align*}
Arguing as in the proof of Theorem \ref{WR1} we get $\lim_{n\to\infty}\E_{\sqrt{\gamma}A}\left[\int_{[T_1,T_2]}\left(S_n^{\int}(t)-Q(t)\right)^2dt \right]
=0$ and then
\begin{align}\label{series2}
\lim_{n\to\infty}\E_{\sqrt{\gamma}A}\left[\int_{[T_1,T_2]}\left(\sum_{k=1}^n\sum_{l=1}^n \lambda^{\beta}_{(1),k,l}\chi^{\sqrt{\gamma}A}_l\phi_k(t)- \int_{[T_1,T_2]} \beta(t,\tau)X^{\sqrt{\gamma}A}_\tau d\tau \right)^2dt \right]
=0.
\end{align}
This implies that
\begin{align*}
\int_{[T_1,T_2]} \beta(t,\tau)X^{\sqrt{\gamma}A}_\tau d\tau 
=
\sum_{k=1}^\infty\sum_{l=1}^\infty \lambda^{(1)}_{k,l}\chi^{\sqrt{\gamma}A}_l\phi_k(t),
\end{align*}
where the limit exists in $L^2(dt\times\P_{\sqrt{\gamma}A})$. This implies that we may rewrite $V$ as, 
\begin{align}\label{V}
V= \left\{\sum_{k=1}^\infty \sum_{l=1}^\infty  \lambda_{k,l}\phi_k(t)
\chi^{\sqrt{\gamma}A}_l: \sum_{k=1}^\infty \sum_{l=1}^\infty  \lambda_{k,l}^2<\infty\right\},
\end{align}
where the series $\sum_{k=1}^\infty \sum_{l=1}^\infty  \lambda_{k,l}\phi_k(t)
\chi^{\sqrt{\gamma}A}_l$ converges in $L^2(dt\times\P_{\sqrt{\gamma}A})$. Let us now show that $\bar{V}=\overline{\mathsf{span}}\left\{\phi_k\frac{\chi^{\sqrt{\gamma}A}_l}{\n\chi^{\sqrt{\gamma}A}_l\n_{L^2(\P)}}\right\}_{k,l\in\N}$. First note that for an arbitrary $h\in V$,
\begin{align*}
h=\sum_{k=1}^\infty \sum_{l=1}^\infty  \lambda_{k,l}\phi_k(t)
\chi^{\sqrt{\gamma}A}_l(\omega)
=
\sum_{k=1}^\infty \sum_{l=1}^\infty  \lambda_{k,l}\lVert\chi_l\rVert_{L^2(\P)}
\frac{\chi^{\sqrt{\gamma}A}_l}{\n\chi^{\sqrt{\gamma}A}_l\n_{L^2(\P)}}\phi_k(t),
\end{align*} 
which means that if $\sum_{k=1}^\infty \sum_{l=1}^\infty  \lambda_{k,l}^2\n\chi^{\sqrt{\gamma}A}_l\n_{L^2(\P)}^2<\infty$, then $h\in \overline{\mathsf{span}}\left\{\phi_k(t)\frac{\chi^{\sqrt{\gamma}A}_l}{\n\chi^{\sqrt{\gamma}A}_l\n_{L^2(\P)}}\right\}_{k,l\in\N}$. Define $\beta_l=\sum_{k=1}^\infty  \lambda_{k,l}^2$ and note that $\{\beta_l\}_{l\in\N}\in \mathit{l}^1\subset \mathit{l}^2$. Since $\overline{\mathsf{span}}\left\{\phi_k(t)\frac{\chi^{\sqrt{\gamma}A}_l}{\n\chi^{\sqrt{\gamma}A}_l\n_{L^2(\P)}}\right\}_{k,l\in\N}$ is ON,
$$\overline{\mathsf{span}}\left\{\phi_k(t)\frac{\chi^{\sqrt{\gamma}A}_l}{\n\chi^{\sqrt{\gamma}A}_l\n_{L^2(\P)}}\right\}_{k,l\in\N}= \left\{\sum_{k=1}^\infty \sum_{l=1}^\infty  \lambda_{k,l}\phi_k(t)
\frac{\chi^{\sqrt{\gamma}A}_l}{\n\chi^{\sqrt{\gamma}A}_l\n_{L^2(\P)}}: \sum_{k=1}^\infty \sum_{l=1}^\infty  \lambda_{k,l}^2<\infty\right\}$$
and by the Cauchy-Schwartz inequality
$$\sum_{k=1}^\infty \sum_{l=1}^\infty  \lambda_{k,l}^2\n\chi^{\sqrt{\gamma}A}_l\n_{L^2(\P)}^2
= 
\sum_{l=1}^\infty  \beta_l\n\chi^{\sqrt{\gamma}A}_l\n_{L^2(\P)}^2
\le
\sqrt{\sum_{l=1}^\infty  \beta_l^2}
\sqrt{\sum_{l=1}^\infty  \n\chi^{\sqrt{\gamma}A}_l\n_{L^2(\P)}^4}
<\infty.$$
Hence $V\subset \overline{\mathsf{span}}\left\{\phi_k(t)\frac{\chi^{\sqrt{\gamma}A}_l}{\n\chi^{\sqrt{\gamma}A}_l\n_{L^2(\P)}}\right\}_{k,l\in\N} $, taking closure yields $\overline{V}\subset \overline{\mathsf{span}}\left\{\phi_k(t)\frac{\chi^{\sqrt{\gamma}A}_l}{\n\chi^{\sqrt{\gamma}A}_l\n_{L^2(\P)}}\right\}_{k,l\in\N} $. Take 
$$g\in \mathsf{span}\left\{\phi_k(t)\frac{\chi^{\sqrt{\gamma}A}_l}{\n\chi^{\sqrt{\gamma}A}_l\n_{L^2(\P)}}\right\}_{k,l\in\N} 
=
\left\{\sum_{k=1}^K \sum_{l=1}^L  \lambda_{k,l}\phi_k(t)
\chi^{\sqrt{\gamma}A}_l(\omega):K,L\in\N, \lambda_{k,l}\in\R\right\}.
$$
Clearly $g\in V$, i.e. $\mathsf{span}\left\{\phi_k(t)\frac{\chi^{\sqrt{\gamma}A}_l}{\n\chi^{\sqrt{\gamma}A}_l\n_{L^2(\P)}}\right\}_{k,l\in\N} \subset V$ and taking closure on both sides shows $\overline{\mathsf{span}}\left\{\phi_k(t)\frac{\chi^{\sqrt{\gamma}A}_l}{\n\chi^{\sqrt{\gamma}A}_l\n_{L^2(\P)}}\right\}_{k,l\in\N}\subset \bar{V}$, thus $\overline{\mathsf{span}}\left\{\phi_k(t)\frac{\chi^{\sqrt{\gamma}A}_l}{\n\chi^{\sqrt{\gamma}A}_l\n_{L^2(\P)}}\right\}_{k,l\in\N}=\bar{V}$. Recall that $Y^{\sqrt{\gamma}A}_t=\sum_{k=1}^\infty Z^{\sqrt{\gamma}A}_k\phi_k(t)$. Since $\bar{V}$ is a closed subspace of $L^2(\P_{\sqrt{\gamma}A}\times dt)$ which is a Hilbert space, we have by the Hilbert space projection theorem that 
$$\arg\min_{W\in \bar{V}}\int_{[T_1,T_2]} \E_{\sqrt{\gamma}A}\left[\left(Y^{\sqrt{\gamma}A}_t-W_t\right)^2\right]dt=Proj_{\bar{V}}(Y^{\sqrt{\gamma}A}).$$ 
As $\bar{V}=\overline{\mathsf{span}}\left\{\phi_k(t)\frac{\chi^{\sqrt{\gamma}A}_l}{\n\chi^{\sqrt{\gamma}A}_l\n_{L^2(\P)}}\right\}_{k,l\in\N}$, it follows that $\left\{\phi_k(t)\frac{\chi^{\sqrt{\gamma}A}_l}{\n\chi^{\sqrt{\gamma}A}_l\n_{L^2(\P)}}\right\}_{k,l\in\N}$ is an ON-basis for this space and therefore 
\begin{align}\label{Proj1}
Proj_{\bar{V}}(Y^{\sqrt{\gamma}A})
&=
\sum_{k=1}^\infty \sum_{l=1}^\infty \langle Y,\phi_k\frac{\chi^{\sqrt{\gamma}A}_l}{\n \chi^{\sqrt{\gamma}A}_l\n_{L^2(\P)}} \rangle_{H} \phi_k\frac{\chi^{\sqrt{\gamma}A}_l}{\n \chi^{\sqrt{\gamma}A}_l\n_{L^2(\P)}}\nonumber
\\
&=
\sum_{k=1}^\infty \sum_{l=1}^\infty \phi_k(t)\frac{\chi^{\sqrt{\gamma}A}_l(\omega)}{\n \chi^{\sqrt{\gamma}A}_l\n_{L^2(\P)}^2}\E\left[\int Y_\tau\chi^{\sqrt{\gamma}A}_l\phi_k(\tau) d\tau \right]\nonumber
\\
&=\sum_{k=1}^\infty \sum_{l=1}^\infty \phi_k(t)\frac{\chi^{\sqrt{\gamma}A}_l(\omega)}{\n \chi^{\sqrt{\gamma}A}_l\n_{L^2(\P)}^2}\int \E\left[Y_\tau\chi^{\sqrt{\gamma}A}_l\right]\phi_k(\tau) d\tau 
\end{align}
Since $S_n^Y\xrightarrow{L^2(\P_{\sqrt{\gamma}A}\times dt)}Y^{\sqrt{\gamma}A}$
\begin{align*}
\n Y^{\sqrt{\gamma}A}\phi_k\chi^{\sqrt{\gamma}A}_l -S_n^{Y^{\sqrt{\gamma}A}}\phi_k\chi^{\sqrt{\gamma}A}_l\n_{L^2(\P_{\sqrt{\gamma}A}\times dt)}
&\le
\n Y^{\sqrt{\gamma}A} -S_n^{Y^{\sqrt{\gamma}A}}\n_{L^2(\P_{\sqrt{\gamma}A}\times dt)}\n \phi_k\chi^{\sqrt{\gamma}A}_l \n_{L^2(\P_{\sqrt{\gamma}A}\times dt)}
\\
&=\n Y^{\sqrt{\gamma}A} -S_n^{Y^{\sqrt{\gamma}A}}\n_{L^2(\P_{\sqrt{\gamma}A}\times dt)}\n \chi^{\sqrt{\gamma}A}_l \n_{L^2(\P_{\sqrt{\gamma}A}\times dt)},
\end{align*}
which converges to zero for every $l$. Therefore
\begin{align*}
\int_{[T_1,T_2]} \E_{\sqrt{\gamma}A}\left[Y^{\sqrt{\gamma}A}_\tau\chi^{\sqrt{\gamma}A}_l\right]\phi_k(\tau) d\tau 
&= 
\sum_{m=1}^\infty \E_{\sqrt{\gamma}A}\left[Z^{\sqrt{\gamma}A}_m\chi^{\sqrt{\gamma}A}_l\right]\int_{[T_1,T_2]} \phi_m(\tau)\phi_k(\tau) d\tau
\\
&=\E_{\sqrt{\gamma}A}\left[Z^{\sqrt{\gamma}A}_k\chi^{\sqrt{\gamma}A}_l\right],
\end{align*}
which we then plug back into \eqref{Proj1} to conclude that 
\begin{align}\label{Proj2}
Proj_{\bar{V}}(Y)
=
\sum_{k=1}^\infty \sum_{l=1}^\infty   \phi_k(t)\chi^{\sqrt{\gamma}A}_l  \frac{\E_{\sqrt{\gamma}A}\left[ Z_k\chi^{\sqrt{\gamma}A}_l\right]}{\E_{\sqrt{\gamma}A}\left[ \left(\chi^{\sqrt{\gamma}A}_l\right)^2\right]}.
\end{align}
If we let $\lambda_{k,l}=\frac{\E_{\sqrt{\gamma}A}\left[ Z_k\chi^{\sqrt{\gamma}A}_l\right]}{\E_{\sqrt{\gamma}A}\left[ \left(\chi^{\sqrt{\gamma}A}_l\right)^2\right]}$ then by assumption
$ \sum_{k=1}^\infty \sum_{l=1}^\infty \lambda_{k,l}^2<\infty $.
This implies that $Proj_{\bar{V}}(Y)\in V$, so the last inequality in \eqref{argmin} is in fact an equality attained at this projection. By representing $\beta\in L^2([T_1,T_2]^2)$ as $\beta(t,\tau)=\sum_{k=1}^\infty \sum_{l=1}^\infty \lambda_{k,l}\phi_k(t)\psi_k(\tau)$ we get 
$$\int_{[T_1,T_2]} B(t,\tau)X(\tau)d\tau= \sum_{k=1}^\infty \sum_{l=1}^\infty  \lambda_{k,l}\phi_k(t)
\chi^{\sqrt{\gamma}A}_l(\omega).$$ 
Comparing with \eqref{Proj2}, we see that a solution for the $\arg\min$ is to let $\lambda_{k,l}=\frac{\E\left[ Z_k\chi^{\sqrt{\gamma}A}_l\right]}{\E\left[ \left(\chi^{\sqrt{\gamma}A}_l\right)^2\right]}$, which yields $\beta(t,\tau)=\sum_{k=1}^\infty \sum_{l=1}^\infty \phi_k(t)\psi_l(\tau) \frac{\E\left[ Z_k\chi^{\sqrt{\gamma}A}_l\right]}{\E\left[ \left(\chi^{\sqrt{\gamma}A}_l\right)^2\right]}$. As $\{\phi_k\otimes\psi_l\}_{k,l\in\N}$ (since $W=\{0\}$) is an ON basis for $L^2([T_1,T_2]^2)$, this element is unique. 
\\
On the other hand, suppose a unique $\arg\min$ solution exists, then there must exist a unique $B\in L^2([T_1,T_2]^2)$ implying that $\int_{[T_1,T_2]} B(t,\tau)X(\tau)d\tau\in V$. We may represent $B=\sum_{k,l\in\N}\lambda_{k,l}\phi_k\otimes\psi_l$. Suppose $\exists k',l'\in \N$ such that $\lambda_{k',l'}\not=\frac{\E_{\sqrt{\gamma}A}\left[Z_k^{\sqrt{\gamma}A}\chi_l^{\sqrt{\gamma}A}\right]}{\E_{\sqrt{\gamma}A}\left[\left(\chi_l^{\sqrt{\gamma}A}\right)^2\right]}$. Suppose first that $\E_{\sqrt{\gamma}A}\left[\left(\chi_{l'}^{\sqrt{\gamma}A}\right)^2\right]=0$, implying that $\chi_{l'}^{\sqrt{\gamma}A}=0$, $\P_{\sqrt{\gamma}A}$ a.s.. Then for any real number $\lambda\not=\lambda_{k,l}$, we have that if we set $\tilde{B}=\sum_{k,l\in\N}\tilde{\lambda}_{k,l}\phi_k\otimes\psi_l$, where $\tilde{\lambda}_{k,l}=\lambda_{k,l}$ if $(k,l)\not=(k',l')$ and $\tilde{\lambda}_{k',l'}=\lambda$ then $\int_{[T_1,T_2]} \tilde{B}(t,\tau)X(\tau)d\tau = \int_{[T_1,T_2]} B(t,\tau)X(\tau)d\tau$, implying that also $\tilde{B}$ is a minimizer, contradicting the uniqueness of $B$. Suppose instead that $\E_{\sqrt{\gamma}A}\left[\left(\chi_{l'}^{\sqrt{\gamma}A}\right)^2\right]\not=0$ and let 
\begin{align*}
f_n(\lambda)&= \sum_{k=1}^{k'-1}\E_{\sqrt{\gamma} A}\left[ \left(Z_k^{\sqrt{\gamma} A}- \sum_{l=1}^n \lambda_{k,l}\chi_l^{\sqrt{\gamma} A}\right)^2\right] 
+
\E_{\sqrt{\gamma} A}\left[ \left(Z_{k'}^{\sqrt{\gamma} A}-\lambda \chi_{l'}^{\sqrt{\gamma} A}\right)^2\right]
\\
&+
\E_{\sqrt{\gamma} A}\left[ \left(Z_{k'}^{\sqrt{\gamma} A}- \sum_{1\le l\le n, l\not=l'} \lambda_{k',l}\chi_l^{\sqrt{\gamma} A}\right)^2\right]
+
\sum_{k=k'+1}^{n}\E_{\sqrt{\gamma} A}\left[ \left(Z_k^{\sqrt{\gamma} A}- \sum_{l=1}^n \lambda_{k,l}\chi_l^{\sqrt{\gamma} A}\right)^2\right] 
\end{align*}
then for any $n\ge k'\vee l'$, $\arg\min_\lambda f_n(\lambda)=\frac{\E_{\sqrt{\gamma}A}\left[Z_k^{\sqrt{\gamma}A}\chi_l^{\sqrt{\gamma}A}\right]}{\E_{\sqrt{\gamma}A}\left[\left(\chi_l^{\sqrt{\gamma}A}\right)^2\right]}$. If we now set $\tilde{\lambda}_{k,l}=\lambda_{k,l}$ if $(k,l)\not=(k',l')$ and $\tilde{\lambda}_{k',l'}=\frac{\E_{\sqrt{\gamma}A}\left[Z_k^{\sqrt{\gamma}A}\chi_l^{\sqrt{\gamma}A}\right]}{\E_{\sqrt{\gamma}A}\left[\left(\chi_l^{\sqrt{\gamma}A}\right)^2\right]}$ implying that for every $n\ge k'\vee l'$
\begin{align}\label{Rineq}
\sum_{k=1}^{n}\E_{\sqrt{\gamma} A}\left[ \left(Z_k^{\sqrt{\gamma} A}- \sum_{l=1}^n \tilde{\lambda}_{k,l}\chi_l^{\sqrt{\gamma} A}\right)^2\right]
\le 
\sum_{k=1}^{n}\E_{\sqrt{\gamma} A}\left[ \left(Z_k^{\sqrt{\gamma} A}- \sum_{l=1}^n \lambda_{k,l}\chi_l^{\sqrt{\gamma} A}\right)^2\right].
\end{align}
Define $\tilde{B}=\sum_{k,l\in\N}\tilde{\lambda}_{k,l}\phi_k\otimes\psi_l$. An argument analogous used to establish \eqref{RAlim}, shows that 
$$ R_{\sqrt{\gamma}A}(B)= \lim_{n\to\infty}\sum_{k=1}^{n}\E_{\sqrt{\gamma} A}\left[ \left(Z_k^{\sqrt{\gamma} A}- \sum_{l=1}^n \lambda_{k,l}\chi_l^{\sqrt{\gamma} A}\right)^2\right]$$
and
$$ R_{\sqrt{\gamma}A}(\tilde{B})= \lim_{n\to\infty}\sum_{k=1}^{n}\E_{\sqrt{\gamma} A}\left[ \left(Z_k^{\sqrt{\gamma} A}- \sum_{l=1}^n \tilde{\lambda}_{k,l}\chi_l^{\sqrt{\gamma} A}\right)^2\right].$$
By letting $n\to\infty $ in \eqref{Rineq} we then obtain
$$R_{\sqrt{\gamma} A}\left(\tilde{B}\right)\le R_{\sqrt{\gamma} A}\left(B\right), $$
which directly contradicts the uniqueness of the minimizer $B$. We also need to establish that in this case $\mathcal{K}$ is injective. Suppose it is not, then $W\not=0$ and in particular $\beta'=\phi\otimes\eta\in W$. Noting that
\begin{align*}
\int_{[T_1,T_2]}\beta'(t,\tau)X_\tau d\tau
=
\sum_{l=1}^\infty\phi(t) \chi_l\int_{[T_1,T_2]}\eta_1(\tau)^T\psi_l(\tau) d\tau=0.
\end{align*}
Hence, if $\beta$ is a solution to \eqref{argmin} then so is $\beta+\beta'$, contradicting the uniqueness.
\\
We now note that $Y^{\sqrt{\gamma}A}=\mathcal{S}_1\left(\sqrt{\gamma}A+\epsilon_{\sqrt{\gamma}A} \right)$ and $X^{\sqrt{\gamma}A}=\mathcal{S}_{2:p}\left(\sqrt{\gamma}A+\epsilon_{\sqrt{\gamma}A} \right)$. Therefore, if we let $F_k(U)=\langle U,\psi_k \rangle_{L^2([T_1,T_2])^p}$, for $U\in L^2([T_1,T_2])^p$, $\overline{F}_k(u)=\langle u,\phi_k \rangle_{L^2([T_1,T_2])}$, for $u\in L^2([T_1,T_2])$, and $S_{2:p}(.)=\left(S_2(.),\ldots,S_p(.)\right)$ then
\begin{align}\label{numerat}
\E_{\sqrt{\gamma}A}\left[ Z_k^{\sqrt{\gamma}A}\chi_l^{\sqrt{\gamma}A}\right]
=
\E_A\left[ \overline{F}_k\left(\mathcal{S}_1\left(\sqrt{\gamma}A+\epsilon_{A} \right) \right)F_l\left(\mathcal{S}_{2:p}\left(\sqrt{\gamma}A+\epsilon_{A} \right) \right)\right]
\end{align}
and
\begin{align}\label{divis}
\E_{\sqrt{\gamma}A}\left[ \left(\chi_l^{\sqrt{\gamma}A}\right)^2\right]
=
\E\left[ F_l\left(\mathcal{S}_{2:p}\left(\sqrt{\gamma}A+\epsilon_{\sqrt{\gamma}A} \right) \right)^2\right].
\end{align}
Expanding \eqref{numerat} and utilizing Claim \ref{claimepsA} yields
\begin{align}\label{numerat1}
\E_{\sqrt{\gamma}A}\left[ Z_k^{\sqrt{\gamma}A}\chi_l^{\sqrt{\gamma}A}\right]
&=
\gamma\E_A\left[ \overline{F}_k\left(\mathcal{S}_1\left(A \right) \right)F_l\left(\mathcal{S}_{2:p}\left(A\right) \right)\right]
+
\E_A\left[ \overline{F}_k\left(\mathcal{S}_1\left(\epsilon_{A} \right) \right)F_l\left(\mathcal{S}_{2:p}\left(\epsilon_{A} \right) \right)\right]\nonumber
\\
&=
\gamma\E_A\left[ \overline{F}_k\left(\mathcal{S}_1\left(A+\epsilon_{A} \right) \right)F_l\left(\mathcal{S}_{2:p}\left(A+\epsilon_{A}\right) \right)\right]
+
\left(1-\gamma\right)\E_A\left[ \overline{F}_k\left(\mathcal{S}_1\left(\epsilon_{A} \right) \right)F_l\left(\mathcal{S}_{2:p}\left(\epsilon_{A} \right) \right)\right]\nonumber
\\
&=
\gamma\E_{A}\left[ Z_k^{A}\chi_l^{A}\right]
+
(1-\gamma)\E_{O}\left[ Z_k^{O}\chi_l^{O}\right].
\end{align}
While expanding \eqref{divis} and utilizing Claim \ref{claimeps} yields
\begin{align}\label{divis1}
\E_{\sqrt{\gamma}A}\left[ \left(\chi_l^{\sqrt{\gamma}A}\right)^2\right]
&=
\gamma\E\left[ F_l\left(\mathcal{S}_{2:p}\left(A \right) \right)^2\right]  +\E\left[ F_l\left(\mathcal{S}_{2:p}\left(\epsilon \right) \right)^2\right]\nonumber
\\
&=\gamma\E\left[ F_l\left(\mathcal{S}_{2:p}\left(A +\epsilon_A\right) \right)^2\right]  +(1-\gamma)\E\left[ F_l\left(\mathcal{S}_{2:p}\left(\epsilon \right) \right)^2\right]\nonumber
\\
&=
\gamma\E_{A}\left[ \left(\chi_l^{A}\right)^2\right]+(1-\gamma)\E_{O}\left[ \left(\chi_l^{O}\right)^2\right].
\end{align}
Therefore
$$\beta(t,\tau)=\sum_{k=1}^\infty \sum_{l=1}^\infty \frac{\gamma\E_{A}\left[ Z_k^{A}\chi_l^{A}\right]
+
(1-\gamma)\E_{O}\left[ Z_k^{O}\chi_l^{O}\right]}{\gamma\E_{A}\left[ \left(\chi_l^{A}\right)^2\right]+(1-\gamma)\E_{O}\left[ \left(\chi_l^{O}\right)^2\right]}\phi_k(t)\psi_l(\tau) $$
	\end{proof}
	
\subsection{Proof of Theorem \ref{Minimizerpg1}}
\begin{proof}
By \eqref{numerat1},
$$\E_{\sqrt{\gamma}A}\left[Z_k^{\sqrt{\gamma}A}F_{1:n}\left(X^{\sqrt{\gamma}A}\right)\right]
=
\gamma\E_{A}\left[Z_k^AF_{1:n}\left(X^{A}\right)\right]
+
(1-\gamma)\E_{O}\left[Z_k^{O}F_{1:n}\left(X^{O}\right)\right]  $$
and similar arguments imply
$$\E_{\sqrt{\gamma}A}\left[F_{1:n}\left(X^{\sqrt{\gamma}A}\right)^TF_{1:n}\left(X^{\sqrt{\gamma}A}\right)\right]=\sqrt{\gamma}\E_{A}\left[F_{1:n}\left(X^{A}\right)^TF_{1:n}\left(X^{A}\right)\right]+(1-\sqrt{\gamma})\E_{O}\left[F_{1:n}\left(X^{O}\right)^TF_{1:n}\left(X^{O}\right)\right].$$
Let $V_n=span\left(\mathcal{H}^{p+1}\left(\{\phi_{i,k}\}_{k\le n, 1\le i\le p+1}\right)\right)$. By assumption there exists an $N\in\N$ such that if $n\ge N$ then $\det\left(G_n\right)\not=0$. For $\{\lambda_{i,k,l}\}_{1\le i\le p, 1\le k,l\le n}$, let
\begin{align*}
h_{n}\left(\lambda_{1,1,1}',\ldots,\lambda_{p,n,n}'\right)
&=
R_{\sqrt{\gamma}A}\left( \left(\sum_{k=1}^n\sum_{l=1}^n\lambda_{1,k,l}\phi_k\otimes\phi_l,\ldots,\sum_{k=1}^n\sum_{l=1}^n\lambda_{p,k,l}\phi_k\otimes\phi_l \right) \right)
\\
&=
\sum_{k=1}^n\E_{\sqrt{\gamma}A}\left[ \left(Z_k^{\sqrt{\gamma}A}- \sum_{i=1}^p\sum_{l=1}^n \lambda_{i,k,l}'\chi_l^{\sqrt{\gamma}A}(i)\right)^2\right]
+
\sum_{k=n+1}^\infty\E_{\sqrt{\gamma}A}\left[ \left(Z_k^{\sqrt{\gamma}A}\right)^2\right].
\end{align*}
Any minimum of $h_{n}$ must fulfil $\nabla h_{n}=0$. Solving $\nabla h_{n}=0$ gives,
\begin{align*}
\E_{\sqrt{\gamma}A}\left[F_{1:n}\left(X^{\sqrt{\gamma}A}\right)^TF_{1:n}\left(X^{\sqrt{\gamma}A}\right)\right]\left(\lambda_{1,k,1}',\ldots,\lambda_{1,k,n}',\ldots,\lambda_{p,k,n}' \right)=\E_{\sqrt{\gamma}A}\left[Z_kF_{1:n}\left(X^{\sqrt{\gamma}A}\right)\right],
\end{align*}
for $1\le k\le n$. If $n\ge N$ we therefore have the unique solution
\begin{align*}
\left(\lambda_{1,k,1}(n),\ldots,\lambda_{1,k,n}(n),\ldots,\lambda_{p,k,n}(n) \right)=G_n^{-1}\E_{\sqrt{\gamma}A}\left[Z_kF_{1:n}\left(X^{\sqrt{\gamma}A}\right)\right].
\end{align*}
Therefore $\beta_n$ is the unique minimizer on $V_n$, implying that 
\begin{align}\label{miniVn}
R_{\sqrt{\gamma}A}(\beta_n)\le R_{\sqrt{\gamma}A}(P_{V_n}\beta)
\end{align} for any $\beta\in L^2([T_1,T_2]^2)^p$
\begin{claim}\label{betacont}
For $\{\tilde{\beta}_n\}_{n\in\N}\subset L^2([T_1,T_2]^2)^p$, if $\lVert \tilde{\beta}_n-\tilde{\beta} \rVert_{L^2([T_1,T_2]^2)^p}\to 0$ for some $\tilde{\beta}\in L^2([T_1,T_2]^2)^p$ then $R_{A'}\left( \tilde{\beta}_n\right)\to R_{A'}\left( \tilde{\beta}\right)$ for any $A'\in\mathcal{V}$.
\end{claim}
\begin{proof}
By the Cacuhy-Schwarz inequality
\begin{align*}
\left| R_{A'}\left( \tilde{\beta}_n\right)-R_{A'}\left( \tilde{\beta}\right) \right|
&\le
\sum_{i=1}^p 2\E_{A'}\left[\int_{[T_1,T_2]^2}\left|Y^{A'}_tX^{A'}_\tau(i)\left((\tilde{\beta}_n(i))(t,\tau)-(\tilde{\beta}(i))(t,\tau)\right)\right|dtd\tau\right] 
\\
&+
\sum_{i=1}^p\E_{A'}\left[\int_{[T_1,T_2]}\left|\left(\int_{[T_1,T_2]} \left(\tilde{\beta}_n(i))(t,\tau)-(\tilde{\beta}(i))(t,\tau)\right)X^{A'}_\tau(i) d\tau\right)\right.\right.
\\
&\left.\left.\left(\int_{[T_1,T_2]}\left(\tilde{\beta}_n(i))(t,\tau)+(\tilde{\beta}(i))(t,\tau)\right)X^{A'}_\tau(i) d\tau\right)\right|dt\right]
\\
&\le
\sum_{i=1}^p 2\lVert \tilde{\beta}_n(i)-\tilde{\beta}(i)\rVert_{L^2([T_1,T_2]^2)}\E_{A'}\left[\left(\int_{[T_1,T_2]^2}\left(Y^{A'}_t\right)^2\left(X^{A'}_\tau(i)\right)^2dtd\tau\right)^{\frac12}\right]
\\
&+
\sum_{i=1}^p\E_{A'}\left[\left(\int_{[T_1,T_2]}\left(\int_{[T_1,T_2]} \left(\tilde{\beta}_n(i))(t,\tau)-(\tilde{\beta}(i))(t,\tau)\right)X^{A'}_\tau(i) d\tau \right)^2dt\right)^{\frac12}\right.
\\
&\left.\left(\int_{[T_1,T_2]}\left(\int_{[T_1,T_2]}\left(\tilde{\beta}_n(i))(t,\tau)+(\tilde{\beta}(i))(t,\tau)\right)X^{A'}_\tau(i) d\tau\right)^2dt\right)^{\frac12}\right] 
\\
&\le
\sum_{i=1}^p 2\lVert \tilde{\beta}_n(i)-\tilde{\beta}(i)\rVert_{L^2([T_1,T_2]^2)}T\E_{A'}\left[ \lVert Y^{A'} \rVert_{L^2([T_1,T_2]^2)} \lVert X^{A'}_\tau(i) \rVert_{L^2([T_1,T_2]^2)}\right]
\\
&+
\sum_{i=1}^p\E_{A'}\left[\left(\int_{[T_1,T_2]}\left(\int_{[T_1,T_2]} \left(\tilde{\beta}_n(i))(t,\tau)-(\tilde{\beta}(i))(t,\tau)\right)^2 d\tau\right)^{\frac12} \left(\int_{[T_1,T_2]}\left(X^{A'}_\tau(i) \right)^2d\tau \right)^{\frac12}dt\right)^{\frac12}\right.
\\
&\left.\left(\int_{[T_1,T_2]}\left(\int_{[T_1,T_2]} \left(\tilde{\beta}_n(i))(t,\tau)+(\tilde{\beta}(i))(t,\tau)\right)^2 d\tau\right)^{\frac12} \left(\int_{[T_1,T_2]}\left(X^{A'}_\tau(i) \right)^2d\tau \right)^{\frac12}dt\right)^{\frac12}\right] 
\\
&\le
2T\lVert \tilde{\beta}_n-\tilde{\beta}\rVert_{L^2([T_1,T_2]^2)^p}\E_{A'}\left[ \lVert Y^{A'} \rVert_{L^2([T_1,T_2]^2)}^2\right]^{\frac12}\E_{A'}\left[ \lVert X^{A'} \rVert_{L^2([T_1,T_2]^2)^p}^2\right]^{\frac12}
\\
&+\sqrt{T}\E_{A'}\left[ \lVert X^{A'} \rVert_{L^2([T_1,T_2]^2)^p}\right]^{\frac12}
\sqrt{T}\lVert \tilde{\beta}_n-\tilde{\beta} \rVert_{L^2([T_1,T_2]^2)^p}^{\frac12}\sqrt{T}\sqrt{\lVert \tilde{\beta}_n\rVert_{L^2([T_1,T_2]^2)^p}+\lVert \tilde{\beta} \rVert_{L^2([T_1,T_2]^2)^p}},
\end{align*}
which converges to zero (note that $\{\lVert \tilde{\beta}_n\rVert_{L^2([T_1,T_2]^2)^p}\}_{n\in\N}$ is bounded since $\lVert \tilde{\beta}_n-\tilde{\beta} \rVert_{L^2([T_1,T_2]^2)^p}\to 0$), proving the claim.
\end{proof}
Combining Claim \ref{betacont} with \eqref{miniVn} implies that 
\begin{align}\label{limsupbeta}
\limsup_{n\to\infty} R_{\sqrt{\gamma}A}(\beta_n)\le R_{\sqrt{\gamma}A}(\beta),
\end{align}
for any $\beta\in L^2([T_1,T_2]^2)^p$. Given that $\{ \lambda(n) \}_{n\in\N}$ contains a subsequence, $\{ \lambda(n_k) \}_{k\in\N}$ that converges in $\mathit{l}^2$ with a corresponding limit $\lambda\in \mathit{l}^2$ then also $\beta_{n_k}\xrightarrow{L^2([T_1,T_2]^2)^p}\beta'$ for some $\beta'\in L^2([T_1,T_2]^2)^p$. By \eqref{limsupbeta}, it follows that $R_{\sqrt{\gamma}A}(\beta')\le R_{\sqrt{\gamma}A}(\beta)$ for all $\beta\in L^2([T_1,T_2]^2)^p$, i.e. $\beta'\in S\not=\emptyset$. This proves the first claim of the theorem. For the second claim suppose $dist(\beta_n,S)\not\to 0$ then there is a subsequence $\{\beta_{n_k}\}_k$ such that  $dist(\beta_{n_k},S)>\delta$  for some $\delta>0$ and every $k\in\N$. By assumption we may then extract a further subsequence $\{\beta_{n_{k_l}}\}_{l\in\N}$ such that $\beta_{n_{k_l}}\xrightarrow{L^2([T_1,T_2]^2)^p}\beta''$ for some $\beta''\in L^2([T_1,T_2]^2)^p$. But since $\beta''\not\in S$, $R_{\sqrt{\gamma}A}(\beta'')>R_{\sqrt{\gamma}A}(\beta)$ for every $\beta\in S$, but due to \eqref{limsupbeta} and Claim \ref{betacont},
$$R_{\sqrt{\gamma}A}(\beta'')=\lim_{l\to\infty} R_{\sqrt{\gamma}A}\left(\beta_{n_{k_l}}\right)\le R_{\sqrt{\gamma}A}(\beta), $$
for any $\beta\in L^2([T_1,T_2]^2)^p$. This contradiction proves the second claim.
\end{proof}
	
	\subsection{Proof of Theorem \ref{ConsisThm}}
We provide the proof for the case $p=1$. The case $p>1$ is analogous, just more notationally cumbersome.

	\begin{proof}
		\textbf{Step 1: Use the whole sample curves.}\\
		We first fix $E\in\N$ and let 
		\begin{itemize}
			\item[] $C_{l}^{A,m}=\langle X^{A,m},\psi_l \rangle_{L^2([T_1,T_2])}$, $C_{l}^{O,m}=\langle X^{O,m},\psi_l \rangle_{L^2([T_1,T_2])}$,
			\item[] $D_k^{A,m}=\langle Y^{A,m},\phi_k \rangle_{L^2([T_1,T_2])}$, $D_k^{O,m}=\langle Y^{O,m},\phi_k \rangle_{L^2([T_1,T_2])}$,
			\item[] $\beta^{(m),E}=\sum_{k=1}^E \sum_{l=1}^E \phi_k\otimes\psi_l \frac{\gamma C_l^{A,m}D_k^{A,m}+(1-\gamma) C_l^{O,m}D_k^{O,m}}{\E\left[\gamma(\chi_l^A)^2+(1-\gamma)(\chi_l^O)^2\right]}$
			\item[] $\beta_E=\sum_{k=1}^E \sum_{l=1}^E \phi_k\otimes\psi_l \frac{\E\left[\gamma\chi_l^{A}Z_k^{A}+(1-\gamma)\chi_l^{O}Z_k^{O}\right]}{\E\left[\gamma(\chi_l^{A})^2+(1-\gamma)(\chi_l^{O})^2\right]}$
			\item[] $\hat{\beta}_{n,E,(1)}=\frac{1}{n}\sum_{m=1}^n \beta^{(m),E}$
			\item[] $Q_l=\gamma(\chi_l^A)^2+(1-\gamma)(\chi_l^O)^2$ 
			\item[] $W_l^m=\gamma (C_l'^{A,m})^2+ (1-\gamma)(C_l'^{O,m})^2$ and
			\item[] $U_{l,k}^m=\gamma C_l^{A,m}D_k^{A,m}+(1-\gamma)C_l^{O,m}D_k^{O,m}$
		\end{itemize}
		Note that by orthogonality of the basis functions and the Cauchy-Schwarz inequality,
		\begin{align}\label{BNorm}
			\E\left[\n \beta^{(m),E}\n_{L^2([T_1,T_2]^2)}\right]
			&=
			\E\left[\left(\int_{[T_1,T_2]^2}\left(\sum_{l=1}^E\sum_{k=1}^E \frac{U_{l,k}^m}{\E\left[Q_l\right]} \phi_k(t)\psi_l(\tau)\right)^2dtd\tau\right)^{\frac 12}\right]\nonumber
			\\
			&=
			\E\left[\left(\sum_{l=1}^E\sum_{k=1}^E\sum_{q=1}^E\sum_{r=1}^E  \frac{ U_{l,k}^mU_{q,r}^m  \delta_{l,q}\delta_{k,r}}{\E\left[Q_l\right]\E\left[Q_q\right]} \right)^{\frac 12}\right]\nonumber
			\\
			&=\E\left[\left(\sum_{l=1}^E \sum_{k=1}^E   \frac{\left(U_{l,k}^m\right)^2}{\E\left[Q_l\right]^2}  \right)^{\frac{1}{2}}\right]\nonumber
			\\
			&\le 
			\sum_{l=1}^E \sum_{k=1}^E   \frac{\E\left[\left|\gamma \chi_l^{A}Z_k^{A}+\chi_l^{O}Z_k^{O}\right|\right]}{\E\left[Q_l\right]}
			\nonumber
			\\
			&\le \sum_{l=1}^E \sum_{k=1}^E   \frac{\gamma\E\left[(\chi_l^A)^2 \right]^{\frac12}\E\left[(Z_k^A)^2 \right]^{\frac12}+\E\left[(\chi_l^O)^2 \right]^{\frac12}\E\left[(Z_k^O)^2 \right]^{\frac12}}{\E\left[Q_l\right]}
			<\infty.
		\end{align}
		By the Banach space version of the law of large numbers (see for instance \cite{bosq}, Theorem 2.4) it follows that for every $E\in\N$, $ \n \hat{\beta}_{n,E,(1)}-\beta_E  \n_{L^2([T_1,T_2]^2)}\xrightarrow{a.s.}0$. Let $f_n(\omega)=\n \frac{1}{n}\sum_{m=1}^n\beta^{(m),E}-\beta_E \n$ and $g_n(\omega)=\frac{1}{n}\sum_{m=1}^n\n \beta^{(m),E}\n +\n \beta_E \n$. Then $0\le f_n\le g_n$, $f_n\xrightarrow{a.s.}0$, (by the scalar law of large numbers) $g_n\xrightarrow{a.s.} G=\E\left[\n \beta^{(1),E}\n \right]+\n \beta_E \n$ and $\E\left[g_n\right]= \E\left[G\right]$. It now follows from Pratt's lemma that $\E\left[  \n \hat{\beta}_{n,E,(1)}-\beta_E  \n_{L^2([T_1,T_2]^2)}\ \right]\to 0$, for every $E\in\N$. If we now let $\{e(n)\}_{n\in\N}$ be a sequence in $\N$ tending to infinity, such that $e(n)\le E_1(n)$ where $E_1(1)=1$, and for $n>1$, 
		\begin{align*}
			&E_1(n)=(E_1(n-1)+1)1_{\E\left[\n \hat{\beta}_{n,E_1(n-1)+1,(1)}-\beta_{E_1(n-1)+1}  \n_{L^2([T_1,T_2]^2)}\right]<\alpha_{E_1(n-1)+1} }
			\\
			+&E_1(n-1)1_{\E\left[\n \hat{\beta}_{n,E_1(n-1)+1,(1)}-\beta_{E_1(n-1)+1}   \n_{L^2([T_1,T_2]^2)}\right]\ge \alpha_{E_1(n-1)+1} },
		\end{align*}
		for some non-increasing sequence of positive numbers $\{\alpha_n\}_{n\in\N}$ that tend to zero. With this construction it follows that if we let $\hat{B}_{n,(1)}=\frac{1}{n}\sum_{m=1}^n B^{(m),e(n)}$ and assume $e(n)\le E_1(n)$ then 
		\begin{align}\label{B1}
			\E\left[\n \hat{\beta}_{n,(1)}-\beta  \n_{L^2([T_1,T_2]^2)}^2\right]^{\frac12}
			&\le \sqrt{2}\E\left[\n \hat{\beta}_{n,(1)}-\beta_{e(n)}  \n_{L^2([T_1,T_2])}^2\right]^{\frac12} + \sqrt{2}\n \beta-\beta_{e(n)}  \n_{L^2([T_1,T_2])}\nonumber
			\\
			&\le \sqrt{2}\alpha_{e(n)} + \sqrt{2\sum_{k=e(n)+1}^\infty\sum_{l=e(n)+1}^\infty \frac{\E\left[U_{l,k}^m\right]^2}{\E\left[Q_l\right]^2}},
		\end{align}
		which converges to zero. By the Markov inequality $\n \hat{\beta}_{n,(1)}-\beta  \n_{L^2([T_1,T_2]^2)}\xrightarrow{\P} 0$.
		\\
		\textbf{Step 2: Replace the population denominators.}\\
		Let 
		$$\beta^{(m),2}=\sum_{k=1}^{e(n)} \sum_{l=1}^{e(n)} \frac{U_{l,k}^m}{\frac{1}{n}\sum_{m=1}^nW_l^m} \phi_k\otimes\phi_l ,$$ 
		$\hat{\beta}_{n,(2)}=\frac{1}{n}\sum_{m=1}^n \beta^{(m),2}$ and 
		$$A_{n,1}=\left\{ \frac{\E\left[Q_l\right]^2}{2} \le \left(\frac{1}{n}\sum_{m=1}^nW_l^m\right)^2\le 2\E\left[Q_l\right]^2, l=1,..,e(n)\right\}. $$
		Let $E_2(n)$ be a sequence tending to infinity such that if $e(n)\le E_1(n)\vee E_2(n)$, $\lim_{n\to\infty}\P\left(A_{n,1}\right)=1$.
		
		\begin{align}\label{LLNMoment}
			&\E\left[\n \hat{\beta}_{n,(2)}-\hat{\beta}_{n,(1)} \n_{L^2([T_1,T_2]^2)}1_{A_{n,1}}\right]\nonumber
			\\
			&\le
			\E\left[\frac{1}{n}\sum_{m=1}^n \n B^{(m),1}-\beta^{(m),2} \n_{L^2([T_1,T_2]^2)}1_{A_{n,1}}\right]
			\nonumber
			\\
			&=
			\E\left[\left(  \sum_{l=1}^{e(n)} \sum_{k=1}^{e(n)} (U_{l,k}^m)^2\left|\frac{1}{\E\left[Q_l\right]^2}-\frac{1}{\left(\frac{1}{n}\sum_{m=1}^n W_l^m\right)^2} \right| \right)^{\frac{1}{2}}1_{A_{n,1}}\right]
			\nonumber
			\\
			&=
			\E\left[\left(  \sum_{l=1}^{e(n)} \sum_{k=1}^{e(n)} (U_{l,k}m)^2 \frac{1}{\E\left[Q_l\right]^2\left(\frac{1}{n}\sum_{m=1}^n W_l^m\right)^2}\left|\E\left[Q_l\right]^2-\left(\frac{1}{n}\sum_{m=1}^n W_l^m\right)^2 \right| \right)^{\frac{1}{2}}1_{A_{n,1}}\right]
		\end{align}
		
		due to independence we have 
		\begin{align*}
			&\E\left[\left(  \sum_{l=1}^{e(n)} \sum_{k=1}^{e(n)} (U_{l,k}^m)^2 \frac{1}{\E\left[Q_l\right]^2\frac{1}{n}\sum_{m=1}^n W_l^m}\left|\E\left[Q_l\right]^2-\left(\frac{1}{n}\sum_{m=1}^n W_l^m \right)^2\right| \right)^{\frac{1}{2}}1_{A_{n,1}}\right]
			\\
			&\le
			\sum_{l=1}^{e(n)} \sum_{k=1}^{e(n)}\frac{\E\left[\left|U_{l,k}^m\right|\right]}{\E\left[Q_l\right]}\E\left[  \left|\frac{\E\left[Q_l\right]^2-\left(\frac{1}{n}\sum_{m=1}^n W_l^m\right)^2}{\frac{1}{n}\sum_{m=1}^n W_l^m}  \right|^{\frac{1}{2}}1_{A_{n,1}} \right]
			\\
			&\le
			\sum_{l=1}^{e(n)} \sum_{k=1}^{e(n)}\frac{\E\left[\left|U_{l,k}^m\right|\right]}{\E\left[Q_l\right]}\E\left[ \left| \frac{\E\left[Q_l\right]^2-\left(\frac{1}{n}\sum_{m=1}^n W_l^m\right)^2}{\frac{1}{n}\sum_{m=1}^n W_l^m}\right| 1_{A_{n,1}}\right]^{\frac{1}{2}},
		\end{align*}
		Noting that
		\begin{align*}
			\E\left[  \left| \frac{\E\left[Q_l\right]^2-\frac{1}{n}\sum_{m=1}^n W_l^m}{\frac{1}{n}\sum_{m=1}^n W_l^m}  \right|1_{A_{n,1}}\right]^{\frac 12}
			&\le
			\sqrt{2}\frac{\E\left[  \left| \E\left[Q_l\right]^2-\left(\frac{1}{n}\sum_{m=1}^n W_l^m\right)^2 \right|1_{A_{n,1}} \right]^{\frac 12}}{\E\left[Q_l\right]}
			:=a_l(n),
		\end{align*}
		where, by dominated convergence, $\lim_{n\to\infty}a_l(n)=0$ for every $l$.
		This implies
		\begin{align*}
			\E\left[\n \hat{\beta}_{n,(2)}-\hat{\beta}_{n,(1)} \n_{L^2([T_1,T_2]^2)}1_{A_{n,1}}\right]
			&\le 
			\sum_{l=1}^{e(n)} \sum_{k=1}^{e(n)}\frac{\E\left[\left|\gamma\chi^A_l Z^A_k+(1-\gamma)\chi^O_l Z^O_k\right|\right]}{\E\left[Q_l\right]}\sqrt{2}a_l(n).
		\end{align*}
		We now let $E_3(1)=1$, and for $n>1$, 
		\begin{align*}
			E_3(n)&=(E_3(n-1)+1)1_{\sum_{l=1}^{E_3(n-1)+1} \sum_{k=1}^{E_3(n-1)+1}\frac{\E\left[\left|\gamma\chi^A_l Z^A_k+(1-\gamma)\chi^O_l Z^O_k\right|\right]}{\E\left[Q_l\right]}\sqrt{2}a_l(n)<\alpha_n }
			\\
			&+E_3(n-1)1_{\sum_{l=1}^{E_3(n-1)+1} \sum_{k=1}^{E_3(n-1)+1}\frac{\E\left[\left|\gamma\chi^A_l Z^A_k+(1-\gamma)\chi^O_l Z^O_k\right|\right]}{\E\left[Q_l\right]}\sqrt{2}a_l(n)\ge\alpha_n }.
		\end{align*}
It follows from Markov's inequality that if $e(n)\le E_1(n)\vee E_2(n)\vee E_3(n)$
		\begin{align*}
			\P\left( \n \hat{\beta}_{n,(2)}-\hat{\beta}_{n,(1)} \n_{L^2([T_1,T_2]^2)}\ge \epsilon\right)
			&\le
			\frac{1}{\epsilon} \sum_{l=1}^{e(n)} \sum_{k=1}^{e(n)}\frac{\E\left[\left|\gamma\chi^A_l Z^A_k+(1-\gamma)\chi^O_l Z^O_k\right|\right]}{\E\left[Q_l\right]}\sqrt{2}a_l(n) + \P\left( A_{n,1}^c \right)
			\\
			&\le \alpha_n+\P\left( A_{n,1}^c \right),
		\end{align*}
		which converges to zero for every $\epsilon>0$.
		\\
		\textbf{Step 3: Discretize the sample curves for the numerators.}\\
		Recall (from Section \ref{sec:estimation}) the definitions of $C_l^{A,m,n}$ and $D_k^{A,m,d,n}$. Let 
		\begin{itemize}
			\item[] $U_{l,k}^{m,n}=\gamma C_l^{A,m,n}D_k^{A,m,n}+(1-\gamma)C_l^{O,m,n}D_k^{O,m,n}$,
			\item[] $\beta^{(m),3}=\sum_{k=1}^{e(n)} \sum_{l=1}^{e(n)} \frac{U_{l,k}^{m,n}}{\frac{1}{n}\sum_{m=1}^n W_l^m}\phi_k\otimes\psi_l $ and
			\item[] $\hat{\beta}_{n,(3)}=\frac{1}{n}\sum_{m=1}^n \beta^{(m),3}$.
		\end{itemize}
		On the set $A_{n,1}$
		\begin{align*}
			&\n \hat{\beta}_{n,(3)}-\hat{\beta}_{n,(2)} \n_{L^2([T_1,T_2]^2)} 
			\\
			&\le  
			\frac{1}{n}\sum_{m=1}^n  \n \beta^{(m),3}-\beta^{(m),2} \n_{L^2([T_1,T_2]^2)}
			\\
			&=
			\frac{1}{n}\sum_{m=1}^n\left(\sum_{l=1}^{e(n)}\sum_{k=1}^{e(n)}\sum_{q=1}^{e(n)}\sum_{r=1}^{e(n)}  \frac{\left(U_{l,k}^m-U_{l,k}^{m,n}\right)\left(U_{r,q}-U_{r,q}^n\right) }{\frac{1}{n}\sum_{m=1}^n W_l^m\frac{1}{n}\sum_{m=1}^n W_r^m} \int_{[T_1,T_2]} \phi_k(t)\phi_r(t)dt\int_{[T_1,T_2]} \psi_l(\tau)\psi_q(\tau)d\tau\right)^{\frac 12}
			\\
			&=
			\frac{1}{n}\sum_{m=1}^n\left(\sum_{k=1}^{e(n)} \sum_{l=1}^{e(n)}  \frac{\left(U_{l,k}^m-U_{l,k}^{m,n}\right)^2}{\left(\frac{1}{n}\sum_{m=1}^n W_l^m\right)^2}\right)^{\frac 12}
			\\
			&=
			\frac{1}{n}\sum_{m=1}^n\left(\sum_{k=1}^{e(n)} \sum_{l=1}^{e(n)}\frac{1}{\frac{1}{n}\sum_{m=1}^n W_l^m}  \left(  \gamma C_l^{A,m,n}(D_k^{A,m,n}-D_k^{A,m})+ \gamma D_k^{A,m}(C_l^{A,m,n}-C_l^{A,m})
			\right.\right.
			\\
			&\left.\left.+
			(1-\gamma)C_l^{O,m,n}(D_k^{O,m,n}-D_k^{O,m})+ (1-\gamma)D_k^{O,m}(C_l^{O,m,n}-C_l^{O,m})\right)^2\right)^{\frac 12}
			\\
			&\le
			\frac{1}{n}\sum_{m=1}^n\left(\sum_{k=1}^{e(n)} \sum_{l=1}^{e(n)} \frac{1}{\E\left[Q_l\right]}\left(\left(  8\left(\gamma C_l^{A,m,n}(D_k^{A,m,n}-D_k^{A,m})\right)^2
			\right.\right.\right.
			\\
			&\left.\left.\left.+ 8\left(\gamma D_k^{A,m}(C_l^{A,m,n}-C_l^{A,m})\right)^2
			+
			8\left((1-\gamma)C_l^{O,m,n}(D_k^{O,m,n}-D_k^{O,m})\right)^2
			\right.\right.\right.
			\\
			&\left.\left.\left.+
			8\left((1-\gamma)D_k^{O,m}(C_l^{O,m,n}-C_l^{O,m})\right)^2\right)\right)\right)^{\frac 12}
			.
		\end{align*}
		
		Note that by Cauchy-Schwarz inequality
		\begin{align*}
			(C_l^{A,m,n}-C_l^{A,m})^2&=\left(\int_{[T_1,T_2]}\left(P_{\Pi_n}(X^{A,m},t)-X^A_t\right)\psi_l(t)dt\right)^2
			\\
			&\le
			\left(\int_{[T_1,T_2]}\left|P_{\Pi_n}(X^{A,m},t)-X^A_t\right|\left|\psi_l(t)\right|dt\right)^2
			\\
			&\le
			\int_{[T_1,T_2]}\left|P_{\Pi_n}(X^{A,m},t)-X^A_t\right|^2dt\int_{[T_1,T_2]}\left|\psi_l(t)\right|^2dt
			\le Td_n^2
		\end{align*}
		and similarly
		\begin{align*}
			(D_k^{A,m,n}-D_k^{A,m})^2
			&=
			\left(\int_{[T_1,T_2]}\left(P_{\Pi_n}(Y^{A,m},s)-Y^A_s\right)\psi_l(s)ds\right)^2
			\\
			&\le
			\int_{[T_1,T_2]}\left|P_{\Pi_n}(Y^{A,m},s)-Y^A_s\right|^2ds\int_{[T_1,T_2]}\left|\psi_l(s)\right|^2ds 
			\le Td_n^2.
		\end{align*}
		Furthermore
		\begin{align*}
			(C_l^{A,m,n})^2
			&=
			\left(\int_{[T_1,T_2]} P_{\Pi_n}(X^{A,m},t)\psi_l(t)dt\right)^2
			\\
			&\le
			\left(\int_{[T_1,T_2]}\left|P_{\Pi_n}(X^{A,m},t)\right|\left|\psi_l(t)\right|dt\right)^2
			\\
			&\le
			\int_{[T_1,T_2]}\left|P_{\Pi_n}(X^{A,m},t)\right|^2dt
			\\
			&\le
			2\int_{[T_1,T_2]}\left|P_{\Pi_n}(X^{A,m},t)-X_t^{A,m}\right|^2dt
			+
			2\int_{[T_1,T_2]}\left|X_t^{A,m}\right|^2dt
			\le
			2Td_n^2
			+
			2\int_{[T_1,T_2]}\left|X_t^{A,m}\right|^2dt.
		\end{align*}
		Letting 
		$$M_n=\max_{1\le m \le e(n)} \int_{[T_1,T_2]} (Y_s^{A,m})^2ds+\int_{[T_1,T_2]} (Y_s^{O,m})^2ds+\int_{[T_1,T_2]} (X_s^{A,m})^2ds+\int_{[T_1,T_2]} (X_s^{O,m})^2ds,$$
		we get on the set $A_{n,1}$,
		\begin{align*}
			\n \hat{\beta}_{n,(3)}-\hat{\beta}_{n,(2)} \n_{L^2([T_1,T_2]^2)} 
			&\le 
			\frac{1}{n}\sum_{m=1}^n \left(2\frac{e(n)^2(8(1-\gamma)^2+8\gamma^2)Td_n^2(2Td_n^2+2M_n)}{\min_{1\le l\le e(n)}\E\left[Q_l\right]} \right)^{\frac 12}
			\\
			&\le
			2^{\frac 52}\sqrt{2+3\gamma^2} Te(n)d_n\frac{\left(\sqrt{M_n}+\sqrt{T}d_n\right)}{\sqrt{ \min_{1\le l\le e(n)}\E\left[Q_l\right]}}.
		\end{align*}
		Therefore
		\begin{align*}
			&\P\left(\right\{\n \hat{\beta}_{n,(3)}-\hat{\beta}_{n,(2)} \n_{L^2([T_1,T_2]^2)} \ge \epsilon \left\}\cap A_{n,1}\right)
			\le
			\P\left(M_n \ge \left(\frac{\epsilon\sqrt{\min_{1\le l\le e(n)}\E\left[Q_l\right]}}{2^{\frac 52}\sqrt{2+3\gamma^2} Te(n)d_n}-\sqrt{T}d_n\right)^2 \right)
			\\
			&\le
			1- F_{\n X^{A} \n_{L^2([T_1,T_2]}^2+\n X^{O} \n_{L^2([T_1,T_2]}^2+\n Y^{A} \n_{L^2([T_1,T_2]}^2+\n Y^{O} \n_{L^2([T_1,T_2]}^2} \left(\left(\frac{\epsilon\sqrt{\min_{1\le l\le e(n)}\E\left[Q_l\right]}}{2^{\frac 52}\sqrt{2+3\gamma^2} Te(n)d_n}-\sqrt{T}d_n\right)^2\right)^{e(n)}
			\\
			&\le
			1- \left(1-\P\left(\n X^{A} \n_{L^2([T_1,T_2]}^2+\n X^{O} \n_{L^2([T_1,T_2]}^2+\n Y^{A} \n_{L^2([T_1,T_2]}^2+\n Y^{O} \n_{L^2([T_1,T_2]}^2
			\right.\right.
			\\
			&\left.\left.\ge \left(\frac{\epsilon\sqrt{\min_{1\le l\le e(n)}\E\left[Q_l\right]}}{2^{\frac 52}\sqrt{2+3\gamma^2} Te(n)d_n}-\sqrt{T}d_n\right)^2\right)\right)^{e(n)}
			\\
			&\le
			1- \left(1-\frac{\E\left[ \n X^{A} \n_{L^2([T_1,T_2]}^2 +\n X^{O} \n_{L^2([T_1,T_2]}^2+\n Y^{A} \n_{L^2([T_1,T_2]}^2+\n Y^{O} \n_{L^2([T_1,T_2]}^2  \right]}{\left(\frac{\epsilon\sqrt{\min_{1\le l\le e(n)}\E\left[Q_l\right]}}{2^{\frac 52}\sqrt{2+3\gamma^2} Te(n)d_n}-\sqrt{T}d_n\right)^2}\right)^{e(n)}
			\\
			&\le
			\frac{e(n)\E\left[ \n X^{A} \n_{L^2([T_1,T_2]}^2 +\n X^{O} \n_{L^2([T_1,T_2]}^2+\n Y^{A} \n_{L^2([T_1,T_2]}^2+\n Y^{O} \n_{L^2([T_1,T_2]}^2 \right]}{\left(\frac{\epsilon\sqrt{\min_{1\le l\le e(n)}\E\left[Q_l\right]}}{2^{\frac 52}\sqrt{2+3\gamma^2} Te(n)d_n}-\sqrt{T}d_n\right)^2}
			\\
			&\le
			\frac{e(n)^3d_n^2\E\left[ \n X^{A} \n_{L^2([T_1,T_2]}^2 +\n X^{O} \n_{L^2([T_1,T_2]}^2+\n Y^{A} \n_{L^2([T_1,T_2]}^2+\n Y^{O} \n_{L^2([T_1,T_2]}^2 \right]}{\frac{\epsilon^2\min_{1\le l\le e(n)}\E\left[Q_l\right]}{2^{5}T^2(2+3\gamma^2)}-e(n)^2d_n^2 },
		\end{align*}
		so if $e(n)\le  d_n^{\eta}$, where $\eta\in\left(-\frac 23,0\right)$ then the above expression will converge to zero. Set $E_4(n)=d_n^\eta$ with $\eta$ as prescribed.
		\\
		\textbf{Step 4: Discretize the sample curves for the denominators.}\\
		Let  
		\begin{itemize}
		\item[] $W_l^{m,n}=\gamma (C_l'^{A,m,n})^2+ (1-\gamma)(C_l'^{O,m,n})^2$
			\item[] $\beta^{(m),4}=\sum_{k=1}^{e(n)} \sum_{l=1}^{e(n)} \frac{U_{l,k}^{m,n}}{\frac{1}{n}\sum_{m=1}^n W_l^{m,n}}\phi_k\otimes\psi_l $ and
			\item[] $\hat{\beta}_{n,(4)}=\frac{1}{n}\sum_{m=1}^n \beta^{(m),4}=\hat{\beta}_n$.
		\end{itemize}
		Noting that \\$\left| C_l'^{A,m,n}-C_l'^{A,m} \right|\le d_n$ implies 
		\begin{align*}
			\left| \frac{1}{n}\sum_{m=1}^n \left(C_l'^{A,m,n}\right)^2-\frac{1}{n}\sum_{m=1}^n \left(C_l'^{A,m}\right)^2 \right|
			&\le
			\frac{1}{n}\sum_{m=1}^n\left|  C_l'^{A,m,n} -C_l'^{A,m} \right|\left|  C_l'^{A,m,n} +C_l'^{A,m} \right|
			\\
			&\le
			d_n(1+d_n)\frac{1}{n}\sum_{m=1}^n\left|C_l'^{A,m} \right|
			\le
			d_n(1+d_n)\frac{1}{n}\left(\sum_{m=1}^n\left(C_l'^{A,m} \right)^2\right)^{\frac 12}
		\end{align*}
with an analogous inequality being valid in the observational setting. This leads to
		\begin{align*}
			&\left| \frac{1}{n}\sum_{l=m}^n \left(\gamma \left(C_l'^{A,m,n}\right)^2+(1-\gamma)\left(C_l'^{O,m,n}\right)^2 \right)-\frac{1}{n}\sum_{m=1}^n\left( \gamma\left(C_l'^{A,m}\right)^2+(1-\gamma)\left(C_l'^{O,m}\right)^2\right) \right|
			\\
			&\le
			\gamma d_n(1+d_n)\frac{1}{n}\left(\sum_{m=1}^n\left(C_l'^{A,m} \right)^2\right)^{\frac 12}+(1-\gamma)d_n(1+d_n)\frac{1}{n}\left(\sum_{m=1}^n\left(C_l'^{O,m} \right)^2\right)^{\frac 12}.
		\end{align*}
Moreover, on $A_{n,1}$
		\begin{align*}
			\left(C_l'^{A,m,n}\right)^2
			&\ge
			\left(\left|C_l'^{A,m,n}\right|-d_n\right)^2
			\\
			&\ge
			\left(C_l'^{A,m,n}\right)^2+d_n^2-4d_n\E\left[(\chi_l^A)^2\right]^{\frac 12}
			\ge \frac{1}{2}\E\left[(\chi_l^A)^2\right]+d_n^2-4d_n\E\left[(\chi_l^A)^2\right]^{\frac 12}
		\end{align*}
		and
		\begin{align*}
			\left(C_l'^{O,m,n}\right)^2
			\ge \frac{1}{2}\E\left[(\chi_l^O)^2\right]+d_n^2-4d_n\E\left[(\chi_l^O)^2\right]^{\frac 12}.
		\end{align*}
		Therefore, letting 
		$$R_{l,n}=\E\left[Q_l\right]\left(\gamma\left(\frac{1}{2}\E\left[(\chi_l^A)^2\right]+d_n^2-4d_n\E\left[(\chi_l^A)^2\right]^{\frac 12}\right)+(1-\gamma)\left(\frac{1}{2}\E\left[(\chi_l^O)^2\right]+d_n^2-4d_n\E\left[(\chi_l^O)^2\right]^{\frac 12}\right)\right)$$
		we have $\frac{1}{n}\sum_{m=1}^n W_l^m\frac{1}{n}\sum_{m=1}^n W_l^{m,n}\ge R_{l,n}$. On $A_{n,1}$,
		\begin{align*}
			\n \hat{\beta}_n-\hat{\beta}_{n,(3)} \n_{L^2([T_1,T_2]^2)} 
			&\le  
			\frac{1}{n}\sum_{m=1}^n\left(\sum_{k=1}^{e(n)} \sum_{l=1}^{e(n)}  (U_{l,k}^{m,n})^2 \left( \frac{1}{\frac{1}{n}\sum_{m=1}^n W_l^m} -  \frac{1}{\frac{1}{n}\sum_{m=1}^n W_l^{m,n}} \right)^2\right)^{\frac 12}
			\\
			&=
			\frac{1}{n}\sum_{m=1}^n\left(\sum_{k=1}^{e(n)} \sum_{l=1}^{e(n)} \frac{(U_{l,k}^{m,n})^2}{\frac{1}{n}\sum_{m=1}^n W_l^m\frac{1}{n}\sum_{m=1}^n W_l^{m,n}} \left( \frac{1}{n}\sum_{m=1}^n W_l^m -  \frac{1}{n}\sum_{m=1}^n W_l^{m,n} \right)^2\right)^{\frac 12}
			\\
			&\le
			\frac{1}{n}\sum_{m=1}^n\left(\sum_{k=1}^{e(n)} \sum_{l=1}^{e(n)} \frac{2(U_{l,k}^{m,n})^2}{R_{l,n}}  d_n^2(1+d_n)^2\frac{1}{n^2}\sum_{m=1}^nW_l^m\right)^{\frac 12}.
		\end{align*}
		This leads to,
		\begin{align*}
			\E\left[\n \hat{\beta}_{n}-\hat{\beta}_{n,(3)} \n_{L^2([T_1,T_2]^2)} 1_{A_{n,1}}\right]
			&\le
			\sum_{k=1}^{e(n)} \sum_{l=1}^{e(n)} \frac{1}{R_{l,n}}2\E\left[\left|U_{l,k}^{m,n} \right|  \right]\E\left[ d_n(1+d_n)\frac{1}{n}\left(\sum_{m=1}^nW_l^m\right)^{\frac 12}\right]
			\\
			&\le
			\sum_{k=1}^{e(n)} \sum_{l=1}^{e(n)} \frac{1}{R_{l,n}}\left( 
			2\E\left[\gamma\left( \left|C_l^{A,m}\right|+Td_n \right)\left(\left|D_k^{A,m}\right|+Td_n \right)
			\right.\right.
			\\
			&\left.\left.+(1-\gamma)\left( \left|C_l^{O,m}\right|+Td_n \right)\left(\left|D_k^{O,m}\right|+Td_n \right)  \right]\right)d_n(1+d_n)\E\left[ Q_l\right]^{\frac12}
			\\
			&\le
			d_n(1+d_n)\sum_{k=1}^{e(n)} \sum_{l=1}^{e(n)} \frac{1}{R_{l,n}}\left(
			2\left(T^2d_n^2+T\gamma\E\left[\left|\chi_l^A\right|+\left|Z_k^A\right|  \right]
			\right.\right.
			\\
			&\left.\left.+\gamma\E\left[(\chi_l^A)^2  \right]^{\frac12}\E\left[(Z_k^A)^2  \right]^{\frac12}
+
T(1-\gamma)\E\left[\left|\chi_l^O\right|+\left|Z_k^O\right|  \right]
\right.\right.
\\
&\left.\left.+
(1-\gamma)\E\left[(\chi_l^O)^2  \right]^{\frac12}\E\left[(D_k^O)^2  \right]^{\frac12} \right)\right)\E\left[ Q_l\right]^{\frac12}
		\end{align*}
		letting $g_4$ be a strictly increasing function such that 
		\begin{align*}
			\lim_{E\to\infty}\frac{1}{g_4(E)}\sum_{k=1}^{E} \sum_{l=1}^{E} &\frac{1}{R_{l,n}}\left(
			2\left(T^2d_n^2+T\gamma\E\left[\left|\chi_l^A\right|+\left|Z_k^A\right|  \right]
			+\gamma\E\left[(\chi_l^A)^2  \right]^{\frac12}\E\left[(Z_k^A)^2  \right]^{\frac12}
+
T(1-\gamma)\E\left[\left|\chi_l^O\right|+\left|Z_k^O\right|  \right]
\right.\right.
\\
&\left.\left.+
(1-\gamma)\E\left[(\chi_l^O)^2  \right]^{\frac12}\E\left[(D_k^O)^2  \right]^{\frac12} \right)\right)\E\left[ Q_l\right]^{\frac12}=0,
		\end{align*}
		letting $$e(n)\le E_5(n)= g_4^{-1}\left(\frac{1}{d_n(1+d_n)}\right)$$ is sufficient to ensure that $\n \hat{B}_{n,(4)}-\hat{B}_{n,(3)} \n_{L^2([T_1,T_2]^2)}\xrightarrow{\P}0$. 
\\
\textbf{Step 5: Replace the population eigenfunctions with the truncated estimators.}\\
Let $B^{(m),5}=\sum_{k=1}^{e(n)} \sum_{l=1}^{e(n)}  \frac{U_{l,k}^{m,n}}{\frac{1}{n}\sum_{m=1}^n W_l^{m,n}}\phi_k\otimes T_M\left(\hat{\psi}_{l,n}\right) $ and $\hat{B}_{n,(5)}=\frac{1}{n}\sum_{m=1}^n B^{(m),5}(t,\tau)$. We then have the following bound

\begin{align}\label{B2B}
\n B^{(m),5}-B^{(m),4} \n_{L^2([T_1,T_2]^2)}
&\le
\sum_{k=1}^{e(n)} \sum_{l=1}^{e(n)}   \frac{\left|U_{l,k}^{m,n} \right|}{\frac{1}{n}\sum_{m=1}^n W_l^{m,n}} \n   \phi_k\otimes \left(T_M\left(\hat{\psi}_{l,n}\right)-\psi_l\right)\n_{L^2([T_1,T_2]^2)}\nonumber
\\
&\le
\gamma\sum_{k=1}^{e(n)} \left|D_k^{A,m,n} \right|\n \phi_k\n_{L^2([T_1,T_2]^2)}\sum_{l=1}^{e(n)}  \frac{\left|C_l^{A,m,n}\right|}{\frac{1}{n}\sum_{m=1}^n W_l^{m,n}}\n T_M\left(\hat{\psi}_{l,n}\right)-\psi_l\n_{L^2([T_1,T_2]^2)}\nonumber
\\
&+
|1-\gamma|\sum_{k=1}^{e(n)} \left|D_k^{O,m,n} \right|\n \phi_k\n_{L^2([T_1,T_2]^2)}\sum_{l=1}^{e(n)}  \frac{\left|C_l^{O,m,n}\right|}{\frac{1}{n}\sum_{m=1}^n W_l^{m,n}}\n T_M\left(\hat{\psi}_{l,n}\right)-\psi_l\n_{L^2([T_1,T_2]^2)}\nonumber
\\
&=(1+\gamma)
\sum_{k=1}^{e(n)} \left(\left|D_k^{A,m,n} \right|+\left|D_k^{O,m,n} \right|\right)\sum_{l=1}^{e(n)} \frac{\left|C_l^{A,m,n}\right|+\left|C_l^{O,m,n}\right|}{\frac{1}{n}\sum_{m=1}^n W_l^{m,n}}\n T_M\left(\hat{\psi}_{l,n}\right)-\psi_l\n_{L^2([T_1,T_2]^2)}.
\end{align}
By the Cauchy-Schwarz inequality
\begin{align}\label{Bn4}
&\E\left[\n \hat{B}_{n,(5)}-\hat{B}_{n,(4)}\n 1_{A_{n,1}}\right]\nonumber
\\
&\le
\frac{1}{n}\sum_{m=1}^n\E\left[\n B^{(m),5}-B^{(m),4} \n_{L^2([T_1,T_2]^2)}1_{A_{n,1}}\right]\nonumber
\\
&\le
\frac{1}{n}\sum_{m=1}^n\E\left[ (1+\gamma)
\sum_{k=1}^{e(n)} \left(\left|D_k^{A,m,n} \right|+\left|D_k^{O,m,n} \right|\right)\sum_{l=1}^{e(n)} \frac{\left|C_l^{A,m,n}\right|+\left|C_l^{O,m,n}\right|}{\frac{1}{n}\sum_{m=1}^n W_l^{m,n}}\n T_M\left(\hat{\psi}_{l,n}\right)-\psi_l\n_{L^2([T_1,T_2]^2)} 1_{A_{n,1}}\right]
\nonumber
\\
&\le
(1+\gamma)\sum_{l=1}^{e(n)}\sum_{k=1}^{e(n)}\E\left[ 2\left|D_k^{A,m,n} \right|^2+2\left|D_k^{O,m,n} \right|^2\right]^{\frac 12}
\E\left[\frac{2\left|C_l^{A,m,n}\right|^2+2\left|C_l^{O,m,n}\right|^2}{\frac{R_{l,n}}{\E\left[Q_l\right]}}\n T_M\left(\hat{\psi}_{l,n}\right)-\psi_l\n_{L^2([T_1,T_2])}^21_{A_{n,1}}\right]^{\frac 12}.
\end{align}
Letting $*$ be either $A$ or $O$, note that 
\begin{align*}
\left|C_l^{*,m,n}\right|
&\le
\int \left|P_{\Pi_n}(X^{*,m},t)\psi_l(t)\right|dt
\\
&\le
\n P_{\Pi_n}(X^{*,m},.)\n_{L^2([T_1,T_2])} \n \psi_l\n_{L^2([T_1,T_2])}
=\n P_{\Pi_n}(X^{*,m},.)\n_{L^2([T_1,T_2])}
\end{align*}
It is straight-forward to show that
$\n P_{\Pi_n}(X^{*,m},.)-X^{*,m}\n_{L^2([T_1,T_2])}^2
\le
Td_n^2$
and therefore
\begin{align*}
\n P_{\Pi_n}(X^{*,m},.)\n_{L^2([T_1,T_2])}
&\le
\n P_{\Pi_n}(X^{*,m},.)-X\n_{L^2([T_1,T_2])} + \n X^{*,m}\n_{L^2([T_1,T_2])}
\\
&\le
\sqrt{T}d_n+\n X^{*,m}\n_{L^2([T_1,T_2])}.
\end{align*}
Similarly
\begin{align*}
\n P_{\Pi_n}(Y^{*,m},.)\n_{L^2([T_1,T_2])}
\le
\sqrt{T}d_n+\n Y^{*,m}\n_{L^2([T_1,T_2])}.
\end{align*}
This leads to
\begin{align*}
\left|C_l^{*,m,n}\right|^2
&=
\n P_{\Pi_n}(X^{*,m},.)\n_{L^2([T_1,T_2])}^2
\\
&\le
2\left(Td_n^2+\n X^{*,m}\n_{L^2([T_1,T_2])}^2\right)
\end{align*}
and similarly
\begin{align*}
\left|D_k^{*,m,n}\right|^2
\le
2\left(Td_n^2+\n Y^{*,m}\n_{L^2([T_1,T_2])}^2\right).
\end{align*}
Utilizing the independence of the $\{\psi_{l,n}\}_l$ from $\mathcal{F}_n$, we plug these findings back into \eqref{Bn4} and get
\begin{align}\label{Bn4_2}
&\E\left[\n \hat{B}_{n,(5)}-\hat{B}_{n,(4)}\n1_{A_{n,1}}\right]
\nonumber
\\
&\le
(1+\gamma)\sum_{l=1}^{e(n)}\sum_{k=1}^{e(n)}\sqrt{2}\left(2Td_n^2+\E\left[\n Y^{A,m}\n_{L^2([T_1,T_2])}^2+\n Y^{O,m}\n_{L^2([T_1,T_2])}^2\right]\right)^{\frac 12}\nonumber
\\
&\times
\E\left[ \frac{2\left(Td_n^2+\n X^{A,m}\n_{L^2([T_1,T_2])}^2+\n X^{O,m}\n_{L^2([T_1,T_2])}^2\right)\n T_M\left(\hat{\psi}_{l,n}\right)-\psi_l\n_{L^2([T_1,T_2])}^2}{\frac{R_{l,n}}{\E\left[Q_l\right]}}  \right]^{\frac 12}\nonumber
\\
&\le
(1+\gamma)e(n)\sum_{l=1}^{e(n)}\sqrt{2}\left(2Td_n^2+\E\left[\n Y^{A}\n_{L^2([T_1,T_2])}^2+\n Y^{O}\n_{L^2([T_1,T_2])}^2\right]\right)^{\frac 12}\sqrt{\frac{\E\left[Q_l\right]}{R_{l,n}}}
\nonumber
\\
&\times\E\left[2\left(Td_n^2+\n X^{A}\n_{L^2([T_1,T_2])}^2+\n X^{O}\n_{L^2([T_1,T_2])}^2\right)\sum_{l=1}^{e(n)}\n T_M\left(\hat{\psi}_{l,n}\right)-\psi_l\n_{L^2([T_1,T_2])}^2\right]^{\frac 12}\nonumber
\\
&=
(1+\gamma)2\left(Td_n^2+\E\left[\n Y^{A}\n_{L^2([T_1,T_2])}^2+\n Y^{O}\n_{L^2([T_1,T_2])}^2\right]\right)^{\frac 12}\E\left[Td_n^2+\n X^{A}\n_{L^2([T_1,T_2])}^2+\n X^{O}\n_{L^2([T_1,T_2])}^2\right]^{\frac 12}
\nonumber
\\
&\times
\sum_{l=1}^{e(n)}\sqrt{\frac{\E\left[Q_l\right]}{R_{l,n}}}e(n)\E\left[\sum_{l=1}^{e(n)}\n T_M\left(\hat{\psi}_{l,n}\right)-\psi_l\n_{L^2([T_1,T_2])}^2\right]^{\frac 12}\nonumber
\\
&\le
2(1+\gamma)\left(Td_n^2+\E\left[\n Y^{A}\n_{L^2([T_1,T_2])}^2+\n Y^{O}\n_{L^2([T_1,T_2])}^2\right]\right)^{\frac 12}\E\left[Td_n^2+\n X^{A}\n_{L^2([T_1,T_2])}^2+\n X^{O}\n_{L^2([T_1,T_2])}^2\right]^{\frac 12}
\nonumber
\\
&\times
\sum_{l=1}^{e(n)}\sqrt{\frac{\E\left[Q_l\right]}{R_{l,n}}}e(n)\sum_{l=1}^{e(n)}\E\left[\n T_M\left(\hat{\psi}_{l,n}\right)-\psi_l\n_{L^2([T_1,T_2])}^2\right]^{\frac 12}.
\end{align}
Since $e(n)\sum_{l=1}^{e(n)}\E\left[\n T_M\left(\hat{\psi}_{l,n}\right)-\psi_l\n_{L^2([T_1,T_2])}^2\right]^{\frac 12}$ converges to zero by assumption we may define $E_6(1)=1$ and for $n>1$
\begin{align*}
			E_6(n)&=(E_6(n-1)+1)1_{(E_6(n-1)+1)\sum_{l=1}^{E_6(n-1)+1} \sqrt{\frac{\E\left[Q_l\right]}{R_{l,n}}}
\E\left[\n T_M\left(\hat{\psi}_{l,n}\right)-\psi_l\n_{L^2([T_1,T_2])}^2\right]^{\frac 12}<\alpha_n }
			\\
			&+E_6(n-1)1_{
(E_6(n-1)+1)\sum_{l=1}^{E_6(n-1)+1}\sqrt{\frac{\E\left[Q_l\right]}{R_{l,n}}}\E\left[\n T_M\left(\hat{\psi}_{l,n}\right)-\psi_l\n_{L^2([T_1,T_2])}^2\right]^{\frac 12}\ge\alpha_n }.
		\end{align*}
It follows by Markov's inequality that $\n \hat{B}_{n,(5)}-\hat{B}_{n,(4)}\n\xrightarrow{\P}0$. 
\\
\textbf{Step 6: replace the numerators using the estimated eigenfunctions.}\\
Let 
\begin{itemize}
\item[] $\tilde{C}_l^{A,m,n}=\int P_{\Pi_n}(X^{A,m},t)\hat{\psi}_l(t)dt,$
\item[] $\tilde{C}_l^{O,m,n}=\int P_{\Pi_n}(X^{O,m},t)\hat{\psi}_l(t)dt,$
\item[]  $B^{(m),6}=\sum_{k=1}^{e(n)} \sum_{l=1}^{e(n)} \frac{\gamma\tilde{C}_l^{A,m,n}D_k^{A,m,n}+(1-\gamma)\tilde{C}_l^{O,m,n}D_k^{O,m,n}}{\frac{1}{n}\sum_{m=1}^n W_l^{m,n}} \phi_k\otimes T_M\left(\hat{\psi}_{l,n}\right) $
\end{itemize}
and $\hat{B}_{n,(6)}=\frac{1}{n}\sum_{m=1}^n B^{(m),6}(t,\tau)$. We then have the following bound
\begin{align}\label{B2B3}
&\n B^{(m),6}-B^{(m),5} \n_{L^2([T_1,T_2]^2)}\nonumber
\\
&\le
\sum_{k=1}^{e(n)} \sum_{l=1}^{e(n)} (1+\gamma)\frac{\left|C_l^{A,m,n}-\tilde{C}_l^{A,m,n}\right|+\left|C_l^{O,m,n}-\tilde{C}_l^{O,m,n}\right|}{\frac{1}{n}\sum_{p=1}^n W_l^{m,n}} \left(\left|D_k^{A,m,n} \right|+\left|D_k^{O,m,n} \right|\right)\n   \phi_k\otimes T_M\left(\hat{\psi}_{l,n}\right)\n_{L^2([T_1,T_2]^2)}\nonumber
\\
&\le
(1+\gamma)\sum_{k=1}^{e(n)} \left(\left|D_k^{A,m,n} \right|+\left|D_k^{O,m,n} \right|\right)\n \phi_k\n_{L^2([T_1,T_2]^2)}\sum_{l=1}^{e(n)} \frac{\left|C_l^{A,m,n}-\tilde{C}_l^{A,m,n}\right|+\left|C_l^{O,m,n}-\tilde{C}_l^{O,m,n}\right|}{\frac{1}{n}\sum_{p=1}^n W_l^{m,n}}  M\nonumber
\\
&=
M(1+\gamma)\sum_{k=1}^{e(n)} \left(\left|D_k^{A,m,n} \right|+\left|D_k^{O,m,n} \right|\right)\sum_{l=1}^{e(n)} \frac{ \left|C_l^{A,m,n}-\tilde{C}_l^{A,m,n}\right|+\left|C_l^{O,m,n}-\tilde{C}_l^{O,m,n}\right| }{\frac{1}{n}\sum_{p=1}^n W_l^{m,n}} .
\end{align}
As \begin{align*}
\n P_{\Pi_n}(X^{*,m},.)\n_{L^2([T_1,T_2])}
\le
\sqrt{T}d_n+\n X^{*,m}\n_{L^2([T_1,T_2])}.
\end{align*}We now note that
\begin{align}\label{tildediff}
\left|C_l^{*,m,n}-\tilde{C}_l^{*,m,n}\right|
&=
\left|\int P_{\Pi_n}(X^{*,m},t)\left(T_M\left(\hat{\psi}_{l,n}(t)\right)-\psi_l(t)\right)dt\right|\nonumber
\\
&\le
\n P_{\Pi_n}(X^{*,m},.)\n_{L^2([T_1,T_2])}
\n T_M\left(\hat{\psi}_{l,n}\right)-\psi_l\n_{L^2([T_1,T_2])}\nonumber
\\
&\le
\left(\sqrt{T}d_n+\n X^{*,m}\n_{L^2([T_1,T_2])}\right)\n T_M\left(\hat{\psi}_{l,n}\right)-\psi_l\n_{L^2([T_1,T_2])}
\end{align}
Due to independence of the eigenfunction estimators from $\mathcal{F}_n$
\begin{align*}
&\E\left[\n \hat{B}_{n,(6)} - \hat{B}_{n,(5)} \n 1_{A_{n,1}}\right]
\\
&\le
\frac{1}{n}\sum_{m=1}^n  \E\left[\ \n B^{(m),6}-B^{(m),5} \n_{L^2([T_1,T_2]^2)} 1_{A_{n,1}}\right]
\\
&\le
(1+\gamma)\frac{M}{n}\sum_{m=1}^n  \sum_{k=1}^{e(n)}\sum_{l=1}^{e(n)}\sqrt{\frac{\E\left[Q_l\right]}{R_{l,n}}}
\\
&\times\E\left[  \left(\left|D_k^{A,m,n}\right|+\left|D_k^{O,m,n} \right|\right)  \left(\sqrt{T}d_n+\n X^{A,m}\n_{L^2([T_1,T_2])}+\n X^{O,m}\n_{L^2([T_1,T_2])}\right)\n T_M\left(\hat{\psi}_{l,n}\right)-\psi_l\n_{L^2([T_1,T_2])}1_{A_{n,1}} \right]
\\
&\le 
2(1+\gamma)\frac{M}{n}\sum_{m=1}^n  \sum_{k=1}^{e(n)}\sum_{l=1}^{e(n)}\sqrt{\frac{\E\left[Q_l\right]}{R_{l,n}}}\E\left[ \left|D_k^{A,m,n} \right|^2+\left|D_k^{O,m,n} \right|^2\right]^{\frac 12}
\\
&\times\E\left[ \left(\sqrt{T}d_n+\n X^{A}\n_{L^2([T_1,T_2])}+\n X^{O}\n_{L^2([T_1,T_2])}\right)^2 \right]^{\frac12}\E\left[\n T_M\left(\hat{\psi}_{l,n}\right)-\psi_l\n_{L^2([T_1,T_2])}^2\right]^{\frac 12}
\\
&\le 
2(1+\gamma)M \left(Td_n^2+\E\left[\n Y^A\n_{L^2([T_1,T_2])}^2+\n Y^O\n_{L^2([T_1,T_2])}^2\right]\right)^{\frac 12}\left(2^2\left(T^2d_n^2+T\E\left[\n X^{A}\n_{L^2([T_1,T_2])}^2+\n X^{O}\n_{L^2([T_1,T_2])}^2\right]\right)\right)^{\frac 12}
\\
&\times e(n)\sum_{l=1}^{e(n)}\sqrt{\frac{\E\left[Q_l\right]}{R_{l,n}}}\E\left[\n T_M\left(\hat{\psi}_{l,n}\right)-\psi_l\n_{L^2([T_1,T_2])}^2\right]^{\frac 12}
\end{align*}
which converges to zero if $e(n)\le E_6(n)$. Therefore, again by the Markov inequality $\n \hat{B}_{n,(4)} - \hat{B}_{n,(5)}  \n_{L^2([T_1,T_2]^2)}\xrightarrow{\P}0$. 
\\	
\textbf{Step 7: Replace the denominators using the estimated eigenfunctions.}\\
Let 
\begin{itemize}
\item[] $\tilde{C}_l'^{A,m,n}=\langle P_{\Pi_n'}(X'^{A,m},.)\hat{\psi}_l \rangle_{L^2([T_1,T_2])^p},$
\item[] $\tilde{C}_l'^{O,m,n}=\langle P_{\Pi_n'}(X'^{O,m},.)\hat{\psi}_l \rangle_{L^2([T_1,T_2])^p},$
\item[] $\tilde{W}_l^{m,n}=\gamma (\tilde{C}_l'^{A,m,n})^2+ (1-\gamma)(\tilde{C}_l'^{O,m,n})^2$
\item[]  $B^{(m),7}=\sum_{k=1}^{e(n)} \sum_{l=1}^{e(n)} \frac{\gamma\tilde{C}_l^{A,m,n}D_ k^{A,m,n}+(1-\gamma)\tilde{C}_l^{O,m,n}D_ k^{O,m,n}\tilde{C}_l^{m,n}D_ k^{m,n}}{\frac{1}{n}\sum_{m=1}^n \tilde{W}_l^{m,n}} \phi_k\otimes T_M\left(\hat{\psi}_{l,n}\right) $ 
\end{itemize}

and $\hat{B}_{n,(7)}=\frac{1}{n}\sum_{m=1}^n B^{(m),7}(t,\tau)$. Similarly to \eqref{tildediff}, 
$$ |\tilde{C}'^{*,m,n}_l| \ge |C'^{*,m,n}_l|-\left(\sqrt{T}d_n+\n X^{*,m}\n_{L^2([0,T])}\right)\n T_M\left(\hat{\psi}_{l,n}\right)-\psi_l\n_{L^2([0,T])},$$
which is positive for all $l$, for sufficiently large $n$. Using that $a^2\ge b^2-2bc$ if $a\ge b-c \ge 0$, $a,b,c\ge 0$ and that the second term on the right-hand side above tends to zero, we may therefore define $E_7(n)$ such that for $1\le l\le E_7(n)$,
$ \E\left[|\tilde{C}'^{m,n}_l|^2 \right]\ge \E\left[\frac{|C'^{m,n}_l|^2}{2}\right]$ (note that this inequality is valid across all $m$ because of identical distribution). Next we note that,
\begin{align*}
\left|W_l^{m,n}-\tilde{W}_l^{m,n}\right|
\le
(1+\gamma)\left(\left|C_l'^{A,m,n}-\tilde{C}_l'^{A,m,n}\right|+\left|C_l'^{O,m,n}-\tilde{C}_l'^{O,m,n}\right|\right)\left(\left|\tilde{C}_l'^{A,m,n}\right|+\left|C_l'^{A,m,n}\right|+\left|\tilde{C}_l'^{O,m,n}\right|+\left|C_l'^{O,m,n}\right|\right)
\end{align*}
implying that
\begin{align*}
\E\left[\left| \frac{1}{n}\sum_{v=1}^n \tilde{W}_l^{v,n}-\frac{1}{n}\sum_{v=1}^nW_l^{v,n}\right|\right]
&\le
2^6(1+\gamma)\E\left[\left(\frac{1}{n}\sum_{v=1}^n\left|C_l'^{A,v,n}-\tilde{C}_l'^{A,v,n}\right|^2+\left|C_l'^{O,v,n}-\tilde{C}_l'^{O,v,n}\right|^2\right)^\frac12\right.
\\
&\left.\times
\left(\frac{1}{n}\sum_{v=1}^n\left|\tilde{C}_l'^{A,v,n}\right|^2+\left|C_l'^{A,v,n}\right|^2+\left|\tilde{C}_l'^{O,v,n}\right|^2+\left|C_l'^{O,v,n}\right|^2\right)^\frac12 \right]
\\
&\le
2^6(1+\gamma)\E\left[\frac{1}{n}\sum_{v=1}^n\left|C_l'^{A,v,n}-\tilde{C}_l'^{A,v,n}\right|^2+\left|C_l'^{O,v,n}-\tilde{C}_l'^{O,v,n}\right|^2\right]^\frac12
\\
&\times
\E\left[\frac{1}{n}\sum_{v=1}^n\left|\tilde{C}_l'^{A,v,n}\right|^2+\left|C_l'^{A,v,n}\right|^2+\left|\tilde{C}_l'^{O,v,n}\right|^2+\left|C_l'^{O,v,n}\right|^2 \right]^\frac12
\\
&=
2^6(1+\gamma)\E\left[\left|C_l'^{A,1,n}-\tilde{C}_l'^{A,1,n}\right|^2+\left|C_l'^{O,1,n}-\tilde{C}_l'^{O,1,n}\right|^2\right]^\frac12
\\
&\times
\E\left[\left|\tilde{C}_l'^{A,1,n}\right|^2+\left|C_l'^{A,1,n}\right|^2+\left|\tilde{C}_l'^{O,1,n}\right|^2+\left|C_l'^{O,1,n}\right|^2 \right]^\frac12.
\end{align*}
Similarly to \eqref{tildediff} and due to the independence of the eigenfunctions estimators,
\begin{align}\label{tildediff1}
\E\left[\left|C_l'^{*,m,n}-\tilde{C}_l'^{*,m,n}\right|^2\right]
\le
2\left(Td_n^2+ \E\left[\n X^{*}\n_{L^2([T_1,T_2])}^2\right]\right)\E\left[\n T_M\left(\hat{\psi}_{l,n}\right)-\psi_l\n_{L^2([T_1,T_2])}^2\right].
\end{align}
Therefore
\begin{align*}
\E\left[\left| \frac{1}{n}\sum_{v=1}^n \tilde{W}_l^{v,n}-\frac{1}{n}\sum_{v=1}^nW_l^{v,n}\right|\right]
&\le 4\left(Td_n^2+ \E\left[\n X^{A}\n_{L^2([T_1,T_2])}^2+\n X^{O}\n_{L^2([T_1,T_2])}^2\right]\right)\E\left[\n T_M\left(\hat{\psi}_{l,n}\right)-\psi_l\n_{L^2([T_1,T_2])}^2\right]
\\
&\times
\E\left[\left|\tilde{C}_l'^{A,1,n}\right|^2+\left|C_l'^{A,1,n}\right|^2+\left|\tilde{C}_l'^{O,1,n}\right|^2+\left|C_l'^{O,1,n}\right|^2 \right]^\frac12,
\end{align*}
which converges to zero for every $l\in\N$. By the triangle inequality it the follows that 
\begin{align*}
\E\left[\left| \frac{1}{n}\sum_{v=1}^n \tilde{W}_l^{v,n}-\frac{1}{n}\sum_{v=1}^nW_l^{v}\right|\right]
\end{align*}
converges to zero as well.
Let
$$A_{n,2}=\left\{ \frac{\E\left[Q_l\right]^2}{2} \le \left(\frac{1}{n}\sum_{v=1}^n \tilde{W}_l^{v,n}\right)^2\le 2\E\left[Q_l\right]^2, l=1,..,e(n)\right\}\cap A_{n,1}. $$
We may now define $E_7(n)$ such that if $e(n)\le E_7(n)$ then $\lim_{n\to\infty}\P\left(A_{n,2}^c\right)=0$.
We have the following bound
\begin{align}\label{B2B2}
&\E\left[\n B^{(m),7}-B^{(m),6} \n_{L^2([0,T]^2)}1_{A_{n,2}}\right]\nonumber
\\
&\le
(1+\gamma)\E\left[\sum_{k=1}^{e(n)} \sum_{l=1}^{e(n)} \left(\left|\tilde{C}_l^{A,m,n}\right|+\left|\tilde{C}_l^{O,m,n}\right|\right)\left|\frac{1}{\frac{1}{n}\sum_{p=1}^n \tilde{W}_l^{m,n}}-\frac{1}{\frac{1}{n}\sum_{p=1}^n W_l^{m,n}} \right|\right.
\nonumber
\\
&\left.\times \left(\left|D_k^{A,m,n} \right|+\left|D_k^{O,m,n} \right| \right)\left(\left|\tilde{C}_l^{A,m,n}\right|+\left|\tilde{C}_l^{O,m,n}\right|\right)\n   \phi_k\otimes T_M\left(\hat{\psi}_{l,n}\right)\n_{L^2([0,T]^2)}1_{A_{n,2}}\right]\nonumber
\\
&\le
(1+\gamma)M\E\left[\sum_{k=1}^{e(n)} \sum_{l=1}^{e(n)}\left(\left|D_k^{A,m,n} \right|+\left|D_k^{O,m,n} \right| \right)\left(\left|\tilde{C}_l^{A,m,n}\right|+\left|\tilde{C}_l^{O,m,n}\right|\right)\n \phi_k\n_{L^2([0,T]^2)} \right]\nonumber
\\
&\times\E\left[\frac{\left| \frac{1}{n}\sum_{p=1}^n \tilde{W}_l^{m,n}-\frac{1}{n}\sum_{p=1}^nW_l^{m,n}\right|}{\frac{1}{n}\sum_{p=1}^n W_l^{m,n} \frac{1}{n}\sum_{p=1}^n \tilde{W}_l^{m,n}}  M1_{A_{n,2}}\right]\nonumber
\\
&\le
2(1+\gamma)M4\left(Td_n^2+ \E\left[\n X^{A}\n_{L^2([T_1,T_2])}^2+\n X^{O}\n_{L^2([T_1,T_2])}^2\right]\right)\sum_{k=1}^{e(n)} \E\left[\left|D_k^{A,m,n} \right|^2+\left|D_k^{O,m,n} \right|^2 \right]^{\frac12}\nonumber
\\
&\times\sum_{l=1}^{e(n)}\E\left[\left|\tilde{C}_l^{A,m,n}\right|^2+\left|\tilde{C}_l^{O,m,n}\right|^2\right]^{\frac12}\frac{\E\left[\n T_M\left(\hat{\psi}_{l,n}\right)-\psi_l\n_{L^2([T_1,T_2])}^2\right]}{\frac 12\E\left[Q_l\right]^2 } .
\end{align}
We can therefore define $E_8(n)$ such that if $e(n)\le E_8(n)$ then $\E\left[\n B^{(m),7}-B^{(m),6} \n_{L^2([0,T]^2)}1_{A_{n,2}}\right]\to 0$.
To summarize, using the triangle inequality together with steps 1-7 yields the final result.
	\end{proof}
	\subsection{Proof of Theorem \ref{Minimizerpg1}}
	\begin{proof}
	\textbf{Step 1: Use the entire sample curves}
	\\
By assumption there exists $N'\in\N$ such that if $n\ge N'$ then $G_n$ is full rank. 
	Define
$$\hat{G}_{n,1}(M)=\sqrt{\gamma}\frac{1}{M}\sum_{m=1}^MF_{1:n}\left(X^{A,m}\right)^TF_{1:n}\left(X^{A,m}\right)
+
(1-\sqrt{\gamma})\frac{1}{M}\sum_{m=1}^MF_{1:n}\left(X^{O,m}\right)^TF_{1:n}\left(X^{O,m}\right) $$
and
\begin{align*}
\left(\hat{\lambda}_{1,k,1}^{(1)}(n,M),\ldots,\hat{\lambda}_{1,k,n}^{(1)}(n,M),\ldots,\hat{\lambda}_{p,k,n}^{(1)}(n,M) \right)
&=
\hat{G}_{n,1}(M)^{-1}\left( \gamma \frac{1}{M}\sum_{m=1}^M Z_k^{A,m}F_{1:n}\left(X^{A,m}\right)
\right.
\\
&\left.+
(1-\gamma)\frac{1}{M}\sum_{m=1}^MZ_k^{O,m}F_{1:n}\left(X^{O,m}\right)
\right)1_{\det\left(\hat{G}_{n,1}(M) \right)\not=0}
,
\end{align*}
for $1\le k\le n$ and
\begin{align*}
\hat{\beta}_{n,1,M}=\left( \sum_{k=1}^n\sum_{l=1}^n\hat{\lambda}_{1,k,1}^{(1)}(n,M)\phi_k\otimes \phi_l,\ldots, \sum_{k=1}^n\sum_{l=1}^n\hat{\lambda}_{p,k,n}^{(1)}(n,M)\phi_k\otimes \phi_l\right).
\end{align*}
For notational convenience we also let
$$\lambda(n)= \left(\lambda_{1,1,1}(n),\ldots,\lambda_{p,n,n}(n) \right).$$
and
$$\hat{\lambda}(n,M)= \left(\hat{\lambda}_{1,1,1}(n,M),\ldots,\hat{\lambda}_{p,n,n}(n) \right).$$
By the law of large numbers and the continuity of the determinant (in terms of its entries) it follows that for sufficiently large $M$, $\det\left(\hat{G}_{n,1}(M) \right)\not=0$, for $n\ge N'$. By continuity and the law of large numbers it then follows that $\lim_{M\to\infty}\hat{\lambda}(n,M)=\lambda(n)$, a.s. for every $n\in\N$. From this we may then readily conclude that $\lim_{M\to\infty}  \lVert  \hat{\beta}_{n,1,M} - \beta_n\rVert_{L^2([T_1,T_2]^2)^p}=0$, a.s. for every $n\ge N$.
\\
\textbf{Step 2: Discretize the "numerators".}
\begin{itemize}
		\item[] $C_l^{A,m,n}(i)=\int P_{\Pi_n}(X^{A,m}(i),t)\phi_l(t)dt,$
		\item[] $C_l^{O,m,n}(i)=\int P_{\Pi_n}(X^{O,m}(i),t)\phi_l(t)dt,$
		\item[] $C_l'^{A,m,n}(i)=\int P_{\Pi_n}(X'^{A,m}(i),t)\phi_l(t)dt,$
		\item[] $C_l'^{O,m,n}(i)=\int P_{\Pi_n}(X'^{O,m}(i),t)\phi_l(t)dt,$
		\item[] $D_k^{A,m,n}=\int P_{\Pi_n}(Y^{A,m},t)\phi_k(t)dt$ and
		\item[] $D_k^{O,m,n}=\int P_{\Pi_n}(Y^{O,m},t)\phi_k(t)dt$.
\end{itemize}
Let
\begin{align*}
\hat{G}_{n,2}(M)
&=\sqrt{\gamma}\frac{1}{M}\sum_{m=1}^M\left(C_1^{A,m,M}(1),\ldots,C_n^{A,m,M}(p)\right)^T\left(C_1^{A,m,M}(1),\ldots,C_n^{A,m,M}(p)\right)
\\
&+(1-\sqrt{\gamma})\frac{1}{M}\sum_{m=1}^M\left(C_1^{O,m,M}(1),\ldots,C_n^{O,m,M}(p)\right)^T\left(C_1^{O,m,M}(1),\ldots,C_n^{O,m,M}(p)\right), 
\end{align*}
\begin{align*}
\left(\hat{\lambda}_{1,k,1}^{(2)}(n,M),\ldots,\hat{\lambda}_{1,k,n}^{(2)}(n,M),\ldots,\hat{\lambda}_{p,k,n}^{(2)}(n,M) \right)
&=
\hat{G}_{n,2}(M)^{-1}\left( \gamma \frac{1}{M}\sum_{m=1}^M D_k^{A,m,M}\left(C_1^{A,m,M}(1),\ldots,C_n^{A,m,M}(p)\right)
+
\right.
\\
&\left.
(1-\gamma)\frac{1}{M}\sum_{m=1}^MD_k^{O,m,M}\left(C_1^{O,m,M}(1),\ldots,C_n^{O,m,M}(p)\right)
\right)1_{\det\left(\hat{G}_{n,2}(M) \right)\not=0}
,
\end{align*}
for $1\le k\le n$ and
\begin{align*}
\hat{\beta}_{n,2,M}=\left( \sum_{k=1}^n\sum_{l=1}^n\hat{\lambda}_{1,k,1}^{(2)}(n,M)\phi_k\otimes \phi_l,\ldots, \sum_{k=1}^n\sum_{l=1}^n\hat{\lambda}_{p,k,n}^{(2)}(n,M)\phi_k\otimes \phi_l\right).
\end{align*}
 Due to the definitions of $C_l^{A,m,n}(i), C_l^{O,m,n}(i), D_k^{A,m,n}$ and $D_k^{O,m,n}$, for every $n\in\N$ we get that $\hat{G}_{n,2}(M) - \hat{G}_{n,1}(M)$ converges to zero as $M\to\infty$, for $n\ge N$. In the final step of the proof we choose $\{e(n)\}_{n\in\N}\subseteq\N$ such that $e(n)\to\infty$, $\lim_{n\to\infty}\lVert  \hat{\beta}_{e(n),1,n} - \beta_{e(n)}\rVert_{L^2([T_1,T_2]^2)^p}=0$ and $\lim_{n\to\infty}\lVert  \hat{\beta}_{e(n),2,n} - \hat{\beta}_{e(n),1,n}\rVert_{L^2([T_1,T_2]^2)^p}=0$. Since $\mathsf{dist}\left(\beta_n,S\right)=\inf_{s\in S}\lVert \beta_n -s\rVert_{L^2([T_1,T_2]^2)^p}\to 0$ it follows that $\mathsf{dist}\left(\hat{\beta}_{e(n),2,n} ,S\right)\to 0$.  
\end{proof}

\end{document}